\documentclass[10pt,a4paper,reqno]{amsart}
\usepackage{graphicx} % Required for inserting images

\title[Extended modules of Nakayama algebras]{A note on the $m$-extended module categories of Nakayama algebras}
\author{Endre Sørmo Rundsveen}
\address{Department of mathematical sciences, NTNU, Trondheim, Norway}
\email{endre.s.rundsveen@ntnu.no}

\usepackage[foot]{amsaddr}

\input{setup.sty}

\usepackage[
    backend=biber,
    style=alphabetic,
    sortlocale=en_US,
    url=false, 
    doi=true,
    eprint=true,
	maxbibnames=99,
	maxalphanames=4,
	minalphanames=3,
	defernumbers=true,
]{biblatex}
\begin{document}

\begin{abstract}
    We study the extended module category $\mmod\Lambda$ of an algebra $\Lambda$, recently introduced by Gupta and Zhou. 
    The existence of a postprojective component of $\mmod\Lambda$ is shown for certain algebras $\Lambda$, and a knitting algorithm through cohomological dimension vectors is provided.
    
    In particular, the extended module category of Nakayama algebras with homogeneous relations are investigated, and it is classified when they are of finite type.
    The paper give fully calculated AR-quivers for several Nakayama algebras.
\end{abstract}

\maketitle
\tableofcontents

\section{Introduction}

\subsubsection*{Background and motivation}
In the last decade, $\tau$-tilting \cite{AIR14} has been adopted as a fruitful tool by representation theorists. 
It corresponds nicely with structures which are otherwise of great interest, namely functorially finite torsion classes, ($2$-term) silting complexes and cluster-tilting objects. 
In addition, it fills a gap which had been left vacant by tilting theory. That is, ensuring that an almost complete $\tau$-tilting module has exactly two complements.

A natural question to ponder in relation to $\tau$-tilting is if there are similar construction which correspond to longer silting objects. 
Gupta provided a partial answer in \cite{Gup24}, by showing that $(m+1)$-silting objects correspond to functorially finite positive torsion classes in the category of $m$-extended modules,
$\mmod\Lambda=\D^{[-(m-1),0]}(\mod\Lambda)$. 
In a following paper, Zhou \cite{Zho25} elaborated on this by showing that $\mmod\Lambda$ admits Auslander-Reiten triangles and that $\tau_{[m]}$-tilting pairs in $\mmod\Lambda$ correspond to $(m+1)$-silting pairs.
In a recent paper, Gupta and Zhou \cite{GZ25} studies $\mmod\Lambda$ further, and add left-finite semibricks and left-finite wide subcategories of $\mmod\Lambda$ to the correspondence, as well as $(m+1)$-term simple minded collections of $\D^b(\mod\Lambda)$.

It seems that $\mmod\Lambda$ might be the proper setting for a generalization of $\tau$-tilting. 
Moreover, it is close enough to the module category to be reasonably well-behaved, 
while still being a big enough part of the derived category to provide insights to its structure.

Another avenue for generalization of $\tau$-tilting can be found in the world of $m$-cluster tilting subcategories $\mathcal{M}$ of $\mod\Lambda$ \cite{Iya08}. 
Several related proposals have been made \cite{JJ20,McM21,ZZ23,MM23}, with the one of \cite{MM23} also being given through $(m+1)$-silting objects of $\homotopy^b(\proj\Lambda)$.
In a paper currently in preparation, the authors of \cite{AHJKPT25+} show that functorially finite $m$-torsion classes of $\mathcal{M}$ give rise to maximal $\tau_m$-rigid pairs in $\mathcal{M}$ with $|\Lambda|$ summands, as well as an injection into the set of $(m+1)$-term silting objects.
This inspired a paper by the author and Laertis Vaso \cite{RV25} in which $\tau_m$-rigid pairs with $|\Lambda|$ summands are investigated for Nakayama algebras.

The study in \cite{RV25} along with the competing generalization of $\tau$-tilting is the original motivation for looking at $\mmod\Lambda$ when $\Lambda$ is a Nakayama algebra. 
Another motivation for choosing Nakayama algebras in particular is that their module categories are reasonably well-behaved, and thus might also provide reasonably well-behaved extended module categories.
This paper thus seeks to give a better understanding of the extended module category of Nakayama algebras.

\subsubsection*{Our results}
We show how cohomological dimension vectors (see \Cref{def:cohomologicalDimension}) can replace the ordinary dimension vectors when knitting AR-quivers of the extended modules. 
We provide an explicit algorithm for computing the postprojective component of $\mmod\Lambda$ starting in a simple projective, when indecomposable projectives have indecomposable radicals, see \Cref{prop:AlgorithmForKnittingDimensionVectors}.

Let $\Lambda(n,l)$ be the algebra given by the quiver
\begin{equation}\label{eq:LinearNakayamaQuiver}\tag{$\bA_{n}$}
\begin{tikzpicture}[xscale=.9,tips=proper,baseline=(current  bounding  box.center)]
    \node (d) at (0,0) [nodeDots] {};
    \node (d-label) at (0,-.1) [below,scale=.7] {$n-1$};
    \node (d-1) at (2,0) [nodeDots] {};
    \node (d-1-label) at (2,-.1) [below,scale=.7] {$n-2$};
    \node (cdot1) at (3.8,0) [nodeCDots] {};
    \node (cdot2) at (4,0) [nodeCDots] {};
    \node (cdot3) at (4.2,0) [nodeCDots] {};
    \node (2) at (6,0) [nodeDots] {};
    \node (2-label) at (6,-.1) [below,scale=.7] {$2$};
    \node (1) at (8,0) [nodeDots] {};
    \node (1-label) at (8,-.1) [below,scale=.7] {$1$};
    \node (0) at (10,0) [nodeDots] {};
    \node (0-label) at (10,-.1) [below,scale=.7] {$0$};

    \draw[-latex] ([xshift=.2cm]d.center)--([xshift=-.2cm]d-1.center)node[midway,above,scale=.7]{$\alpha_{n-1}$};
    \draw[-latex] ([xshift=.2cm]d-1.center)--([xshift=-.2cm]cdot1.center)node[midway,above,scale=.7]{$\alpha_{n-2}$};
    \draw[-latex] ([xshift=.2cm]cdot3.center)--([xshift=-.2cm]2.center)node[midway,above,scale=.7]{$\alpha_3$};
    \draw[-latex] ([xshift=.2cm]2.center)--([xshift=-.2cm]1.center)node[midway,above,scale=.7]{$\alpha_2$};
    \draw[-latex] ([xshift=.2cm]1.center)--([xshift=-.2cm]0.center)node[midway,above,scale=.7]{$\alpha_1$};
\end{tikzpicture}
\end{equation}
with relations all paths of length $l$. 
This algebra has a simple projective, and we can therefore easily knit the shape of $\AR(\mmod\Lambda(n,l))$ with cohomological dimension vectors.
Having done this for a collection of examples, we postulate that \Cref{tab:FiniteNakayama} classifies when $\mmod\Lambda(n,l)$ is of finite type, i.e. has only finitely many indecomposables up to isomorphism. This is in fact our first main result.
\begin{theorem}[{\Cref{thm:WhenIsLinearNakayamaFinite}}]
    $\mmod\Lambda(n,l)$ is of finite type if and only if it is contained in \Cref{tab:FiniteNakayama}.
    \begin{table}[h]
    \centering
    \begin{tabularx}{.9\linewidth}{c?YYYYYYYYYY}
        \diagbox{$l$}{$n$}%$l\backslash n$
         & $3$ & $4$ & $5$ & $6$ & $7$ & $8$ & $9$ & $10$ & $11$ &$\cdots$ \\ \thickhline
        $2$ & {\normalfont All} & {\normalfont All} & {\normalfont All} & {\normalfont All} & {\normalfont All} & {\normalfont All} & {\normalfont All} & {\normalfont All} & {\normalfont All} & $\cdots$ \\ 
        $3$ &  & {\normalfont All} & {\normalfont All} & {\normalfont All} & {\normalfont All} & {\normalfont All} & $\leq 4$ & $\leq 4$ & $\leq 4$ & $\cdots$ \\ 
        $4$ &  &  & {\normalfont All} & {\normalfont All} & {\normalfont All} & $\leq 2$ & $\leq2$ & $\leq 2$ & $\leq 2$ & $\cdots$ \\ 
        $5$ &  &  &  & {\normalfont All} & {\normalfont All} & {\normalfont All} & $\leq 2$ & $\leq 2 $ & $\leq 2$ & $\cdots$ \\ 
        $6$ &  &  &  &  & {\normalfont All} & {\normalfont All} & $=1$ & $=1$ & $=1$ & $\cdots$ \\ 
        $7$ &  &  &  &  &  & {\normalfont All} & $=1$ & $=1 $ & $=1$ & $\cdots$ \\ 
        $\vdots$ &  &  &  &  &  & &$\ddots$ &$\ddots$ &$\ddots$ &$\ddots$
    \end{tabularx}
    \caption{For which $m$ is $\mmod\Lambda(n,l)$ of finite type. }
    \label{tab:FiniteNakayama}
\end{table}
\end{theorem}
We also observe that when $\mmod\Lambda(n,l)$ is of finite type, the cohomological dimension vector of $X^\bullet\in\mmod\Lambda$ is non-zero only in small intervals of $[0,n-1]$, see \Cref{lemma:boundOnCohomology}.

In the last part of the paper, we move from the linear to the cyclic case. 
Let $\Delta(n,l)$ be the algebra given by
\begin{equation}\label{eq:CyclicNakayamaQuiver}\tag{$\Delta_{n}$}
\begin{tikzpicture}[scale=.8,baseline=(current  bounding  box.center)]
    \node (n-1) at (90:1cm) [nodeDots] {};
    \node (n-1-label) at (90:1.1cm) [scale=.7,above]{$n-1$};
    
    \node (n-2) at (30:1cm) [nodeDots] {};
    \node (n-2-label) at (30:1.1cm) [scale=.7,right] {$n-2$};

    \node (cdot1) at ($(-30:1)+(0,.08)$) [nodeCDots,scale=.7] {};
    \node (cdot2) at ($(-30:1)+(-.04,0)$) [nodeCDots,scale=.7] {};
    \node (cdot3) at ($(-30:1)+(-.08,-.08)$) [nodeCDots,scale=.7] {};

    \node (2) at (-90:1cm) [nodeDots] {};
    \node (2-label) at (-90:1.1cm) [scale=.7,below] {$2$};

    \node (1) at (-150:1cm) [nodeDots] {};
    \node (1-label) at (-150:1.1cm) [scale=.7,left] {$1$};

    \node (0) at (150:1cm) [nodeDots] {};
    \node (0-label) at (150:1.1cm) [scale=.7,left] {$0$};

    \draw[-latex] ([xshift=2.5pt,yshift=-1.5pt]n-1.center)--([xshift=-2.5pt,yshift=1pt]n-2.center)node[midway,anchor=south west,scale=.7]{$\delta_{n-1}$};
    \draw[-latex] ([yshift=-1.5pt]n-2.center)--([yshift=1.5pt]cdot1.center)node[midway,right,scale=.7]{$\delta_{n-2}$};
    \draw[-latex] ([xshift=-2.5pt,yshift=-1.5pt]cdot3.center)--([xshift=2.5pt,yshift=1pt]2.center)node[midway,anchor=north west,scale=.7]{$\delta_{3}$};
    \draw[-latex] ([xshift=-2.5pt,yshift=1.5pt]2.center)--([xshift=2.5pt,yshift=-1pt]1.center)node[midway,anchor=north east,scale=.7]{$\delta_{2}$};
    \draw[-latex] ([yshift=3pt]1.center)--([yshift=-2.5pt]0.center)node[midway,anchor=east,scale=.7]{$\delta_{1}$};
    \draw[-latex] ([xshift=2.5pt,yshift=1.5pt]0.center)--([xshift=-2.5pt,yshift=-1pt]n-1.center)node[midway,anchor=south east,scale=.7]{$\delta_{0}$};
\end{tikzpicture}
\end{equation}
with relations all paths of length $l$. 
This algebra does not have a simple projective, nor a simple injective, hence the AR-quivers are not as easily obtained.
We therefore utilize the covering theory of Riedtmann \cite{Rie80} and Bongartz--Gabriel \cite{BG81} to transfer our knowledge from the linear case.
This leads to our final result.
\begin{theorem}[{\Cref{theorem:finitenessOfCyclicNakayama}}]\label{introThm}
    $\mmod\Delta(n,l)$ is of finite type if and only if
    \begin{enumerate}
        \item $l=2$,
        \item $m\leq4$ and $l=3$,
        \item $m\leq 2$ and either $l=4$ or $l=5$, or
        \item $m=1$
    \end{enumerate}
\end{theorem}

\subsubsection*{Structure}
The paper consists of five sections. Section 2 is spent recalling relevant notions and theory about derived categories and quiver representations.
In Section 3 we introduce the extended module categories and assemble tools which help us construct its AR-quivers. 
Section 4 is devoted to prove the validity of \Cref{tab:FiniteNakayama}. 
At last, in section 5 we recall relevant theory about coverings before using it to prove \Cref{introThm}. We end Section 5 with a few examples.
Throughout, we showcase the form of several fully calculated AR-quivers.

\subsubsection*{Notation and conventions}
We fix a field $\K$.
All algebras are assumed to be non-semi simple, connected and basic $\K$-algebras.
We denote by $\mod\Lambda$ the category of left finitely generated modules over $\Lambda$. 
Given a complete set of orthogonal idempotents $\{e_0,\ldots,e_{n-1}\}$ of $\Lambda$, we denote by $S_i$, $P_i=\Lambda e_i$ and $I_i=D(e_i\Lambda)$ the simples, and the indecomposable projectives and injectives respectively. 
$D\coloneqq\Hom_\K(-,\K)$ is the standard duality. Morphisms are composed from right to left: for $f\colon A\to B$ and $g\colon B\to C$ we have $g\circ f\colon A\to C$. 
All subcategories are asummed to be full, additive and closed under isomorphisms, and all functors to be additive.

\section*{Acknowledgements}
The author is grateful for the kind and continued support of Steffen Oppermann throughout the work on this project.
They also thank Jacob Fjeld Grevstad for giving valuable input and being available for several discussions. 
Moreover, the author express gratitude to the QPA-team \cite{qpa} for having made a tool which significantly reduced the cost of preliminary exploration.

\section{Preliminaries}

\subsection{Derived categories}
This section only give a brief glimpse into the world of derived categories. We refer to the books \cite{Wei94} and \cite{Kra21} and lecture notes of Dragan Milicic \cite{milicic2014lectures} for more details and background.
Let $\cA$ be an additive category.
A complex in $\cA$ is a pair $X^\bullet=(X^i,d_{X^\bullet}^i)_{i\in\bZ}$, with $X^i\in\cA$ and morphisms $d_{X^\bullet}^i\colon X^{i}\to X^{i+1}$ such that $d_{X^\bullet}^{i}\circ d_{X^\bullet}^{i-1}=0$ for all $i\in\bZ$. 
We can also view $X^\bullet$ as a diagram
\[
\begin{tikzpicture}
    \node (dotsLeft1) at (0,0) [nodeCDots]{};
    \node[right=.1cm of dotsLeft1,nodeCDots] (dotsLeft2) {};
    \node[right=.1cm of dotsLeft2,nodeCDots] (dotsLeft3) {};
    \node[right=1cm of dotsLeft3] (A) {$X^{-1}$};
    \node[right=1cm of A] (B) {$X^0$};
    \node[right=1cm of B] (C) {$X^1$};
    \node[right=1cm of C,nodeCDots] (dotsRight1){};
    \node[right=.1cm of dotsRight1,nodeCDots] (dotsRight2){};
    \node[right=.1cm of dotsRight2,nodeCDots] (dotsRight3){};
    \draw[->] ([xshift=4pt]dotsLeft3.center)--(A)node[midway,above,scale=.7]{$d_{X^\bullet}^{-2}$};
    \draw[->] (A)--(B)node[midway,above,scale=.7]{$d_{X^\bullet}^{-1}$};
    \draw[->] (B)--(C)node[midway,above,scale=.7]{$d_{X^\bullet}^{0}$};
    \draw[->] (C)--([xshift=-4pt]dotsRight1.center)node[midway,above,scale=.7]{$d_{X^\bullet}^{1}$};
\end{tikzpicture}
\]
The category of complexes over $\cA$ is denoted by $\complexC(\cA)$.
There is an embedding of $\cA$ in $\complexC(\cA)$, where an object is a complex concentrated in the zeroth degree; we call this a \emph{stalk complex} of the object. 

Complexes $X^\bullet$ whose terms are zero for all indices $i\gg 0$ are called \emph{right bounded}, complexes whose terms are zero for all indices $i\ll 0$ are called \emph{left bounded}, and a complex which is both left and right bounded is simply called \emph{bounded}. 
The full subcategories of $\complexC(\cA)$ consisting of right bounded (left bounded, bounded) complexes are denoted by $\complexC^-(\cA)$ ($\complexC^+(\cA)$,$\complexC^b(\cA)$). 

The \emph{homotopy} category $\homotopy(\cA)$ is the triangulated category whose objects are complexes and morphisms are the equivalence classes of morphisms in $\complexC(\cA)$ under the homotopy relation. 
We denote by $\homotopy^-(\cA)$, $\homotopy^+(\cA)$ and $\homotopy^b(\cA)$ the induced full subcategories. When $\cA$ is abelian we denote by $\D^-(\cA)$, $\D^+(\cA)$ and $\D^b(\cA)$ the corresponding derived categories.

\begin{remark}
    For complexes which are non-zero only for small intervals in $\bZ$, we may find fit to think of them as finite sequences. In particular, when represented as a finite diagram, the right-most term will be assumed to be the zeroth term. That is, $X^\bullet =[A\to B\to C]$ is a complex such that $X^0=C$, $X^{-1}=B$, $X^{-2}=A$ and $X^i=0$ for $i>0$ and $i<-2$. Further, $Y^\bullet=\nshift{[A\to B]}{2}$, is a complex such that $Y^{-3}=A$, $Y^{-2}=B$ and zero elsewhere.
\end{remark}

\subsubsection{\texorpdfstring{$t$}{t}-structures and truncations} Two full subcategories of $\D^b(\cA)$ which will be playing a crucial part in the construction of the extended module categories are 
\[
\D^{\leq 0}(\cA)\coloneqq \{X^\bullet \in \D^b(\cA)\ |\ H^i(X^\bullet)=0\text{ for }i>0\}
\]
and
\[
\D^{\geq 0}(\cA)\coloneqq \{X^\bullet \in \D^b(\cA)\ |\ H^i(X^\bullet)=0\text{ for }i<0\}
\]
along with their shifts, $\D^{\leq p}(\cA)=\nshift{\D^{\leq 0}(\cA)}{-p}$, $\D^{\geq p}(\cA)=\nshift{\D^{\geq 0}(\cA)}{-p}$. The pair $(\D^{\leq 0}(\cA),\D^{\geq 0}(\cA))$ is the \emph{canonical $t$-structure} on $\D^b(\cA)$. 
For an introduction to $t$-structures in general, the reader can for example see \cite{KS94} or the original paper \cite{BBD82}. Most important for us is the existence of triangles
\begin{equation}\label{eq:softTruncation}
\begin{tikzpicture}[baseline]
    \node (A) at (0,0) [] {$\softTrunc^{\leq p}(X^\bullet)\vphantom{[1]}$};
    \node[right=1.5cm of A] (B) {$X^\bullet\vphantom{[1]^{\geq 1}}$};
    \node[right=1.5cm of B] (C) {$\softTrunc^{\geq p+1}(X^\bullet)\vphantom{[1]}$};
    \node[right=1.5cm of C] (Ashift){$\shift{\softTrunc^{\leq p}(X^\bullet)}$,};
    \draw[->] ([yshift=1pt]A.east)--([yshift=1pt]B.west);
    \draw[->] ([yshift=1pt]B.east)--([yshift=1pt]C.west);
    \draw[->] ([yshift=1pt]C.east)--([yshift=1pt]Ashift.west);
\end{tikzpicture}
\end{equation}
for each complex $X^\bullet\in\D^b(\cA)$, where $\softTrunc^{\leq p}$ ($\softTrunc^{\geq p}$) is the right (left) adjoint of the inclusion $\D^{\leq p}(\cA)\to\D^b(\cA)$ ($\D^{\geq p}(\cA)\to\D^b(\cA)$). We call $\softTrunc^{\leq p}$ and $\softTrunc^{\geq p}$ \emph{soft truncations}, and they are given by
\[
\begin{tikzpicture}
    \node (X) at (-1.3,0) [anchor=east] {$X^\bullet\colon$};
    \node (XdotsLeft1) at (-.2,0) [nodeCDots]{};
    \node (XdotsLeft2) at (0,0) [nodeCDots]{};
    \node (XdotsLeft3) at (.2,0) [nodeCDots]{};
    \node (Xp-1) at (2,0) [] {$X^{p-1}$};
    \node (Xp) at (4,0) [] {$X^p$};
    \node (Xp+1) at (6,0) [] {$X^{p+1}$};
    \node (XdotsRight1) at (7.8,0) [nodeCDots]{};
    \node (XdotsRight2) at (8,0) [nodeCDots]{};
    \node (XdotsRight3) at (8.2,0) [nodeCDots]{};

    \draw[->] ([xshift=4pt]XdotsLeft3.center)--(Xp-1)node[midway,above,scale=.7]{$d_{X^\bullet}^{p-2}$};
    \draw[->] (Xp-1)--(Xp)node[midway,above,scale=.7]{$d_{X^\bullet}^{p-1}$};
    \draw[->] (Xp)--(Xp+1)node[midway,above,scale=.7]{$d_{X^\bullet}^{p}$};
    \draw[->] (Xp+1)--([xshift=-4pt]XdotsRight1.center)node[midway,above,scale=.7]{$d_{X^\bullet}^{p+1}$};

    \node (X) at (-1.3,-1) [anchor=east] {$\softTrunc^{\leq p}X^\bullet\colon$};
    \node (XdotsLeft1) at (-.2,-1) [nodeCDots]{};
    \node (XdotsLeft2) at (0,-1) [nodeCDots]{};
    \node (XdotsLeft3) at (.2,-1) [nodeCDots]{};
    \node (Xp-1) at (2,-1) [] {$X^{p-1}$};
    \node (Xp) at (4,-1) [] {$\ker d_{X^\bullet}^p$};
    \node (Xp+1) at (6,-1) [] {$0$};
    \node (XdotsRight1) at (7.8,-1) [nodeCDots]{};
    \node (XdotsRight2) at (8,-1) [nodeCDots]{};
    \node (XdotsRight3) at (8.2,-1) [nodeCDots]{};

    \draw[->] ([xshift=4pt]XdotsLeft3.center)--(Xp-1)node[midway,above,scale=.7]{$d_{X^\bullet}^{p-2}$};
    \draw[->] (Xp-1)--(Xp)node[midway,above,scale=.7]{};
    \draw[->] (Xp)--(Xp+1)node[midway,above,scale=.7]{};
    \draw[->] (Xp+1)--([xshift=-4pt]XdotsRight1.center)node[midway,above,scale=.7]{};

    \node (X) at (-1.3,-2) [anchor=east] {$\softTrunc^{\geq p}X^\bullet\colon$};
    \node (XdotsLeft1) at (-.2,-2) [nodeCDots]{};
    \node (XdotsLeft2) at (0,-2) [nodeCDots]{};
    \node (XdotsLeft3) at (.2,-2) [nodeCDots]{};
    \node (Xp-1) at (2,-2) [] {$0$};
    \node (Xp) at (4,-2) [] {$\coker d_{X^\bullet}^{p-1}$};
    \node (Xp+1) at (6,-2) [] {$X^{p+1}$};
    \node (XdotsRight1) at (7.8,-2) [nodeCDots]{};
    \node (XdotsRight2) at (8,-2) [nodeCDots]{};
    \node (XdotsRight3) at (8.2,-2) [nodeCDots]{};

    \draw[->] ([xshift=4pt]XdotsLeft3.center)--(Xp-1)node[midway,above,scale=.7]{};
    \draw[->] (Xp-1)--(Xp)node[midway,above,scale=.7]{};
    \draw[->] (Xp)--(Xp+1)node[midway,above,scale=.7]{};
    \draw[->] (Xp+1)--([xshift=-4pt]XdotsRight1.center)node[midway,above,scale=.7]{$d_{X^\bullet}^{p+1}$};

\end{tikzpicture}
\]
We also note that $\softTrunc^{\geq p}\circ\softTrunc^{\leq q}\cong \softTrunc^{\leq q}\circ \softTrunc^{\geq p}$ and $H^n(X^\bullet)\cong \softTrunc^{\geq n}\softTrunc^{\leq n}(X^\bullet)$.

We will also be using \emph{brutal truncations}, which are maps $\brutalTrunc_{\leq p}\colon \D^b(\cA)\to D^{\leq p}(\cA)$ and $\brutalTrunc_{\geq p}\colon \D^b(\cA)\to D^{\geq p}(\cA)$ given by
\[
\begin{tikzpicture}
    \node (X) at (-1.3,-1.05) [anchor=east] {$\brutalTrunc_{\leq p}X^\bullet\colon$};
    \node (XdotsLeft1) at ($(-.2,-1)-(0,1pt)$) [nodeCDots]{};
    \node (XdotsLeft2) at ($(0,-1)-(0,1pt)$) [nodeCDots]{};
    \node (XdotsLeft3) at ($(.2,-1)-(0,1pt)$) [nodeCDots]{};
    \node (Xp-1) at (2,-1) [] {$X^{p-1}\vphantom{^{p+1}}$};
    \node (Xp) at (4,-1) [] {$X^p\vphantom{^{p+1}}$};
    \node (Xp+1) at (6,-1) [] {$0\vphantom{^{p+1}}$};
    \node (XdotsRight1) at ($(7.8,-1)-(0,1pt)$) [nodeCDots]{};
    \node (XdotsRight2) at ($(8,-1)-(0,1pt)$) [nodeCDots]{};
    \node (XdotsRight3) at ($(8.2,-1)-(0,1pt)$) [nodeCDots]{};

    \draw[->] ([xshift=4pt]XdotsLeft3.east)--([yshift=-1pt]Xp-1.west)node[midway,above,scale=.7]{$d_{X^\bullet}^{p-2}$};
    \draw[->] ([yshift=-1pt]Xp-1.east)--([yshift=-1pt]Xp.west)node[midway,above,scale=.7]{$d_{X^\bullet}^{p-1}$};
    \draw[->] ([yshift=-1pt]Xp.east)--([yshift=-1pt]Xp+1.west);
    \draw[->] ([yshift=-1pt]Xp+1.east)--([xshift=-4pt]XdotsRight1.west);

    \node (X) at (-1.3,-2.05) [anchor=east] {$\brutalTrunc_{\geq p}X^\bullet\colon$};
    \node (XdotsLeft1) at ($(-.2,-2)-(0,1pt)$) [nodeCDots]{};
    \node (XdotsLeft2) at ($(0,-2)-(0,1pt)$) [nodeCDots]{};
    \node (XdotsLeft3) at ($(.2,-2)-(0,1pt)$) [nodeCDots]{};
    \node (Xp-1) at (2,-2) [] {$0\vphantom{^{p+1}}$};
    \node (Xp) at (4,-2) [] {$X^p\vphantom{^{p+1}}$};
    \node (Xp+1) at (6,-2) [] {$X^{p+1}\vphantom{^{p+1}}$};
    \node (XdotsRight1) at ($(7.8,-2)-(0,1pt)$) [nodeCDots]{};
    \node (XdotsRight2) at ($(8,-2)-(0,1pt)$) [nodeCDots]{};
    \node (XdotsRight3) at ($(8.2,-2)-(0,1pt)$) [nodeCDots]{};

    \draw[->] ([xshift=4pt]XdotsLeft3.east)--([yshift=-1pt]Xp-1.west)node[midway,above,scale=.7]{};
    \draw[->] ([yshift=-1pt]Xp-1.east)--([yshift=-1pt]Xp.west)node[midway,above,scale=.7]{};
    \draw[->] ([yshift=-1pt]Xp.east)--([yshift=-1pt]Xp+1.west)node[midway,above,scale=.7]{$d_{X^\bullet}^{p}$};
    \draw[->] ([yshift=-1pt]Xp+1.east)--([xshift=-4pt]XdotsRight1.west)node[midway,above,scale=.7]{$d_{X^\bullet}^{p+1}$};

\end{tikzpicture}
\]
Note that the brutal truncations are not functors. However, they do give rise to a triangle for each complex $X\in D^b(\cA)$, namely
\[
\begin{tikzpicture}[baseline]
    \node (A) at (0,0) [] {$\brutalTrunc_{\geq p+1}(X^\bullet)\vphantom{[1]}$};
    \node[right=1.5cm of A] (B) {$X^\bullet\vphantom{[1]^\bullet_{\geq 1}}$};
    \node[right=1.5cm of B] (C) {$\brutalTrunc_{\leq p}(X^\bullet)\vphantom{[1]}$};
    \node[right=1.5cm of C](Ashift){$\shift{\brutalTrunc_{\geq p+1}(X^\bullet)}$,};
    \draw[->] ([yshift=1pt]A.east)--([yshift=1pt]B.west);
    \draw[->] ([yshift=1pt]B.east)--([yshift=1pt]C.west);
    \draw[->] ([yshift=1pt]C.east)--([yshift=1pt]Ashift.west);
\end{tikzpicture}
\]

\subsubsection{The Nakayama functor}
Let $\Lambda$ be a finite-dimensional $\K$-algebra and $D= \Hom_\K(-,\K)$ the standard $\K$-linear duality. 
The Nakayama functors 
\[\nu\coloneqq D\Hom_\Lambda(-,A)\colon \mod\Lambda\to\mod\Lambda\text{ and }\nu^-\coloneqq \Hom_\Lambda(DA,-)\colon \mod\Lambda\to \mod\Lambda
\]
restrict to quasi-inverses $\nu\colon \proj\Lambda\leftrightarrow \inj\Lambda \colon \nu^-$. They also induce an equivalence
\[
    \nu\colon \homotopy^b(\proj\Lambda)\to \homotopy^b(\inj\Lambda)
\]
with quasi-inverse
\[
    \nu^-\colon \homotopy^b(\inj\Lambda)\to \homotopy^b(\proj\Lambda)
\]

\subsection{Quivers} In this subsection we introduce notation and elementary results on quivers and representations; for a more comprehensive account on this subject we refer to the textbooks \cite{ARS}, \cite{ASS06} and \cite{Rin84}.

A \emph{quiver} $Q=(Q_0,Q_1,s,t)$ is a directed graph, consisting of a set $Q_0$ of \emph{vertices}, a set $Q_1$ of \emph{arrows} and maps $s,t\colon Q_1\to Q_0$. 
Given an arrow $\alpha\in Q_0$, $s(\alpha)=v$ is the \emph{source} and $t(\alpha)=w$ is the \emph{target}. If both $Q_0$ and $Q_1$ are finite, we call $Q$ a \emph{finite} quiver. 
We call $Q$ \emph{locally finite} if both the set $Q_1(v,-)$ of arrows $\alpha$ with $s(\alpha)=v$ and the set $Q_1(-,v)$ of arrows $\beta$ with $t(\beta)=v$ are finite for all vertices $v\in Q_0$. 

A \emph{path} in $Q$ is either a \emph{trivial} path $e_v$ at a vertex $v$ or a concatenation of arrows $p=\alpha_l\alpha_{l-1}\cdots \alpha_1$ and $s(\alpha_i)=t(\alpha_{i-1})$ for $1<i\leq l$. The maps $s,t$ extends to paths by letting $s(p)=s(\alpha_1)$ and $t(p)=t(\alpha_l)$
We denote by $Q(v,w)$ the set of all paths from $v$ to $w$. 

The \emph{path algebra} $\K Q$ of $Q$ is the $\K$-algebra whose elements are $\K$-linear sums of paths in $Q$. The multiplication is given by concatenation if the paths are compatible and zero otherwise, i.e. $p_2\cdot p_1=p_2p_1$ if $s(p_2)=t(p_1)$.
Note that if $Q$ is finite then $\K Q$ has an identity given by the sum of all trivial paths. 

A $\K$-linear combination $\sum c_i p_i$ of paths $p_i$, with $c_i\in \K$ is a \emph{relation} on $\K Q$ if the paths have length at least two and all share the same source and target. Given a set of relations $\rho$ of a quiver $Q$, we call the pair $(Q,\rho)$ a \emph{bound quiver}. 
The path algebra of a bound quiver is given by the quotient $\K Q/\langle \rho\rangle$. If $Q$ is finite and there is a $k>0$ such that $\langle \rho\rangle$ contain all paths of length $k$, then $\K Q/\langle \rho\rangle$ is a finite dimensional algebra.

To a bound quiver $(Q,\rho)$ we can also associate an additive $\K$-category, where the indecomposable objects are the vertices, morphisms are the paths modulo relations and composition is given as in the path algebra. 
A $(Q,\rho)$-representation is then a $\K$-linear functor from this category to the category of $\K$-vector spaces. 
Representations $M$ such that $\sum_{v\in Q_0}\dim_\K M(v)<\infty$ are called finite-dimensional, and the category of all finite-dimensional $(Q,\rho)$-representations is denoted by $\rep(Q,\rho)$. 
This is an abelian Krull-Schmidt category, and we have an equivalence $\mod\K Q/\langle \rho\rangle\simeq \rep(Q,\rho)$, which we will be using as an identification whenever $Q$ is finite.

A \emph{morphism of quivers} $f\colon (Q_0,Q_1,s,t)\to (Q'_0,Q'_1,s',t')$ is a pair of maps $f_0\colon Q_0\to Q'_0$ and $f_1\colon Q_1\to Q_1'$ such that $s'(f_1(\alpha))=f_0(s(\alpha))$ and $t'(f_1(\alpha))=f_0(t(\alpha))$. 
Further, a \emph{morphism of bounded quivers} $f\colon (Q,\rho)\to (Q',\rho')$ is a quiver morphism such that $\overline{f}(\langle \rho\rangle)\subseteq \langle \rho'\rangle$, where $\overline{f}$ is the induced linear map $\overline{f}\colon \K Q\to \K Q'$.

\section{\texorpdfstring{$m$}{m}-extended module categories}
Let $\Lambda$ be a finite dimensional algebra over a field $\K$. The extended module category $\mmod\Lambda$ is defined to be the full extension-closed subcategory 
\[
    \begin{split}
        \mmod\Lambda&\coloneqq \D^{\leq 0}(\mod\Lambda)\cap \D^{\geq -(m-1)}(\mod\Lambda)\\
        &=\{X^\bullet \in \D^b(\mod\Lambda) \ | \ H^i (X^\bullet)=0\ \text{ for }i\notin[-(m-1),0]\,\}.
    \end{split}
\]
For objects $X^\bullet$ and $Y^\bullet$ in $\mmod\Lambda$, denote the Hom-space by $\Hom_{m\mhyphen\Lambda}(X^\bullet,Y^\bullet)\coloneqq\Hom_{\D^b(\mod\Lambda)}(X^\bullet,Y^\bullet)$.

Observe that since $\mmod\Lambda$ is a full subcategory of $\D^b(\mod\Lambda)$, it is $\K$-linear, $\Hom$-finite and Krull-Schmidt. 
Note also that as an extension-closed subcategory of a triangulated category, $\mmod\Lambda$ carries an \emph{extriangulated} structure \cite{NP19}, however, seeing as we work with this explicit category we will not recall the big machinery of extriangulations. 

Let us look at how representatives of objects in $\mmod\Lambda$ may look. We can first note that by using soft truncations we can always choose a representative $X^\bullet$, such that $X^i=0$ for $i\notin[-(m-1),0]$. 
Furthermore, we may also observe that $\softTrunc^{\geq -(m-1)}\projres X^\bullet\cong X^\bullet\cong \softTrunc^{\leq 0}\injres X^\bullet$, where $\projres X^\bullet$ ($\injres X^\bullet$) is a minimal projective (injective) resolution of $X^\bullet$. 
Hence, we can choose representatives that have projective terms in all but the $-(m-1)$th-term or injective terms in all but the zeroth-term.

Both the notion of projectives and injectives can be carried over to the extended setting. Let projective objects be defined as those objects $P^\bullet\in \mmod\Lambda$ for which $\Hom_{m\mhyphen\Lambda}(P^\bullet,\shift{-})$ vanish on $\mmod\Lambda$. 
Likewise, injective objects $I^\bullet\in\mmod\Lambda$ are defined through $\Hom_{m\mhyphen\Lambda}(-,\shift{I^\bullet})$ vanishing on $\mmod\Lambda$. 
The following description of projective and injective objects can be found in \cite[Proposition 2.14]{Zho25} (see also \cite[Corollary 1.4]{AST08}).

\begin{lemma}
    $X^\bullet$ is projective in $\mmod\Lambda$ if and only if $X^\bullet\cong P$ for $P$ projective in $\mod\Lambda$. Dually, $Y^\bullet$ is injective in $\mmod\Lambda$ if and only if $Y^\bullet\cong \nshift{I}{m-1}$ for $I$ injective in $\mod\Lambda$.
\end{lemma}

\subsection{AR-triangles}
In \cite{Zho25} the existence of Auslander-Reiten triangles in the extended module category was proven through the method of finding AR-triangles in subcategories of triangulated categories from \cite{Jor09}, together with the existence of AR-triangles ending in compact objects of $\homotopy(\inj\Lambda)$, proven in \cite{KL06}. 
We will shortly recall the existence result, but let us first spend a bit of time recalling the definition of AR-triangles and some of their properties.

\begin{definition}[{\cite[Definition 1.3]{Jor09}}]\label{def:ARtriangles}
    A triangle
    \[
    \begin{tikzpicture}[anchor=base, baseline]
        \node (A) at (0,0) [] {$A^\bullet\vphantom{[1]}$};
        \node[right=1.5cm of A] (B) {$B^\bullet\vphantom{[1]}$};
        \node[right=1.5cm of B] (C)  {$C^\bullet\vphantom{[1]}$};
        \node[right=1.5cm of C](Ashift) {$\shift{A^\bullet}$,};
        \draw[->] ([yshift=1pt]A.east)--([yshift=1pt]B.west)node[midway,above,scale=.8]{$\alpha$};
        \draw[->] ([yshift=1pt]B.east)--([yshift=1pt]C.west)node[midway,above,scale=.8]{$\beta$};
        \draw[->] ([yshift=1pt]C.east)--([yshift=1pt]Ashift.west)node[midway,above,scale=.8]{$\delta$};
    \end{tikzpicture}
    \]
    in $\D^b(\mod\Lambda)$ with $A^\bullet,B^\bullet$ and $C^\bullet$ in $\mmod\Lambda$, is called an \emph{Auslander-Reiten triangle} (AR-triangle) in $\mmod\Lambda$ if
    \begin{enumerate}[align=left]
        \item[\hypertarget{ARcondition1}{(AR1)}] the morphism $\delta\colon C^\bullet \to \shift{A^\bullet}$ is non-zero,
        \item[\hypertarget{ARcondition2}{(AR2)}] if $\alpha'\colon A^\bullet\to X^\bullet$ is not a section, then it factors through $\alpha$
        \[
        \begin{tikzpicture}
            \node (A) at (0,0) [] {$A^\bullet$};
            \node (B) at (1.5,0) [] {$B^\bullet$};
            \node (X) at (0,-1) [] {$X^\bullet$};
            \draw[->] (A)--(B)node[midway,above,scale=.8]{$\alpha$};
            \draw[->] (A)--(X)node[midway,left,scale=.8]{$\alpha'$};
            \draw[->,dashed] (B)--(X)node[midway,anchor=north west,scale=.8]{$\exists$};
        \end{tikzpicture}
        \]
        \item[\hypertarget{ARcondition3}{(AR3)}] if $\beta'\colon Y^\bullet\to C^\bullet$ is not a retraction, then it factors through $\beta$ 
        \[
        \begin{tikzpicture}
            \node (B) at (-.5,0) [] {$B^\bullet$};
            \node (C) at (1,0) [] {$C^\bullet$};
            \node (Y) at (1,1) [] {$Y^\bullet$};
            \draw[->] (B)--(C)node[midway,above,scale=.8]{$\beta$};
            \draw[->] (Y)--(C)node[midway,right,scale=.8]{$\beta'$};
            \draw[->,dashed] (Y)--(B)node[midway,anchor=south east,scale=.8]{$\exists$};
        \end{tikzpicture}
        \]
    \end{enumerate}
\end{definition}
Since $\mmod\Lambda$ is an extension closed subcategory of $\D^b(\mod\Lambda)$ we can replace condition (AR2) with the following condition.
\begin{enumerate}[align=left]
    \item[\hypertarget{ARcondition2marked}{(AR2')}]  $A^\bullet$ and $C^\bullet$ are indecomposable.\label{def:ARtriangles2tilde}
\end{enumerate} 
The fact that (AR2') follows from (AR1), (AR2) and (AR3) is proven in \cite[Lemma 2.4]{Jor09}, 
and the fact that (AR2) follows from (AR1), (AR2') and (AR3) can be proven as for the derived category in \cite[Lemma 4.2]{Hap88}\footnote{only noting that by exactness of $\mmod\Lambda$, the morphism $t_1\colon W\to Y'$ is in $\mmod\Lambda$}. 
It also follows from \cite[Theorem 2.9]{INP24}, in which we find another useful description, namely that
\[
    \begin{tikzpicture}[baseline]
        \node (A) at (0,0) [] {$A^\bullet\vphantom{[1]}$};
        \node[right=1.5cm of A] (B) {$B^\bullet\vphantom{[1]}$};
        \node[right=1.5cm of B] (C) {$C^\bullet\vphantom{[1]}$};
        \node[right=1.5cm of C](Ashift) {$\shift{A^\bullet},$};
        \draw[->] ([yshift=1pt]A.east)--([yshift=1pt]B.west)node[midway,above,scale=.8]{$\alpha$};
        \draw[->] ([yshift=1pt]B.east)--([yshift=1pt]C.west)node[midway,above,scale=.8]{$\beta$};
        \draw[->] ([yshift=1pt]C.east)--([yshift=1pt]Ashift.west)node[midway,above,scale=.8]{$\delta$};
    \end{tikzpicture}
    \]
is an AR-triangle if and only if $\alpha$ is a \emph{source morphism}, if and only if $\beta$ is a \emph{sink morphism}. 
A morphism $\alpha$ is a source morphism if it satisfies \hyperlink{ARcondition2}{(AR2)} and is \emph{left minimal}, meaning if $h\colon B^\bullet\to B^\bullet$ is an endomorphism such that $h\circ\alpha=\alpha$, then $h$ is an automorphism. 
\emph{Right minimal} morphisms and sink morphisms are defined dually.

Basic facts about AR-theory of Krull-Schmidt categories can be found in \cite{Liu10} and \cite{Rin84}. 
Let us give some highlights before stating the existence theorem from \cite{Zho25}. We denote the Jacobson radical of a Krull-Schmidt category $\cA$ by $\rad\cA$ (see e.g. \cite{Bau83} for background). Recall that for indecomposables $X,Y\in\cA$, $\rad(X,Y)$ are the non-isomorphisms $X\to Y$.
A morphism $f\colon X\to Y$ in $\cA$ is \emph{irreducible} if it is neither a section nor a retraction, and if $f=g\circ h$, then either $g$ is a section or $h$ is a retraction. 
Equivalently, $f\colon X\to Y$ is irreducible if it lies in $\rad(X,Y)\setminus\rad^2(X,Y)$ \cite[Proposition 2.4]{Bau83}. 

Irreducible morphisms are intrinsically related to AR-triangles as the following result tells us.
\begin{lemma}[{\cite[Thm. 1.4]{Liu10}}]\label{thm:irreducibleInARsequence}
    Let 
    \begin{equation}\label{eq:arTriangleAndIrreducibles}
    \begin{tikzpicture}[baseline]
        \node (A) at (0,0) [] {$A^\bullet\vphantom{[1]}$};
        \node[right=1.5cm of A] (B)  {$B^\bullet\vphantom{[1]}$};
        \node[right=1.5cm of B] (C) {$C^\bullet\vphantom{[1]}$};
        \node[right=1.5cm of C](Ashift) {$\shift{A^\bullet}$,};
        \draw[->] ([yshift=1pt]A.east)--([yshift=1pt]B.west)node[midway,above,scale=.8]{$\alpha$};
        \draw[->] ([yshift=1pt]B.east)--([yshift=1pt]C.west)node[midway,above,scale=.8]{$\beta$};
        \draw[->] ([yshift=1pt]C.east)--([yshift=1pt]Ashift.west)node[midway,above,scale=.8]{$\delta$};
    \end{tikzpicture}
    \end{equation}
    be an AR-triangle in $\mmod\Lambda$, then
    \begin{enumerate}
        \item Any AR-triangle ending in $C^\bullet$ or starting in $A^\bullet$ is isomorphic to \eqref{eq:arTriangleAndIrreducibles},
        \item Each irreducible morphism $\alpha_1\colon A^\bullet\to B_1^\bullet$ or $\beta_1\colon B_1^\bullet\to C^\bullet$ fits into an AR-triangle
        \[
    \begin{tikzpicture}[baseline]
        \node[anchor=east] (A) at (-1,0) [] {$A_{\vphantom{1}}^\bullet$};
        \node[right=1.5cm of A] (B) {$B_1^\bullet\oplus B_2^\bullet$};
        \node[right=1.5cm of B] (C) {$C_{\vphantom{1}}^\bullet$};
        \node[right=1.5cm of C](Ashift) {$\shift{A_{\vphantom{1}}^\bullet}.$};
        \draw[->] ([yshift=1pt]A.east)--([yshift=1pt]B.west)node[midway,above,scale=.8]{$
            \begin{psmallmatrix}
                \alpha_1 & \alpha_2
            \end{psmallmatrix}
        $};
        \draw[->] ([yshift=1pt]B.east)--([yshift=1pt]C.west)node[midway,above,scale=.9]{$
            \begin{psmallmatrix}
                \beta_1 \\ \beta_2
            \end{psmallmatrix}
        $};
        \draw[->] ([yshift=1pt]C.east)--([yshift=1pt]Ashift.west)node[midway,above,scale=.9]{};
    \end{tikzpicture}
    \]
    \end{enumerate}
\end{lemma}

In order to state the existence result we are now only left with defining some constructions.
First, we will for objects $X^\bullet\in\mmod\Lambda$ denote by $\projres_m X^\bullet$ the brutal truncation $\brutalTrunc_{\geq -m}\projres X^\bullet$ of a minimal projective resolution of $X^\bullet$, and likewise denote by $\injres_m X^\bullet$ the brutal truncation $\brutalTrunc_{\leq 1}\injres X^\bullet$ of a minimal projective resolution of $X^\bullet$.
We call these the projective/injective presentations of $X^\bullet$.
\begin{definition}
    For any $Z^\bullet\in \mmod\Lambda$, we define
    \[
    \tau_{[m]}(Z^\bullet)=\sigma^{\leq 0}(\nu\nshift{\projres_m(Z^\bullet)}{-1})\quad \text{and}\quad \tau_{[m]}^-(Z^\bullet)=\sigma^{\geq -(m-1)}(\nu^-\shift{\injres_m(Z^\bullet)}).
    \]
\end{definition}
By restricting to $1\nmod\Lambda=\mod\Lambda$, we obtain the usual AR-translates. 
We can, maybe not surprisingly, also see that $\tau_{[m]}$ behaves much like its classical counterpart.
\begin{proposition}[{\cite[Proposition 3.11]{Zho25}}]
    Let $Z^\bullet$ be an indecomposable object in $\mmod\Lambda$.
    \begin{enumerate}
        \item $\tau_{[m]}Z^\bullet$ has no non-zero injective summands in $\mmod\Lambda$.
        \item $\tau_{[m]}^-Z^\bullet$ has no non-zero projective summands in $\mmod\Lambda$.
        \item $Z^\bullet$ is projective in $\mmod\Lambda$ if and only if $\tau_{[m]}Z^\bullet=0$.
        \item $Z^\bullet$ is injective in $\mmod\Lambda$ if and only if $\tau_{[m]}^-Z^\bullet=0$.
        \item If $Z^\bullet$ is not projective in $\mmod\Lambda$, then $\tau_{[m]}Z^\bullet$ is indecomposable and 
        \[
            \tau_{[m]}^-\tau_{[m]}Z^\bullet\cong Z^\bullet.
        \]
        \item If $Z^\bullet$ is not injective in $\mmod\Lambda$, then $\tau_{[m]}^-Z^\bullet$ is indecomposable and 
        \[
            \tau_{[m]}\tau_{[m]}^-Z^\bullet\cong Z^\bullet.
        \]
    \end{enumerate}
\end{proposition}

Most importantly, they are directly related to AR-triangles.
\begin{theorem}[{\cite[Thm. 3.12]{Zho25}}]\label{theorem:ZhouExistenceAR}
    Let $Z^\bullet$ be an indecomposable object in $\mmod\Lambda$.
    \begin{enumerate}
        \item If $Z^\bullet$ is not projective in $\mmod\Lambda$, there is an AR-triangle in $\mmod\Lambda$
        \[
        \begin{tikzpicture}
            \node[anchor=east] (A) at (0,0) {$\tau_{[m]}Z^\bullet\vphantom{[1]_{[m]}}$};
            \node[right = 1.5cm of A] (B) {$Y^\bullet\vphantom{[1]_{[m]}}$};
            \node[right= 1.5cm of B] (C) {$Z^\bullet\vphantom{[1]_{[m]}}$};
            \node[right=1.5cm of C](Ashift) {$\shift{\tau_{[m]}Z^\bullet}.$};
            \draw[->] ([yshift=1pt]A.east)--([yshift=1pt]B.west);
            \draw[->] ([yshift=1pt]B.east)--([yshift=1pt]C.west);
            \draw[->] ([yshift=1pt]C.east)--([yshift=1pt]Ashift.west);
        \end{tikzpicture}
        \]
        \item If $Z^\bullet$ is not injective in $\mmod\Lambda$, there is an AR-triangle in $\mmod\Lambda$
        \[
        \begin{tikzpicture}
            \node[anchor=east] (A) at (0,0) {$Z^\bullet\vphantom{[1]_{[m]}}$};
            \node[right=1.5cm of A] (B) {$Y^\bullet\vphantom{[1]_{[m]}}$};
            \node[right=1.5cm of B] (C) {$\tau_{[m]}^-Z^\bullet\vphantom{[1]_{[m]}}$};
            \node[right=1.5cm of C](Ashift) {$\shift{Z^\bullet}\vphantom{[1]_{[m]}}.$};
            \draw[->] ([yshift=1pt]A.east)--([yshift=1pt]B.west);
            \draw[->] ([yshift=1pt]B.east)--([yshift=1pt]C.west);
            \draw[->] ([yshift=1pt]C.east)--([yshift=1pt]Ashift.west);
        \end{tikzpicture}
        \]
    \end{enumerate}
\end{theorem}

\subsection{AR-quivers}
The information of AR-triangles can be encoded in an AR-quiver. We give a brief treatment of the required notions here, and refer to \cite{Liu10} for the details.
\begin{definition}
    A \emph{translation} of a quiver $Q$ is a bijection $\tau\colon Q_0\setminus Q_0^p\to Q_0\setminus Q_0^i$ for some subsets $Q_0^p,Q_0^i$ of $Q_0$, such that for $x\in Q_0\setminus Q_0^p$ we have $Q_1(\tau x,-)=Q_1(- ,x)$.
    We call a tuple $(Q,\tau)$ a \emph{translation quiver} if $Q$ is a quiver with translation $\tau$. 
\end{definition}

As previously noted, the set of irreducible morphisms between two objects $X^\bullet$ and $Y^\bullet$ in $\mmod\Lambda$ is given by $\rad(X^\bullet,Y^\bullet)\setminus\rad^2(X^\bullet,Y^\bullet)$. 
As the irreducible morphisms are the building blocks for our AR-triangles, we want to consider 
\[
\irr(X^\bullet,Y^\bullet)\coloneqq\rad(X^\bullet,Y^\bullet)/\rad^2(X^\bullet,Y^\bullet).
\] 
$\irr(X^\bullet,Y^\bullet)$ is a $D_{X^\bullet}\mhyphen D_{Y^\bullet}$-bimodule where $D_{Z^\bullet}=\End(Z^\bullet)/\rad(Z^\bullet,Z^\bullet)$.

\begin{defprop}[{\cite[Prop. 2.1]{Liu10}}]
    The \emph{Auslander-Reiten quiver} $\AR(\mmod\Lambda)$ of $\mmod\Lambda$ is the translation quiver $(\AR(\mmod\Lambda),\tau_{[m]})$ given by
    \begin{itemize}
        \item The vertices are the indecomposable objects of $\mmod\Lambda$, 
        \item The arrows are given by existence of irreducible morphisms, i.e.\\ $\AR(\mmod\Lambda)_1$ has an arrow $X^\bullet\to Y^\bullet$ if $\irr(X^\bullet,Y^\bullet)\neq 0$
        \item The translate, called the \emph{Auslander-Reiten translation}, is given by $\tau_{[m]} Z^\bullet=X^\bullet$ if and only if there is an AR-triangle $X^\bullet\to Y^\bullet\to Z^\bullet\to \shift{X^\bullet}$.
    \end{itemize}
\end{defprop}

\begin{remark}
    The general situation calls for a valuation of the arrows in the AR-quiver, however all our examples will have the trivial valuation, so we omit this complication.
\end{remark}

In general there are indecomposable objects for which the translation is not defined. 
In $\mod\Lambda$ we know that these  objects are in fact the indecomposable projectives, and as shown in \cite[Thm. 3.12 and Prop. 3.11]{Zho25} this is also the case for the extended module category $\mmod\Lambda$. 
However, the irreducible morphisms into projectives are still well-behaved.

\begin{lemma}\label{lemma:radPsink}
    \begin{enumerate}
        \item Let $P$ be an indecomposable projective module in $\mmod\Lambda$, then $$\iota\colon\rad P\to P$$ is a sink morphism in $\mmod\Lambda$.
        \item Let $\nshift{I}{m-1}$ be an indecomposable injective module in $\mmod\Lambda$, then $$\nshift{\pi}{m-1}\colon\nshift{I}{m-1}\to \nshift{(I/\soc(I))}{m-1}$$ is a source morphism in $\mmod\Lambda$.
    \end{enumerate}
\end{lemma}
\begin{proof}
    We only argue the projective case. In order for the morphism to be a sink morphism, we need 
    \begin{enumerate}
        \item[(a)] \textit{(minimality)} that any $h\in \End_{\mmod\Lambda}( \rad P)$, such that $\iota\circ h=\iota$, is an automorphism,
        \item[(b)] \textit{(almost split)} if $g\colon X^\bullet \to P$ is not a retraction, then $g=\iota\circ a$ for some $a\colon X^\bullet\to \rad P$.
    \end{enumerate} 
    
    (a) Let $M,N\in \mod\Lambda$, then we have 
    \[
    \Hom_{m\mhyphen\Lambda}(M,N)=\Hom_{\Db(\mod\Lambda)}(M,N)=\Hom_{\mod\Lambda}(M,N).
    \] 
    Thus, any morphism $h\in \End_{m\mhyphen\Lambda}(\rad)=\End_{\mod\Lambda}(\rad P)$ is an automorphism, since $\iota\colon \rad P \to P$ is a monomorphism in $\mod\Lambda$. 

    For (b), let us utilize the equivalence $\D^-(\mod\Lambda)\to\homotopy^-(\proj\Lambda)$. Let $X^\bullet\in \homotopy^-(\proj\Lambda)$ be in the image of $\mmod\Lambda$ under this equivalence, that is $X^i=0$ for $i>0$ and $\Homology^i(X^\bullet)= 0$ for $i\notin[-(m-1),0]$.
    Assume we have a morphism $f\colon X^\bullet \to P$ in $\homotopy^-(\proj\Lambda)$ which does not factor through $\projres(\rad P)$. Then we can see that $f^0\colon X^0\to P$ is a retraction in $\mod\Lambda$. 
    That is, we have $g\colon P\to X^0$ such that $f^0\circ g=\mathrm{id}_P$. Obviously, we then can find $\overline{g}\colon P \to X^\bullet$ in $\homotopy^-(\proj\Lambda)$, such that $f\circ \overline{g}=\mathrm{id}_P$, and we are done.
    \[
    \begin{tikzpicture}
        \node (X) at (0,0) [] {$X^\bullet$};
        \node (P) at (0,-1) [] {$P$};
        \draw[->] (X)--(P)node[midway,left]{$f$};
        \draw[dashed,->] (P) to[out=60,in=-60]node[midway,right]{$\overline{g}$} (X);
        \node (Xm-1) at (2,0) [] {$X^{-(m-1)}$};
        \node (Xdots2) at (4,0) [] {$\cdots$};
        \node (X1) at (6,0) [] {$X^{-1}$};
        \node (X0) at (8,0) [] {$X^0$};
        \draw[->] (Xm-1)edge(Xdots2) (Xdots2)edge(X1) (X1)edge(X0);
        
        \node (Pm-1) at (2,-1) [] {$0$};
        \node (Pdots2) at (4,-1) [] {$\cdots$};
        \node (P1) at (6,-1) [] {$0$};
        \node (P0) at (8,-1) [] {$P$};
        \draw[->] (Pm-1)edge(Pdots2) (Pdots2)edge(P1) (P1)edge(P0);

        \draw[->] (X0)--(P0)node[midway,left]{$f^0$};
        \draw[dashed,->] (P0) to[out=60,in=-60]node[midway,right]{$g$} (X0);
        \draw[->] (X1)--(P1)node[midway,left]{$f^1$};
        \draw[dashed,->] (P1) to[out=60,in=-60]node[midway,right]{$\overline{g}^1$} (X1);
        \draw[->] (Xm-1)--(Pm-1)node[midway,left]{$f^{m-1}$};
        \draw[dashed,->] (Pm-1) to[out=60,in=-60]node[midway,right]{$\overline{g}^{m-1}$} (Xm-1);
        
    \end{tikzpicture}
    \]
    
\end{proof}
It follows that $\AR(\mmod\Lambda)$ is a locally finite quiver.
Before moving on, let us name two components of the AR-quivers which, if they exists, are especially well-behaved.
\begin{definition}
    Let $\Lambda$ be a finite dimensional $\K$-algebra.
    \begin{enumerate}
        \item An indecomposable object $X^\bullet$ in $\mmod\Lambda$ is \emph{postprojective} if there exists some $l\geq 0$ such that $\tau_{[m]}^lX^\bullet$ is projective. 
    Further, an acyclic connected component $\cP$ of $\AR(\mmod\Lambda)$ is called a \emph{postprojective component} if every object in $\cP$ is postprojective
        \item Dually, an indecomposable object $Y$ in $\mmod\Lambda$ is \emph{preinjective}, if $Y=\tau^lI$ for $l\geq 0$ and $I$ injective. 
    A \emph{preinjective} component $\cI$ is an acyclic connected component of $\AR(\mmod\Lambda)$ whose objects are all preinjective.
    \end{enumerate}
\end{definition}

\subsection{Knitting AR-quivers}
In order to work with the extended module category it would help if we could obtain the AR-quiver in an algorithmic way. The following result along with \Cref{lemma:radPsink} which reflects the module-case are the first steps towards that goal. 

\begin{lemma}
    \begin{enumerate}
        \item Let $P$ be a simple projective in $\mod\Lambda$, and let $f\colon P\to X^\bullet$ be an irreducible morphism in $\mmod\Lambda$ with $X^\bullet$ indecomposable. Then $X^\bullet$ is projective.
        \item Let $I$ be a simple injective in $\mod\Lambda$, and let $g\colon Y^\bullet\to \nshift{I}{m-1}$ be an irreducible morphism in $\mmod\Lambda$ with $Y^\bullet$ indecomposable. Then $Y^\bullet$ is injective.
    \end{enumerate}
\end{lemma}
\begin{proof}
    We only prove the first claim as the second one is similar. Assume to the contrary that $X^\bullet$ is not projective, then we can find an AR-sequence
    \[
    \begin{tikzpicture}[baseline]
        \node[anchor=east] (A) at (-1,0) {$\tau_{[m]}X^\bullet\vphantom{[1]_{[m]}}$};
        \node[right = 1.5cm of A] (B) {$P^\bullet\oplus Z^\bullet\vphantom{[1]_{[m]}}$};
        \node[right=1.5cm of B] (C)  {$X^\bullet\vphantom{[1]_{[m]}}$};
        \node[right=1.5cm of C](Ashift)  {$\shift{\tau_{[m]}X^\bullet}\vphantom{[1]_{[m]}}.$};
        \draw[->] ([yshift=1pt]A.east)--([yshift=1pt]B.west)node[midway,above,scale=.9]{$
            \begin{psmallmatrix}
                \alpha_1 & \alpha_2
            \end{psmallmatrix}
        $};
        \draw[->] ([yshift=1pt]B.east)--([yshift=1pt]C.west)node[midway,above,scale=.9]{$
            \begin{psmallmatrix}
                f \\ \beta_2
            \end{psmallmatrix}
        $};
        \draw[->] ([yshift=1pt]C.east)--([yshift=1pt]Ashift.west)node[midway,above,scale=.9]{};
    \end{tikzpicture}
    \]
    Hence, there is some irreducible morphism $\alpha\colon \tau_{[m]}X^\bullet\to P$, which by \Cref{lemma:radPsink} factors through $\rad P=0$, however this is a contradiction, and we are done.
\end{proof}

Now, we can see that under certain constraints\footnote{The criteria here is sufficient, but not necessary. 
For our use case it is however not too strict.} we may algorithmically \emph{knit} together a postprojective component.
The algorithm is essentially the same as for the classical case; we refer to the proof of \cite[Lemma IX.4.5]{ASS06} or \cite[Sections 7.2 and 7.3]{Bar15} for this. 
In \Cref{prop:AlgorithmForKnittingDimensionVectors} we provide the algorithm when complexes are replaced with an invariant resembling the dimension vector of modules.
\begin{lemma}\label{lemma:postprojective2}
    Let $\Lambda$ be a connected  algebra where the radical of indecomposable projectives is either zero or indecomposable. If $\Lambda$ has a simple projective $Q$, then there exists a postprojective component $\cP$ of $\AR({\mmod\Lambda})$ containing $Q$.
\end{lemma}
\begin{proof}
    The proof is mutatis mutandis the same as the construction part of \cite[Lemma IX.4.5]{ASS06}. 
    The main difference is that by assuming indecomposable radicals we can easily handle the projectives without considering a subalgebra of $\Lambda$. 
    That is, if $M^\bullet$ lie in $\mathcal{P}_n$ and an immediate successor $P$ of $M^\bullet$ is projective, then $\rad P=M^\bullet$ and thus $\cP_{n+1}$ is obviously closed under predecessors of $P$. 
\end{proof}

We also have a dual construction for preinjective components.

\begin{lemma}\label{lemma:preinjective1}
    Let $\Lambda$ be a connected algebra such that $I/\soc I$ is either zero or indecomposable for each indecomposable injective $I$. If $\Lambda$ has a simple injective $E$, then there exists a preinjective component $\cI$ of $\AR({\mmod\Lambda})$ which contain $\nshift{E}{m-1}$.\qed
\end{lemma}

Let us illustrate this on a simple example.

\begin{example}\label{ex:2modCommutativeSquare}
    Consider the algebra $\Lambda$ given by the commutative square,
    \[
    \begin{tikzpicture}[tips=proper,scale=.8,anchor=base, baseline]
        \node (1) at (1,0) [nodeDots] {};
        \node(1-label) at (.9,0) [anchor=east,scale=.8]{$4$};
        \node (2) at (2,1) [nodeDots] {};
        \node (2-label) at (2,1.1)[above,scale=.8]{$3$};
        \node (3) at (2,-1) [nodeDots] {};
        \node (3-label) at (2,-1.1)[below,scale=.8]{$2$};
        \node (4) at (3,0) [nodeDots] {};
        \node (4-label) at (3.1,0) [anchor=west,scale=.8]{$1$};
        \draw[-{Latex}] ([xshift=3pt,yshift=3pt]1.center)edge([xshift=-3pt,yshift=-3pt]2.center) 
        ([xshift=3pt,yshift=-3pt]2.center)edge([xshift=-3pt,yshift=3pt]4.center) 
        ([xshift=3pt,yshift=-3pt]1.center)edge([xshift=-3pt,yshift=3pt]3.center) 
        ([xshift=3pt,yshift=3pt]3.center)edge([xshift=-3pt,yshift=-3pt]4.center);
        \draw[dotted,thick] ([xshift=4pt]1.center)--([xshift=-4pt]4.center);
    \end{tikzpicture}
    \]
    This algebra satisfies the hypothesis of both preceding lemmas, hence we know that $\mmod\Lambda$ has both a postprojective and preinjective component. 
    We can also knit these components; the postprojective forwards from $P_1$, or the preinjective backwards from $\nshift{I_4}{m-1}$. 
    Here we have done this for $2\nmod\Lambda$\footnote{Where the indecomposables are represented through their composition series.}, for which the postprojective and preinjective component coincide:
    \[
    \includegraphics{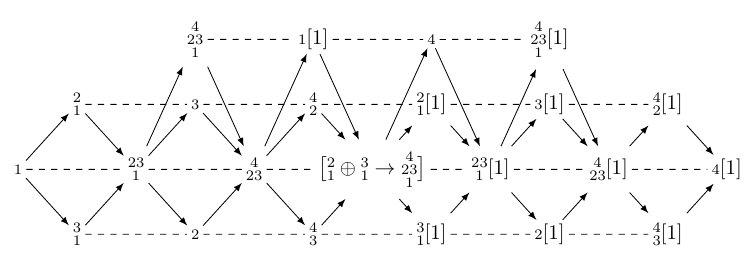}
    \]
\end{example}

\subsubsection{Knitting dimensions}
As working with complexes can be quite complex, we would like to simplify. 
In $\mod\Lambda$ one can simplify the process through only working with dimension vectors. 
It turns out that we can do something similar in the extended module category, using \emph{cohomological dimension vectors}.
\begin{definition}\label{def:cohomologicalDimension}
    Let $X^\bullet\in \D^b(\mod\Lambda)$ be a complex. The \emph{cohomological dimension vector} of $X^\bullet$ is the sequence
    \[
        \DimVec(\boldX)\coloneqq (\dimVec(H^iX^\bullet))_{i\in\bZ}
    \] 
    where $\dimVec(H^i(X^\bullet))$ is the dimension vector of the $i$-th cohomology of $\boldX$
\end{definition}
Note that for complexes $X^\bullet$ in $\mmod\Lambda$, we may think of $\DimVec(X^\bullet)$ as $m\times n$-matrices, where $|\Lambda|=n$, since their cohomology is non-zero only if $i\in [-(m-1),0]$.
We will moreover be writing these matrices such that the first row give the zeroth cohomology, the second row the $-1$th cohomology and so forth. 

Before we show how this notion can simplify the knitting procedure, let us note that it encodes a bit of information about the AR-quiver as a whole. 
Let $\{e_0,e_1,\ldots,e_{n-1}\}$ be a complete set of primitive orthogonal idempotents of $\Lambda$, and consider $X^\bullet\in \D^b(\mod\Lambda)$. Then 
\[
\begin{split}
    (\DimVec(X^\bullet)_i)_v&=\dimVec(H^{i}(X^\bullet))_v\\
    &=\Hom_\Lambda(P_v,H^{i}(X^\bullet))\\
    &\cong H^{i}\Hom_{\Lambda}(P_v,X^\bullet)\\
    &\cong\Hom_{\D^b(\mod\Lambda)}(\nshift{P_v}{-i},X^\bullet)
\end{split}
\]
where $P_v$ is the indecomposable projective at vertex $v$. Likewise,
\[
\begin{split}
    (\DimVec(X^\bullet)_i)_x&\cong \Hom_\Lambda(H^i(X^\bullet),I_v)\\
    &\cong H^i\Hom_\Lambda(X^\bullet,I_v)\\
    &\cong \Hom_{\D^b(\mod\Lambda)}(X^\bullet,\nshift{I_v}{-i})\\
\end{split}
\]
Now, we may observe that 
for $0\leq k<m-1$ we have an AR-triangle
\[
\begin{tikzpicture}[baseline]
        \node (A) at (0,0) [] {$\nshift{I_x}{k}\vphantom{[1]_x}$};
        \node[right=1.5cm of A] (B) {$Y^\bullet\vphantom{[1]_x}$};
        \node[right=1.5cm of B] (C){$\nshift{P_x}{k+1}\vphantom{[1]_x}$};
        \node[right=1.5cm of C](Ashift){$\nshift{I_x}{k+1}\vphantom{[1]_x}$,};
        \draw[->] ([yshift=1pt]A.east)--([yshift=1pt]B.west);
        \draw[->] ([yshift=1pt]B.east)--([yshift=1pt]C.west);
        \draw[->] ([yshift=1pt]C.east)--([yshift=1pt]Ashift.west);
    \end{tikzpicture}
\]
in $\mmod\Lambda$. 

\begin{lemma}\label{lemma:pathFromProjToInj}
    Let $X^\bullet$ be an indecomposable object in $\mmod\Lambda$. If $(\DimVec(X^\bullet)_{-i})_v\neq 0$ for some $i\in[0,-(m-1)]$, then there is a sequence of non-zero morphisms from $P_v$ to $X^\bullet$ and a sequence of non-zero morphisms from $X^\bullet$ to $\nshift{I_v}{m-1}$. 
\end{lemma}

Let us look at exactly how $\DimVec$ can help us knit.
For a general triangle $X^\bullet\to Y^\bullet\to Z^\bullet\to \shift{X^\bullet}$, we can only guarantee the inequality $\DimVec(Y^\bullet)\leq \DimVec(X^\bullet)+\DimVec(Z^\bullet)$. 
However, if the triangle is an AR-triangle in $\mmod\Lambda$, we can see that the cohomological dimension vector of one part of the triangle follows from the other two.

\begin{lemma}\label{lemma:knittingDimensions}
    Let 
    \[
    \begin{tikzpicture}[baseline]
        \node (A) at (0,0) [] {$A^\bullet\vphantom{[1]}$};
        \node[right=1.5cm of A] (B) {$B^\bullet\vphantom{[1]}$};
        \node[right=1.5cm of B] (C) {$C^\bullet\vphantom{[1]}$};
        \node[right=1.5cm of C](Ashift) {$\shift{A^\bullet}$,};
        \draw[->] ([yshift=1pt]A.east)--([yshift=1pt]B.west)node[midway,above,scale=.8]{$\alpha$};
        \draw[->] ([yshift=1pt]B.east)--([yshift=1pt]C.west)node[midway,above,scale=.8]{$\beta$};
        \draw[->] ([yshift=1pt]C.east)--([yshift=1pt]Ashift.west)node[midway,above,scale=.8]{$\delta$};
    \end{tikzpicture}
    \]
    be an AR-triangle in $\mmod\Lambda$. The following is equivalent
    \begin{enumerate}
        \item $C^\bullet$ is not a summand of $\nshift{\Lambda}{i}$ for any $i\in[1,m-1]$,
        \item $A^\bullet$ is not a summand of $\nshift{D\Lambda}{i+1}$ for any $i\in[0,m-2]$,
        \item $\DimVec(C^\bullet)_{i}=\DimVec(B^\bullet)_{i}-\DimVec(A^\bullet)_{i}$ for all $i\in[-(m-1),0]$.
    \end{enumerate}
   On the other hand, if $A^\bullet=\nshift{I}{k}$ for an indecomposable injective $I$ and $k\in [0,m-1]$, then
    \begin{equation*}
        \DimVec(C^\bullet)_{-(k+1)}=\DimVec(A^\bullet)_{-k}+\DimVec(B^\bullet)_{-(k+1)}-\DimVec(B^\bullet)_{-k}.
    \end{equation*}
    and $\DimVec(C^\bullet)_i=0$ for $i\neq -(k+1)$.
\end{lemma}
\begin{proof}
    We first look at the equivalent statements. Consider the long exact sequence in $\Db(\mod\Lambda)$ induced by cohomology
    \[
        \cdots \to H^{i-1}(C^\bullet)\to H^{i}(A^\bullet)\to H^{i}(B^\bullet)\to H^{i}(C^\bullet)\to H^{i+1}(A^\bullet)\to\cdots
    \]
    Using $H^{-i}(-)\cong\Hom_{\Db(\mod\Lambda)}(\nshift{\Lambda}{i},-)$, we see that $H^{-i}(B^\bullet)\to\Homology^{-i}(C^\bullet)$ is epimorphic if and only if 
    $$
    \beta\circ -\colon\Hom_{\Db(\mod\Lambda)}(\nshift{\Lambda}{i},B^\bullet)\to\Hom_{\Db(\mod\Lambda)}(\nshift{\Lambda}{i},C^\bullet)
    $$
    is epimorphic. If $C^\bullet$ is not a summand of $\nshift{\Lambda}{i}$, then every morphism $\nshift{\Lambda}{i}\to C^\bullet$ factors through $\beta$ by definition of AR-triangles.
    If $C^\bullet$ is a summand of $\nshift{\Lambda}{i}$ and $\beta\circ -$ is epi, then $\pi=\beta\circ\pi'$ for some $\pi\colon \nshift{\Lambda}{i}\to B^\bullet$, making $\beta$ a retraction, which is a contradiction.
    The arguments for $D\Lambda[i+1]$ are similar, using $H^{-i}(-)\cong D\Hom(-,\nshift{D\Lambda}{i}).$

    If $A^\bullet=\nshift{I}{k}$, then for some $Y^\bullet\in\mmod\Lambda$ the AR-triangle is given as
    \[
    \begin{tikzpicture}[anchor=base, baseline]
        \node (A) at (0,0) [] {$\nshift{I}{k}\vphantom{[1]^\bullet}$};
        \node (B) at (2,0) [] {$Y^\bullet\vphantom{[1]^\bullet}$};
        \node (C) at (4.5,0) [] {$\nshift{\nu^-I}{k+1}\vphantom{[1]^\bullet}$};
        \node(Ashift) at (6.7,0) [anchor=base west] {$\nshift{I}{k+1}\vphantom{[1]^\bullet}$,};
        \draw[->] ([yshift=1pt]A.east)--([yshift=1pt]B.west);
        \draw[->] ([yshift=1pt]B.east)--([yshift=1pt]C.west);
        \draw[->] ([yshift=1pt]C.east)--([yshift=1pt]Ashift.west);
    \end{tikzpicture}
    \]
    and the claim follows from the long exact sequence of cohomology
\[
\begin{tikzpicture}[xscale=.9,every node/.style={scale=.9}]
    \node (leftDots1) at (-2.1,0) [nodeCDots]{};
    \node (leftDots2) at (-2,0) [nodeCDots]{};
    \node (leftDots3) at (-1.9,0) [nodeCDots]{};
    \node (A) at (0,0) [] {$H^{-(k+1)}(\nshift{I}{k})$};
    \node (B) at (3,0) [] {$H^{-(k+1)}(Y^\bullet)\vphantom{[1]}$};
    \node (C) at (6.6,0) [] {$H^{-(k+1)}(\nshift{\nu^-I}{k+1})\vphantom{[1]}$};
    \node (D) at (0,-1) [] {$H^{-k}(\nshift{I}{k})$};
    \node (E) at (3,-1) [] {$H^{-k}(Y^\bullet)\vphantom{[1]}$};
    \node (F) at (6.6,-1) [] {$H^{-k}(\nshift{\nu^-I}{k+1})\vphantom{[1]}$};
    \node (rightDots1) at (8.9,-1) [nodeCDots]{};
    \node (rightDots2) at (9,-1) [nodeCDots]{};
    \node (rightDots3) at (9.1,-1) [nodeCDots]{};
    \draw[->] ([xshift=.2cm]leftDots3.center)--(A);
    \draw[->] (A)--(B);
    \draw[->] (B)--(C);
    \draw[->,rounded corners=2pt] (C.east) --(9,0)--(9,-.5)--(-2,-.5)--(-2,-1)--   (D.west);
    \draw[->] (D)--(E);
    \draw[->] (E)--(F);
    \draw[->] (F)--([xshift=-.2cm]rightDots1.center);
    \draw[->,red] (A.south west)--(A.north east)node[right]{$0$};
    \draw[->,red] (F.south west)--(F.north east)node[right]{$0$};
\end{tikzpicture}
\]
\end{proof}

\begin{example}
    Let us revisit \Cref{ex:2modCommutativeSquare}. 
    The reader may verify that the following is the AR-quiver of $2\nmod\Lambda$, where the objects are represented through their cohomological dimension vectors.
    \[
    \includegraphics{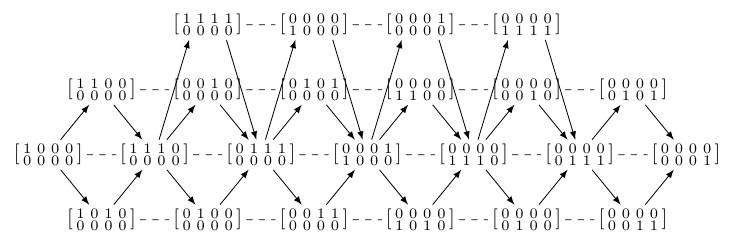}
    \]
\end{example}

In the classical case, modules in the postprojective component are uniquely given through their dimension vector. 
Hence, we are naturally wondering whether the same can be said in the extended module categories as well. 
We provide the following partial answer, and urge those interested to explore further.
\begin{lemma}
    Let $\nshift{M}{k}$ be a postprojective object in $\mmod\Lambda$, then $\DimVec(X^\bullet)=\DimVec(\nshift{M}{k})$ implies $X^\bullet\cong \nshift{M}{k}$.
\end{lemma}
\begin{proof}
    First, observe that since $\DimVec(X^\bullet)$ is concentrated in a single degree, we can use soft truncations to show that $X^\bullet$ is isomorphic to a stalk complex.
    If $\nshift{M}{k}$ is postprojective in $\mmod\Lambda$ it necessarily has to be postprojective in $\mod\Lambda$. Otherwise we could find an infinite sequence of non-zero morphisms from a projective $P$ to $M$ in $\mod\Lambda$, which then give an infinite sequence of non-zero morphisms from $\nshift{P}{k}$ to $\nshift{M}{k}$. 
    The result now follows from uniqueness of dimension vectors of postprojectives in $\mod\Lambda$, see e.g. \cite[Theorem]{Hap82}.
\end{proof}

Through combining what we have done so far, we can now extrapolate the following knitting algorithm for cohomological dimension vectors.
Note that the restrictions are strict and one could find a more general algorithm, but they suffice for our study.

\begin{proposition}[Knitting algorithm]\label{prop:AlgorithmForKnittingDimensionVectors}
    Let $\Lambda$ be a connected not semi-simple $\K$-algebra such that radicals of indecomposable projectives are indecomposable. 
    For each $0\leq j\leq m-1$, let
    \[
    \mathcal{Q}^{\inj,\, j}\coloneqq \{\DimVec(\nshift{I}{j})\,|\, I\in \inj\Lambda \text{ indecomposable }\}.
    \]
    We inductively define quivers ${}_i\mathcal{Q}$.
    \begin{enumerate}
        \item[(0)] Let ${}_{-1}\mathcal{Q}=\varnothing$ and let ${}_0\mathcal{Q}$ be the quiver given by vertices $\DimVec(P)$ for each simple projective $P$ and no arrows.
        \item[(1)] Construct ${}_{1}\mathcal{Q}$ from ${}_0\mathcal{Q}$ by adding vertices $\DimVec(P)$ and arrows $[M\to \DimVec(P)]$ for each $M\in {}_0\mathcal{Q}$ such that $M=\DimVec(\rad P)$, for an indecomposable projective $P$.
        \item[(k)] Assume ${}_1\mathcal{Q}$ is constructed for $i<k$. We construct ${}_k\mathcal{Q}$ from ${}_{k-1}\mathcal{Q}$ as follows.\\
            \begin{enumerate}
                \item For each vertex $M\in {}_{k-2}\mathcal{Q}\setminus{}_{k-3}\mathcal{Q}$ such that $M\notin \mathcal{Q}^{\inj,\, m-1}$ let
            \[
                M^+=\{\, N\in {}_{k-1}\mathcal{Q}\, |\, \text{there is an arrow }[M\to N]\in {}_{k-1}\mathcal{Q}\}
            \] 
            \begin{enumerate}
                \item if $M\in \mathcal{Q}^{\inj,\, j}$ for some $0\leq j<m-1$, define $\tau^-M$ by
                    \[
                        (\tau^- M)_{-j}=M_{-j}+\sum_{N\in M^+}\left(N_{-(j+1)}-N_{-j}\right)
                    \]
                    and $(\tau^- M)_{-i}=0$ for $i\neq j$.
                \item if $M\notin \mathcal{Q}^{\inj,\, j}$ for any $0\leq j< m-1$, define $\tau^-M$ By
                    \[
                        \tau^- M=-M+\sum_{N\in M^+}N
                    \]
            \end{enumerate}
            Add the vertex $\tau^-M$ to ${}_{k-1}\mathcal{Q}$ and for each $N\in M^+$ add an arrow $[N\to \tau^- M]$.
                \item For each vertex $M\in {}_{k-1}\mathcal{Q}\setminus {}_{k-2}\mathcal{Q}$ such that $M=\DimVec(\rad P)$ for an indecomposable projective $P$, add the vertex $\DimVec(P)$ and an arrow $[M\to \DimVec(P)]$.
            \end{enumerate}
    \end{enumerate}
    Then, ${}_{\infty}\mathcal{Q}=\bigcup_{i=0}^\infty{}_{i}\mathcal{Q}$ is a postprojective component of $\mmod\Lambda$. 
\end{proposition}

\begin{example}
    Let $\Lambda$ now be the path algebra over the quiver
    \[
    \begin{tikzpicture}[anchor=base, baseline]
        \node (A) at (0,0) [nodeDots] {};
        \node (A-label) at (0,-.1) [below,scale=.8]{$3$};
        \node (B) at (1.2,0) [nodeDots] {};
        \node (B-label) at (1.2,-.1) [below,scale=.8]{$2$};
        \node (C) at (2.4,0) [nodeDots] {};
        \node (C-label) at (2.4,-.1) [below,scale=.8]{$1$};
        \node (D) at (3.6,0) [nodeDots] {};
        \node (D-label) at (3.6,-.1) [below,scale=.8]{$0$};
        \draw[-latex] ([xshift=3pt]A.center)--([xshift=-3pt]B.center)node[midway,above,scale=.7]{$\alpha$};
        \draw[-latex] ([xshift=3pt]B.center)--([xshift=-3pt]C.center)node[midway,above,scale=.7]{$\beta$};
        \draw[-latex] ([xshift=3pt]C.center)--([xshift=-3pt]D.center)node[midway,above,scale=.7]{$\delta$};

        \draw[dashed] (1.9,.1) to[out=30,in=150] (2.9,.1);
    \end{tikzpicture}
    \]
    with relation $(\beta\delta)$. 
    We can use the algorithm above to find $\AR(2\nmod\Lambda)$ given through $\DimVec$ as follows.
    \[
    \includegraphics[width=.9\linewidth]{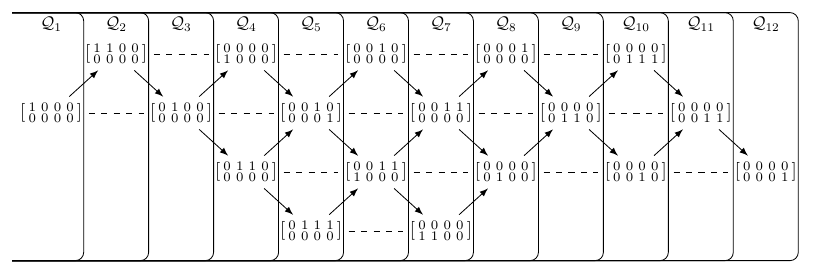}
    \]
\end{example}

We have implemented the algorithm in Python to calculate postprojective components for extended module categories of linear Nakayama algebras with homogeneous relations, $\Lambda(n,l)$. 
The results of these calculations inspired \Cref{tab:FiniteNakayama}.

\begin{remark}
    $\DimVec$ encodes information about the AR-quiver. Subsequently, we can also infer knowledge about $\DimVec$ from the AR-quiver. As a prelude to the upcoming section on linear Nakayama algebras, we showcase this on the AR-quiver of $3\nmod\Lambda(9,3)$.
    Recall that $\Lambda(9,3)$ is the path algebra of the bound quiver:
\[
\begin{tikzpicture}[xscale=.6,tips=proper,baseline=(current  bounding  box.center)]
    \node (8) at (0,0) [nodeDots] {};
    \node (8-label) at (0,-.1) [below,scale=.7] {$8$};
    \node (7) at (2,0) [nodeDots] {};
    \node (7-label) at (2,-.1) [below,scale=.7] {$7$};
    \node (6) at (4,0) [nodeDots] {};
    \node (6-label) at (4,-.1) [below,scale=.7] {$6$};
    \node (5) at (6,0) [nodeDots] {};
    \node (5-label) at (6,-.1) [below,scale=.7] {$5$};
    \node (4) at (8,0) [nodeDots] {};
    \node (4-label) at (8,-.1) [below,scale=.7] {$4$};
    \node (3) at (10,0) [nodeDots] {};
    \node (3-label) at (10,-.1) [ below,scale=.7]{$3$};
    \node (2) at (12,0) [nodeDots] {};
    \node (2-label) at (12,-.1) [below,scale=.7] {$2$};
    \node (1) at (14,0) [nodeDots] {};
    \node (1-label) at (14,-.1) [below,scale=.7] {$1$};
    \node (0) at (16,0) [nodeDots] {};
    \node (0-label) at (16,-.1) [below,scale=.7] {$0$};

    \draw[-latex] ([xshift=.2cm]8.center)--([xshift=-.2cm]7.center);
    \draw[-latex] ([xshift=.2cm]7.center)--([xshift=-.2cm]6.center);
    \draw[-latex] ([xshift=.2cm]6.center)--([xshift=-.2cm]5.center);
    \draw[-latex] ([xshift=.2cm]5.center)--([xshift=-.2cm]4.center);
    \draw[-latex] ([xshift=.2cm]4.center)--([xshift=-.2cm]3.center);
    \draw[-latex] ([xshift=.2cm]3.center)--([xshift=-.2cm]2.center);
    \draw[-latex] ([xshift=.2cm]2.center)--([xshift=-.2cm]1.center);
    \draw[-latex] ([xshift=.2cm]1.center)--([xshift=-.2cm]0.center);

    \draw[dashed] (1,.1) to[out=20,in=160] (5,.1);
    \draw[dashed] (3,.1) to[out=20,in=160] (7,.1);
    \draw[dashed] (5,.1) to[out=20,in=160] (9,.1);
    \draw[dashed] (7,.1) to[out=20,in=160] (11,.1);
    \draw[dashed] (9,.1) to[out=20,in=160] (13,.1);
    \draw[dashed] (11,.1) to[out=20,in=160] (15,.1);
\end{tikzpicture}
\]
where the relations are given by all paths of length $3$. The AR-quiver of $3\nmod\Lambda(9,3)$ is given by
\[
    \includegraphics[width=\linewidth]{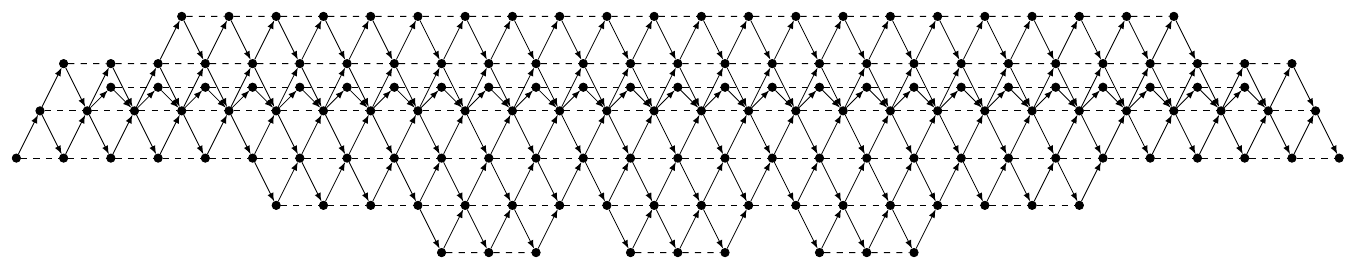}
\]
Using \Cref{lemma:pathFromProjToInj} we can find sections of the AR-quiver whose objects $X^\bullet$ have $(\DimVec(X^\bullet)_i)_6= 0$ for all $i\in[-2,0]$. 
\[
    \includegraphics[width=\linewidth]{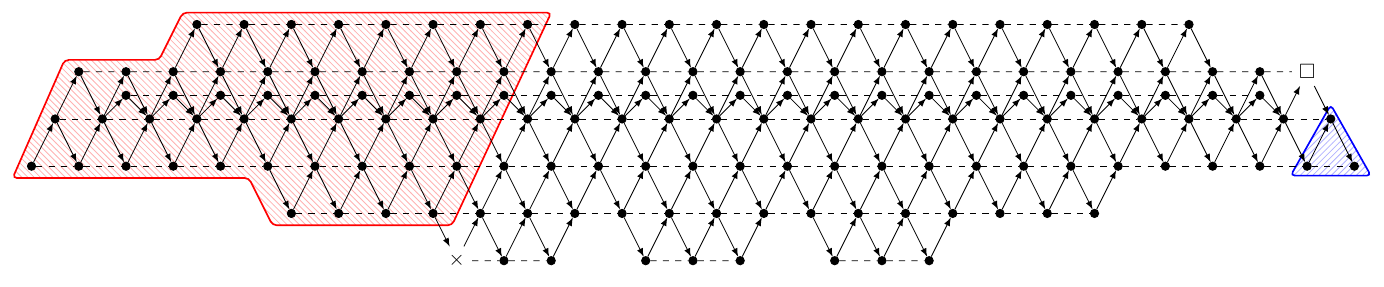}
\]
First, observe that the part encased in red do not admit a path from $P_6$ (Marked with a cross) into it, hence if $X^\bullet$ lie here then $(\DimVec(X^\bullet)_i)_6= 0$ for all $i\in [-2,0]$. Second, the part encased in blue do not admit a path into $\nshift{I_6}{2}$ (marked with a square), thus any $Y^\bullet$ in there has $(\DimVec(Y^\bullet)_i)_6= 0$ for all $i\in [-2,0]$.

\end{remark}

\subsection{A criterion for finiteness}
In order to show the validity of \Cref{tab:FiniteNakayama}, we need some way of determining when the extended module categories are of finite type. A very fundamental property of the extended module categories is that 
\[
\mod\Lambda\subseteq 2\nmod\Lambda\subseteq3\nmod\Lambda\subseteq \cdots \subseteq \mmod\Lambda\subseteq\cdots
\]
for any algebra $\Lambda$. Thus, if we already know that $\mmod\Lambda$ is of finite type, we also know that $m'\nmod\Lambda$ is of finite type for $m'\leq m$. Likewise, if $n\nmod\Lambda$ is of infinite type, then $n'\nmod\Lambda$ is of infinite type for all $n'\geq n$.

In addition, mirroring the classical case, the AR-quiver can be used to conclude on the finiteness of $\mmod\Lambda$.
\begin{lemma}\label{lemma:finiteComponentIsWholeAR}
    Let $\Lambda$ be a connected $\K$-algebra which is not semi-simple. Assume $C$ is a finite and connected component of $\AR(\mmod\Lambda)$, then $C=\AR(\mmod\Lambda)$.
\end{lemma}
\begin{proof}
    By \cite[Lemma 2.2.6]{Rin84} it is enough to show that for any $X^\bullet\in\mmod\Lambda$ we can find a sequence of non-zero morphisms from a projective to $X^\bullet$. Which, as long as $X^\bullet$ is non-zero, follows from \Cref{lemma:pathFromProjToInj}. 
    
\end{proof}

\begin{lemma}\label{lemma:connectedComponentBothPostProjAndPreInj}
    Assume $\Lambda$ is a connected $\K$-algebra which is not semi-simple and $C$ is a connected component of $\AR(\mmod\Lambda)$ which is both postprojective and preinjective, then $C$ is finite and $C=\AR(\mmod\Lambda)$. 
\end{lemma}
\begin{proof}
    Every object $X^\bullet \in C$ lies in the $\tau_{[m]}$-orbit of both a projective and an injective, and since each injective in $C$ is postprojective and each projective in $C$ is preinjective there are only finitely many objects in each orbit. 
    Combining this with the fact that there are only finitely many projectives (injectives), we see that $C$ is finite. Hence, we are done by \Cref{lemma:finiteComponentIsWholeAR}.
\end{proof}

Finally, we have the following result which will let us reduce questions of whether an algebra has en infinite extended module category, to the same question for a smaller algebra.

\begin{lemma}\label{lemma:restrictionAdjointExtended}
    Let $\res_e\colon \mod\Lambda\to \mod\Gamma$ be the restriction functor given by an idempotent $e\in\Lambda$ and $\Gamma=e\Lambda e$. Then $\res_e$ and it's left adjoint $L_e\coloneqq \Gamma e\otimes_\Gamma -\colon \mod\Gamma\to\mod\Lambda$, induce an adjoint pair $(\tilde{\bL} L_e,\res_e)$ of the $m$-extended module categories,
    \[
    \begin{tikzcd}
        \mmod\Lambda \arrow[bend right,swap]{r}{\res_e} & \mmod\Gamma \arrow[bend right,swap]{l}{\tilde{\bL} L_e}
    \end{tikzcd}
    \]
    where $\tilde{\bL} L_e\coloneqq \sigma_\Lambda^{\geq -(m-1)}\circ\bL L_e=\sigma_\Lambda^{\geq -(m-1)}\circ L_e\circ \mathbf{p}(-)$. Moreover, $\tilde{\bL} L_e$ is fully faithful.
\end{lemma}
\begin{proof}
    Note that the restriction functor $\mathrm{res_e}$ is exact and the left adjoint $L_e$ is fully faithful and preserves projectives \cite[Theorem I.6.8]{ASS06}. 
    Also note that by \cite[Lemma 15.6]{Kel96} (see also \cite[Proposition 4.3.16]{Kra21}), an adjoint pair 
    \[
    L\colon \mod\Gamma \rightleftarrows \mod\Lambda \colon R,
    \]
    give rise to an adjoint pair 
    \[
    \bL L\colon \Db(\mod\Gamma)\leftrightarrows \Db(\mod\Lambda)\colon \bR R,
    \]
    where $\bL L$ is the left derived functor of $L$ and $\bR R$ is the right derived functor of $R$.

    Recall that the truncation functor $\sigma_{\Lambda}^{\geq -(m-1)}\colon \Db(\mod\Lambda)\to \cD_\Lambda^{\geq -(m-1)}$ is left adjoint to the inclusion. Hence, we can see that
    \[
    \begin{split}
        \Hom_{m\mhyphen\Lambda}(\tilde{\bL}L_e(Y^\bullet),X^\bullet)&\cong \Hom_{\Db(\mod\Lambda)}(\bL L_e(Y^\bullet),X^\bullet)\\
        &\cong \Hom_{\Db(\mod\Gamma)}(Y^\bullet,\res_e(X^\bullet))\\
        &=\Hom_{m\mhyphen\Gamma}(Y^\bullet,\res_e(X^\bullet))
    \end{split}
    \]
    for $X^\bullet\in\mmod\Lambda$ and $Y^\bullet\in\mmod\Gamma$. We used here that since $\res_e$ is exact, we have $\res_e=\bR \res_e$.
    
    Since $L_e$ is fully faithful, $\res_e\circ L_e(-)\cong \mathrm{id}_{-}$. Thus 
    \[
    \begin{split}
        \res_e\circ\tilde{\bL}(-)&=\res_e\circ \sigma^{\geq -(m-1)}\circ L_e\circ \projres(-)\\
        &\cong \sigma^{\geq -(m-1)}\circ \res_e\circ L_e\circ \projres(-)\\
        &\cong \sigma^{\geq -(m-1)}\circ \mathrm{id}_{\projres(-)}\\
        &=\sigma^{\geq-(m-1)}\circ \projres(-)\\
        &\cong (-)
    \end{split}
    \]
    where the first isomorphism follows from $\res_e$ being exact. It follows that $\tilde{\bL} L_e$ is fully faithful,
    \[
    \begin{split}
        \Hom_{m\mhyphen\Gamma}(Y^\bullet,\overline{Y}^\bullet)&\cong \Hom_{m\mhyphen\Gamma}(Y^\bullet,\res_e\circ\tilde{\bL} L_e(\overline{Y}^\bullet)\\
        &\cong \Hom_{m\mhyphen\Lambda}(\tilde{\bL} L_e(Y^\bullet),\tilde{\bL} L_e(\overline{Y}^\bullet)).
    \end{split}
    \]
\end{proof}

\begin{corollary}\label{Corol:lemma:restrictionAdjointExtended}
    Let $\Lambda$ be an algebra and $e$ an idempotent of $\Lambda$. If $\mmod\Lambda$ is of finite type, then $\mmod e\Lambda e$ is of finite type.
\end{corollary}

\section{Linear Nakayama algebras}
A natural question to ask is when $\mmod\Lambda$ is of finite type. 
We chose to explore this in the relatively nicely behaved world of linear Nakayama algebras with homogeneous relations. 
That is, the algebras $\Lambda(n,l)$, given by the quiver 
\[
\begin{tikzpicture}[xscale=.9,tips=proper,baseline=(current  bounding  box.center)]
    \node (d) at (0,0) [nodeDots] {};
    \node (d-label) at (0,-.1) [below,scale=.7] {$n-1$};
    \node (d-1) at (2,0) [nodeDots] {};
    \node (d-1-label) at (2,-.1) [below,scale=.7] {$n-2$};
    \node (cdot1) at (3.8,0) [nodeCDots] {};
    \node (cdot2) at (4,0) [nodeCDots] {};
    \node (cdot3) at (4.2,0) [nodeCDots] {};
    \node (2) at (6,0) [nodeDots] {};
    \node (2-label) at (6,-.1) [below,scale=.7] {$2$};
    \node (1) at (8,0) [nodeDots] {};
    \node (1-label) at (8,-.1) [below,scale=.7] {$1$};
    \node (0) at (10,0) [nodeDots] {};
    \node (0-label) at (10,-.1) [below,scale=.7] {$0$};

    \draw[-latex] ([xshift=.2cm]d.center)--([xshift=-.2cm]d-1.center)node[midway,above,scale=.7]{$\alpha_{n-1}$};
    \draw[-latex] ([xshift=.2cm]d-1.center)--([xshift=-.2cm]cdot1.center)node[midway,above,scale=.7]{$\alpha_{n-2}$};
    \draw[-latex] ([xshift=.2cm]cdot3.center)--([xshift=-.2cm]2.center)node[midway,above,scale=.7]{$\alpha_3$};
    \draw[-latex] ([xshift=.2cm]2.center)--([xshift=-.2cm]1.center)node[midway,above,scale=.7]{$\alpha_2$};
    \draw[-latex] ([xshift=.2cm]1.center)--([xshift=-.2cm]0.center)node[midway,above,scale=.7]{$\alpha_1$};
\end{tikzpicture}
\]
and relations given by paths of length $l$. 

\subsubsection*{Overview of the section} In this section we are proving the following theorem.
\begin{theorem}\label{thm:WhenIsLinearNakayamaFinite}
    $\mmod\Lambda(n,l)$ is of finite type if and only if it is contained in \Cref{tab:FiniteNakayama}.
\end{theorem}
The proof is split into two subsections, in the first we show that when $\mmod(\Lambda(n,l))$ is in the table, then it is in fact of finite type, see \Cref{Prop:WhenNakayamaFinite}.
In the second section we show that outside of these cases, $\mmod\Lambda(n,l)$ is of infinite type, see \Cref{prop:linNakHomOfInfType}.

\subsubsection*{Description of modules} The indecomposable modules in $\mod\Lambda(n,l)$ can be described through their support, which is always an interval $[a,b]\subseteq[0,n-1]$ of length $\leq l$.
We denote by $M_{a,b}$, the indecomposable with support $[a,b]$.
An indecomposable $M_{a,b}$ is given by $M_i=\K$ for $x\in [a,b]$ and $0$ elsewhere;
The structure maps are the identity where possible and zero otherwise.
\[
\begin{tikzpicture}[scale=.7,every node/.style={scale=.9}]
    \node (M) at (-1,0)[] {$M_{a,b}\colon$};
    \node (cdotLeft1) at (-.2,0)[nodeCDots]{};
    \node (cdotLeft2) at (0,0)[nodeCDots]{};
    \node (cdotLeft3) at (.2,0)[nodeCDots]{};
    \node (b+1) at (2,0) [] {$0$};
    \node at (2,-.2)[below,scale=.7] {$b+1\vphantom{b+1}$};
    \node (b) at (4,0) [] {$\K$};
    \node at (4,-.2)[below,scale=.7] {$b\vphantom{b+1}$};
    \node (b-1) at (6,0) [] {$\K$};
    \node at (6,-.2)[below,scale=.7] {$b-1\vphantom{b+1}$};
    \node (cdotMiddle1) at (7.8,0) [nodeCDots]{};
    \node (cdotMiddle2) at (8,0) [nodeCDots]{};
    \node (cdotMiddle3) at (8.2,0) [nodeCDots]{};
    \node (a+1) at (10,0)[] {$\K$};
    \node at (10,-.2)[below,scale=.7] {$a+1\vphantom{b+1}$};
    \node (a) at (12,0) [] {$\K$};
    \node at (12,-.2)[below,scale=.7] {$a\vphantom{b+1}$};
    \node (a-1) at (14,0) [] {$0$};
    \node at (14,-.2)[below,scale=.7] {$a-1\vphantom{b+1}$};
    \node (cdotRight1) at (15.8,0)[nodeCDots]{};
    \node (cdotRight2) at (16,0) [nodeCDots]{};
    \node (cdotRight3) at (16.2,0) [nodeCDots]{};

    \draw[->] ([xshift=.3cm]cdotLeft3.center)--([xshift=-.3cm]b+1.center)node[midway,above,scale=.8]{$0$};
    \draw[->] ([xshift=.3cm]b+1.center)--([xshift=-.3cm]b.center)node[midway,above,scale=.8]{$0$};
    \draw[->] ([xshift=.3cm]b.center)--([xshift=-.3cm]b-1.center)node[midway,above,scale=.8]{$1$};
    \draw[->] ([xshift=.3cm]b-1.center)--([xshift=-.3cm]cdotMiddle1.center)node[midway,above,scale=.8]{$1$};
    \draw[->] ([xshift=.3cm]cdotMiddle3.center)--([xshift=-.3cm]a+1.center)node[midway,above,scale=.8]{$1$};
    \draw[->] ([xshift=.3cm]a+1.center)--([xshift=-.3cm]a.center)node[midway,above,scale=.8]{$1$};
    \draw[->] ([xshift=.3cm]a.center)--([xshift=-.3cm]a-1.center)node[midway,above,scale=.8]{$0$};
    \draw[->] ([xshift=.3cm]a-1.center)--([xshift=-.3cm]cdotRight1.center)node[midway,above,scale=.8]{$0$};
\end{tikzpicture}
\]
The projectives are either those with support starting in $0$ or with length exactly $l$, and injectives are those with support ending in $n-1$ or length exactly $l$. 
For the simples, projectives and injectives in $i\in[0,n-1]$ we therefore have
\[
    S_i=M_{i,i},\quad
    P_i=\begin{cases}
        M_{0,i}&\text{ if }i\leq l-1\\
        M_{i-l+1,i}&\text{ otherwise}
    \end{cases},\quad
    I_i=\begin{cases}
        M_{i,n-1}&\text{ if }i\geq n-l\\
        M_{i,i+l-1}&\text{ otherwise.}
    \end{cases}
\]
Note that $P_i=I_{i-l+1}$ for $l-1\leq i\leq n-1$. We may also note that $\rad M_{a,b}=M_{a,b-1}$ and $\soc M_{a,b}= S_a$. 

The projective and injective resolutions of $M_{a,b}$ are easily calculated:
\[
    \begin{tikzpicture}[xscale=1,every node/.style={scale=.9}]
        \node at (-8.9,1) [anchor=east] {$\projres M_{a,b}\colon$};
        \node (cdotRight1) at (1.9,1)[nodeCDots]{};
        \node (cdotRight2) at (2,1)[nodeCDots]{};
        \node (cdotRight3) at (2.1,1)[nodeCDots]{};
        \node (P1) at (.7,1) [] {$0$};
        \node (M) at (0,0) [] {$M_{a,\,b}$};
        \node (P0) at (-1,1) [] {$P_b$};
        \node (OM) at (-2,0) [] {$M_{b-l+1,\,a-1}$};
        \node (P-1) at (-3,1) [] {$P_{a-1}$};
        \node (O2M) at (-4,0) [] {$M_{a-l,\,b-l}$};
        \node (P-2) at (-5,1)  [] {$P_{b-l}$};
        \node (O3M) at (-6,0) [] {$M_{b-2l+1,\,a-l-1}$};
        \node (P-3) at (-7,1) [] {$P_{a-l-1}$};
        \node (cdotLeft1) at (-7.9,0) [nodeCDots]{};
        \node (cdotLeft2) at (-8,0) [nodeCDots]{};
        \node (cdotLeft3) at (-8.1,0) [nodeCDots]{};
        \node (cdot2Left1) at (-8.4,1) [nodeCDots]{};
        \node (cdot2Left2) at (-8.5,1) [nodeCDots]{};
        \node (cdot2Left3) at (-8.6,1) [nodeCDots]{};
        \draw[->] ([xshift=.2cm]cdot2Left1.center)--(P-3);
        \draw[->] (P-3)--(P-2);
        \draw[->] (P-2)--(P-1);
        \draw[->] (P-1)--(P0);
        \draw[->] (P0)--(P1);
        \draw[->] (P1)--([xshift=-.2cm]cdotRight1.center);
        \draw[->>] (P0)--(M);
        \draw[>->] (OM)--(P0);
        \draw[->>] (P-1)--(OM);
        \draw[>->] (O2M)--(P-1);
        \draw[->>] (P-2)--(O2M);
        \draw[>->] (O3M)--(P-2);
        \draw[->>] (P-3)--(O3M);
        \draw[>->] ([xshift=.25cm,yshift=.25cm]cdotLeft2.center)--(P-3);
    \end{tikzpicture}
\]
\[
\begin{tikzpicture}[xscale=1,every node/.style={scale=.9}]
    \node at (-2.4,1)[anchor=east] {$\injres M_{a,b}\colon$};
    \node (cdotLeft1) at (-2.1,1) [nodeCDots]{};
    \node (cdotLeft2) at (-2,1) [nodeCDots] {};
    \node (cdotLeft3) at (-1.9,1) [nodeCDots] {};
    \node (I-1) at (-.7,1) [] {$0$};
    \node (M) at (0,0) [] {$M_{a,\,b}$};
    \node (I0) at (1,1) [] {$I_{1}$};
    \node (UM) at (2,0) [] {$M_{b+1,\,a+l-1}$};
    \node (I1) at (3,1) [] {$I_{b+1}$};
    \node (U2M) at (4,0) [] {$M_{a+l,\,b+l}$};
    \node (I2) at (5,1) [] {$I_{a+l}$};
    \node (U3M) at (6,0) [] {$M_{b+l+1,\,a+2l-1}$};
    \node (I3) at (7,1) [] {$I_{b+l+1}$};
    \node (cdotRight1) at (7.9,0)[nodeCDots]{};
    \node (cdotRight2) at (8,0) [nodeCDots]{};
    \node (cdotRight3) at (8.1,0) [nodeCDots]{};
    \node (cdot2Right1) at (8.4,1) [nodeCDots]{};
    \node (cdot2Right2) at (8.5,1) [nodeCDots]{};
    \node (cdot2Right3) at (8.6,1) [nodeCDots]{};

    \draw[>->](M)--(I0);
    \draw[>->](UM)--(I1);
    \draw[>->](U2M)--(I2);
    \draw[>->](U3M)--(I3);
    \draw[->>] (I0)--(UM);
    \draw[->>] (I1)--(U2M);
    \draw[->>] (I2)--(U3M);
    \draw[->>] (I3)--([xshift=-.3cm,yshift=.3cm]cdotRight2.center);
    \draw[->] ([xshift=.2cm]cdotLeft3.center)--(I-1);
    \draw[->] (I-1)--(I0);
    \draw[->] (I0)--(I1);
    \draw[->] (I1)--(I2);
    \draw[->] (I2)--(I3);
    \draw[->] (I3)--([xshift=-.2cm]cdot2Right1.center);

\end{tikzpicture}
\]
Finally, we may see that $\nu P_i=I_i=P_{i+l-1}$ if $0\leq i\leq n-l$ and $\nu^- I_i=P_i=I_{i-l+1}$ for $l-1\leq i\leq n-1$.

\subsection{The finite cases} We start by showing when $\mmod\Lambda(n,l)$ is of finite type.
The work can be significantly reduced by two initial and simple observations:
\begin{enumerate}
    \item 
    By \cite[Table 1]{HS10}, we know that
    $\Lambda(n,l)$ for $2\leq l\leq n-1$ is piecewise hereditary of Dynkin type if and only if
    \begin{itemize}
        \item $n\leq 7$,
        \item $n=8$ and $l\neq 4$, or
        \item $l=2$ or $l=n-1$.
    \end{itemize}
    For a piecewise hereditary algebra $\Lambda$ of Dynkin type, we have that $\mmod\Lambda$ is of finite type for all $m<\infty$ \cite{Vos01}.
    \item $1\nmod\Lambda=\mod\Lambda$, thus if $\Lambda$ is representation finite then $1\nmod$ is obviously of finite type.
\end{enumerate}
Hence, the parts of \Cref{tab:FiniteNakayama} marked "All" and "$m=1$" are shown.
We are left with showing:
\begin{proposition}\label{Prop:WhenNakayamaFinite}
    Let $2\leq l\leq n-1$, then $\AR(\mmod\Lambda(n,l))$ is finite if
    \begin{enumerate}
        \item $m\leq 4$ and $l=3$,
        \item $m=2$, and either $l=4$ or $l=5$.
    \end{enumerate}
\end{proposition}

It is sufficient to show finiteness of $4\nmod\Lambda(n,3)$ with $n\geq 9$, $2\nmod(n,4)$ with $n\geq 8$ and $2\nmod\Lambda(n,5)$ with $n\geq 9$. 
For each of these three cases we will show that there is a connected component in the AR-quiver which is both postprojective and preinjective, then by \Cref{lemma:connectedComponentBothPostProjAndPreInj} we are done.
The manner in which we show this is through constructing a finite path from the simple projective to the simple injective, which we know lie in a postprojective and preinjective component respectively (\Cref{lemma:postprojective2,lemma:preinjective1}).

\subsubsection{Case: \texorpdfstring{$2\nmod\Lambda(n,4)$}{2-mod Lambda(n,4)}}
A preliminary study of a handful calculated examples offers a strategy. 
We can observe from the AR-quivers of $2\nmod\Lambda(n,4)$ for $n=8,9,10,11$ (see \Cref{fig:ARquiver2modLambdan4}) and some preliminary calculations in QPA \cite{qpa}, that our effort should be spent on the $\tau_{[2]}$-orbit of the simple projective $P_0$. We also observe that depending on the parity of $n$, we may expect to either obtain the simple injective $\shift{I_{n-1}}$ in this orbit or the injective $\shift{I_{n-4}}$.

Note that $P_k=I_{k-3}$ for $3\leq k\leq n-1$, and $\tau_{[m]}^-(I_i[j])=\nshift{P_i}{j-1}$ if $j<m-1$. 

\begin{lemma}
    Let $n\geq 6$, then we have
    \[
    \tau^{-k}_{[2]}P_0=\begin{cases}
        \shift{I_{n-1}}&\text{, for }k=3(n-3)+2 \text{ if }n\text{ is odd,}\\
        \shift{I_{n-6}}&\text{, for }k=3(n-4)+1 \text{ if }n\text{ is even.}
    \end{cases}
    \]
    Moreover, we have $\rad P_{2+2(i+2)}=\tau^{-6i}P_0$ for $1\leq i\leq (n-5)/5$.
\end{lemma}
\begin{proof}
    The first step in the proof is to show that $\tau_{[2]}^{-3}P_0=S_{3}$, which can be easily verified. 
    We then claim that
    \begin{enumerate}
        \item for $3\leq i \leq n-3$, we have $\tau_{[2]}^{-3}S_i=M_{i,i+2}$, and
        \item for $3\leq i\leq n-4$, we have $\tau_{[2]}^{-3}M_{i,i+2}=S_{i+2}$.
    \end{enumerate}
    The procedure to show the claims are the same, so we only verify the first.
    Observe that 
    \[
    \shift{\mathbf{i}_2(S_{i})}=[I_i\to I_{i+1}]
    \] 
    and thus 
    \[
    \tau_{[2]}^-(S_i)=\softTrunc^{\geq -1}(\nu^-\shift{\mathrm{i}_2(S_i)})=[P_i\to P_{i+1}].
    \] 
    Then for the next application of $\tau_{[2]}$, we have 
    \[
    \nu^-\shift{\mathbf{i}_2([P_i \to P_{i+1}])}=\shift{[P_{i-3}\to P_{i-2}]},
    \]
    which after the truncation $\softTrunc^{\geq -1}$ is equal to $\shift{S_{i-2}}$. Hence,
    \[
    \tau_{[2]}^{-2}(S_i)=\tau_{[2]}^-([P_i\to P_{i+1}])=\shift{S_{i-2}}
    \] 
    Taking an injective resolution of $\shift{S_{i-2}}$ yields 
    \[
    \shift{\mathbf{i}_2(\shift{S_{i-2}})}=[I_{i-2}\to I_{i-1}\to I_{i+1}],
    \]
    and applying $\softTrunc^{\geq -1}\circ \nu^-$ to this complex then gives $[S_{i-1}\to P_{i+2}]\cong M_{i,i+2}$. Thus, the first claim is proven.

    Through combining the claims, one obtains the following sequence
    \[
    P_0 \xrsquigarrow{\tau_{[2]}^{-3}} S_{3} \xrsquigarrow{\tau_{[2]}^{-3}} M_{3,5}\xrsquigarrow{\tau_{[2]}^{-3}} S_{5}\xrsquigarrow{\tau_{[2]}^{-3}} M_{5,7}\xrsquigarrow{\tau_{[2]}^{-3}}\cdots 
    \]
    
    If $n$ is odd, the sequence will terminate in $S_{n-2}$ after $n-4$ steps, and then we calculate that 
    \[
    S_{n-2}\,\xrsquigarrow{\tau_{[2]}^-}\,[P_{n-2}\to P_{n-1}]\,\xrsquigarrow{\tau_{[2]}^-}\,\shift{S_{n-4}}\,\xrsquigarrow{\tau_{[2]}^-}\,\shift{S_{n-3}}\,\xrsquigarrow{\tau_{[2]}^-}\,\shift{S_{n-2}}\,\xrsquigarrow{\tau_{[2]}^-}\,\shift{S_{n-1}}
    \]
    Hence, we see that $\tau_{[2]}^{-3(n-3)-2}P_0=\shift{S_{n-1}}=\shift{I_{n-1}}$. 

    If $n$ is even, the sequence terminates in $M_{n-3,n-1}=I_{n-3}$ after $n-4$ steps. Now, $\tau_{[2]}^-(I_{n-3})=\shift{P_{n-3}}=\shift{I_{n-6}}$, and we see that $\tau_{[2]}^{-3(n-4)-1}P_0=\shift{I_{n-6}}$.
\end{proof}

We are left with showing that $\shift{I_{n-6}}$ is connected to $\shift{I_{n-1}}$, before we can conclude that $2\nmod\Lambda(n,4)$ has a component that is both postprojective and preinjective.
\begin{lemma}
    There is a finite path from $\shift{I_{n-6}}$ to $\shift{I_{n-1}}$ in $\AR({2\nmod\Lambda(n,4)})$.
\end{lemma}
\begin{proof}
    Note that 
    \[
        \shift{I_i/\soc(I_i)}=\begin{cases}
            \shift{I_{i+1}}&\text{, if }n-5<i<n-1\\
            \shift{M_{i+1,i+3}}&\text{, if }i\leq n-5.
        \end{cases}
    \]
    and we have irreducible morphisms $\shift{I_i}\to\shift{I_i/\soc(I_i)}=\shift{M_{i+1,i+3}}$ for $i>1$. 
    Hence, we look at $\shift{I_{n-6}/\soc(I_{n-6})}=\shift{M_{n-5,n-3}}$.
    The truncated injective resolution of this complex is $\shift{\mathbf{i}_2(\shift{M_{n-5,n-3}})}=[I_{n-5}\to I_{n-2}\to I_{n-1}]$, and after applying $\softTrunc^{\geq-1}\circ\nu^-$ we get $[M_{n-4,n-2}\to P_{n-1}]\cong S_{n-1}=I_{n-1}$. 
    Finally, we get $\tau_{[2]}^-S_{n-1}=\shift{P_{n-1}}=\shift{I_{n-4}}$, which has a path of irreducible morphisms to $\shift{I_{n-1}}$ as noted above.
\end{proof}

\subsubsection{Case: \texorpdfstring{$2\nmod\Lambda(n,5)$}{2-mod Lambda(n,5)}}
As in the previous case, we begin by looking at a handful calculated examples, see the  AR-quivers of $2\nmod\Lambda(n,5)$ for $n=9,10,11$ in \Cref{fig:ARquiver2modLambdan5}. 
Based on these examples, we decide to take a closer look at the $\tau_{[2]}$-orbit of $P_{7}$, which give rise to a finite path from the simple projective to the simple injective. 
The calculations are similar to those in the previous case, hence we omit most of them.

Note that we still have $\tau_{[m]}^-(\nshift{I_i}{j})=\nshift{P_i}{j-1}$ if $j<m-1$, and now $P_k=I_{k-4}$ for $4\leq k\leq n-1$.

\begin{lemma}
    Let $n\geq 9$, then $\rad P_{7}=\tau^{-7}P_0$ and $\nshift{I_{n-8}/\soc I_{n-8}}{1}=\tau^7\nshift{I_{n-1}}{1}$. 
    That is, $P_{7}$ lie in the postprojective component of $P_0$ and dually $\nshift{I_{n-8}}{1}$ lie in the preinjective component of $\nshift{I_{n-1}}{1}$. 
\end{lemma}
\begin{proof}
    The computations are straight-forward, hence omitted.
\end{proof}

\begin{lemma}
    Let $n\geq 9$, then $\tau^{-k}P_{7}=\nshift{I_{n-8}}{1}$ for $k=6(n-7)$. 
    Moreover, for $0\leq j\leq n-9$, we have $\rad P_{8+j}=\tau_{[2]}^{-(6j+5)}P_{7}$.

    Consequently, the components coincide and $2\nmod(n,5)$ is of finite type.
\end{lemma}
\begin{proof}[Sketch of proof]
    We claim that
    \begin{enumerate}
        \item $\tau^{-2}_{[2]}P_{7}=S_{5}$,
        \item $\tau^{-3}_{[2]}S_i=M_{i-1,i+2}$ for $5\leq i\leq n-3$,
        \item $\tau^{-3}_{[2]}M_{i,i+3}=S_{i+2}$ for $4\leq i\leq n-5$, and
        \item $\tau^{-}_{[2]}I_{n-4}=\shift{I_{n-8}}$
    \end{enumerate}
    This give rise to the sequence
    \[
    \begin{tikzpicture}
        \node (start) at (0,1.5) [] {$P_{7}$};
        \node (Step1) at (0,0) [] {$S_{5}$};
        \node (Step2) at (2,0) [] {$M_{4,7}$};
        \node (Step3) at (4,0) [] {$S_{6}$};
        \node (dots) at (6,0) [] {$\cdots$};
        \node (Stept-1) at (8,0) [] {$S_{n-3}$};
        \node (Stept) at (10,0) [] {$M_{n-4,n-1}$};
        \node (end) at (10,1.5) [] {$\nshift{I_{n-8}}{1}$};
        \draw[draw,->,decorate,decoration={zigzag,amplitude=0.7pt,segment length=1.2mm,pre=lineto,post length=4pt}] (start)--(Step1)node[midway,left]{$\tau_{[2]}^{-2}$};
        \draw[draw,->,decorate,decoration={zigzag,amplitude=0.7pt,segment length=1.2mm,pre=lineto,post length=4pt}] (Step1)--(Step2)node[midway,above]{$\tau_{[2]}^{-3}$};
        \draw[draw,->,decorate,decoration={zigzag,amplitude=0.7pt,segment length=1.2mm,pre=lineto,post length=4pt}] (Step2)--(Step3)node[midway,above]{$\tau_{[2]}^{-3}$};
        \draw[draw,->,decorate,decoration={zigzag,amplitude=0.7pt,segment length=1.2mm,pre=lineto,post length=4pt}] (Step3)--(dots)node[midway,above]{$\tau_{[2]}^{-3}$};
        \draw[draw,->,decorate,decoration={zigzag,amplitude=0.7pt,segment length=1.2mm,pre=lineto,post length=4pt}] (dots)--(Stept-1)node[midway,above]{$\tau_{[2]}^{-3}$};
        \draw[draw,->,decorate,decoration={zigzag,amplitude=0.7pt,segment length=1.2mm,pre=lineto,post length=4pt}] (Stept-1)--(Stept)node[midway,above]{$\tau_{[2]}^{-3}$};
        \draw[draw,->,decorate,decoration={zigzag,amplitude=0.7pt,segment length=1.2mm,pre=lineto,post length=4pt}] (Stept)--(end)node[midway,right]{$\tau_{[4]}^{-}$};
    \end{tikzpicture}
    \]
    wherein the bottom part is of length $2(n-8)+1$, hence $\tau^{-k}P_{7}=\shift{I_{n-8}}$ for $k=6(n-7)$.
\end{proof}

\subsubsection{Case: \texorpdfstring{$4\nmod\Lambda(n,3)$}{4-mod Lambda(n,3)}}
As in the preceding two cases we start by looking at a few calculated examples. 
See \Cref{fig:ARquiver4modLambdan3} for the AR-quivers of $4\nmod\Lambda(n,3)$ for $n=9,10,11$. 
These suggest that we should consider the $\tau_{[4]}$-orbit of $P_{7}$. 
The calculations for the following observations are of the same flavour as in the preceding two cases, hence we give only a sketch of the proofs.

Note that that if $j<m-1$, then $\tau_{[m]}^-(\nshift{I_i}{j})=\nshift{P_i}{j-1}$, and for $1\leq k\leq n-2$ we have $P_k=I_{k-2}$. 

\begin{lemma}
    Let $n\geq 9$, then $\tau_{[4]}^{-k}P_{7}=\nshift{I_{n-8}}{3}$ for $k=10(n-8)+12$. Moreover, for $0\leq j\leq n-9$, we have $\rad P_{8+j}=\tau_{[4]}^{-(9+10j)}P_{7}$.
\end{lemma}
\begin{proof}[Sketch of proof]
    The proof is broken up into four claims:
    \begin{enumerate}
        \item $\tau_{[4]}^{-4} P_{7}\cong S_{6}$,
        \item For $6\leq i\leq n-2$, $\tau_{[4]}^{-5}S_i \cong M_{i,i+1}$,
        \item For $6\leq i\leq n-3$, $\tau_{[5]}^{-5}M_{i,i+1}\cong S_{i+1}$, and
        \item $\tau_{[4]}^{-3}I_{2}=\tau_{[4]}^{-3}M_{n-2,n-1}\cong \nshift{I_{n-8}}{3}$.
    \end{enumerate}
    After combining these four claims, we then obtain the sequence
    \[
    \begin{tikzpicture}
        \node (start) at (0,1.5) [] {$P_{7}$};
        \node (Step1) at (0,0) [] {$S_{6}$};
        \node (Step2) at (2,0) [] {$M_{6,7}$};
        \node (Step3) at (4,0) [] {$S_{7}$};
        \node (dots) at (6,0) [] {$\cdots$};
        \node (Stept-1) at (8,0) [] {$S_{n-2}$};
        \node (Stept) at (10,0) [] {$M_{n-2,n-1}$};
        \node (end) at (10,1.5) [] {$\nshift{I_{n-8}}{3}$};
        \draw[draw,->,decorate,decoration={zigzag,amplitude=0.7pt,segment length=1.2mm,pre=lineto,post length=4pt}] (start)--(Step1)node[midway,left]{$\tau_{[4]}^{-4}$};
        \draw[draw,->,decorate,decoration={zigzag,amplitude=0.7pt,segment length=1.2mm,pre=lineto,post length=4pt}] (Step1)--(Step2)node[midway,above]{$\tau_{[4]}^{-5}$};
        \draw[draw,->,decorate,decoration={zigzag,amplitude=0.7pt,segment length=1.2mm,pre=lineto,post length=4pt}] (Step2)--(Step3)node[midway,above]{$\tau_{[4]}^{-5}$};
        \draw[draw,->,decorate,decoration={zigzag,amplitude=0.7pt,segment length=1.2mm,pre=lineto,post length=4pt}] (Step3)--(dots)node[midway,above]{$\tau_{[4]}^{-5}$};
        \draw[draw,->,decorate,decoration={zigzag,amplitude=0.7pt,segment length=1.2mm,pre=lineto,post length=4pt}] (dots)--(Stept-1)node[midway,above]{$\tau_{[4]}^{-5}$};
        \draw[draw,->,decorate,decoration={zigzag,amplitude=0.7pt,segment length=1.2mm,pre=lineto,post length=4pt}] (Stept-1)--(Stept)node[midway,above]{$\tau_{[4]}^{-5}$};
        \draw[draw,->,decorate,decoration={zigzag,amplitude=0.7pt,segment length=1.2mm,pre=lineto,post length=4pt}] (Stept)--(end)node[midway,right]{$\tau_{[4]}^{-3}$};
    \end{tikzpicture}
    \]
    where the bottom part is of length $2(n-8)+1$. Hence, $\nshift{I_{n-8}}{3}=\tau_{[4]}^{-k}$ where $k=5\cdot(2(n-8)+1)+7=10(n-8)+12$.
\end{proof}

\begin{lemma}
    Let $n\geq 9$, then $P_{7}$ lie in the postprojective component of $P_0$ and dually $\nshift{I_{n-8}}{3}$ lie in the preinjective component of $\nshift{I_{n-1}}{3}$. 
    Consequently, the components coincide and $4\nmod\Lambda(n,3)$ is of finite type.
\end{lemma}
\begin{proof}[Sketch of proof]
    Through calculations, one can verify that $\tau_{[4]}^{-5}P_0=\rad P_{5}$, $\tau_{[4]}^{-3}P_{5}=\rad P_{6}$ and $\tau_{[4]}^{-4}P_{6}=P_{7}$. Similar calculations give a path from $\nshift{I_{n-8}}{3}$ to $\nshift{I_{n-1}}{3}$.
\end{proof}

\begin{landscape}
\begin{figure}
    \centering
    \begin{subfigure}{.44\linewidth}
        \centering
        \includegraphics[width=\linewidth]{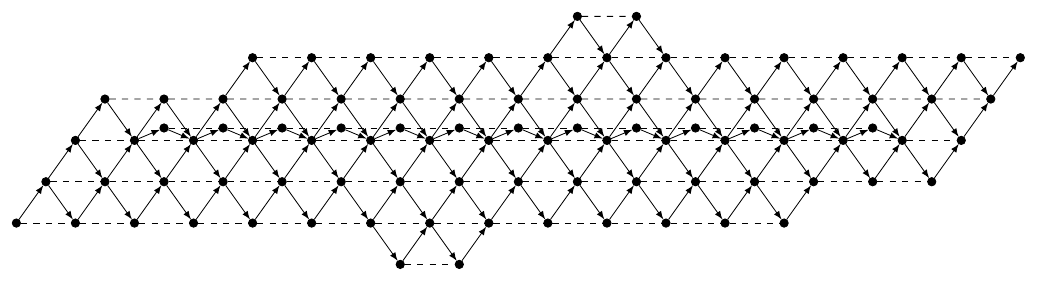}
        \caption{$\AR(2\nmod\Lambda(8,4))$}
        \label{fig:ARquiver2modLambda84}
    \end{subfigure}
    \begin{subfigure}{.55\linewidth}
        \centering
        \includegraphics[width=\linewidth]{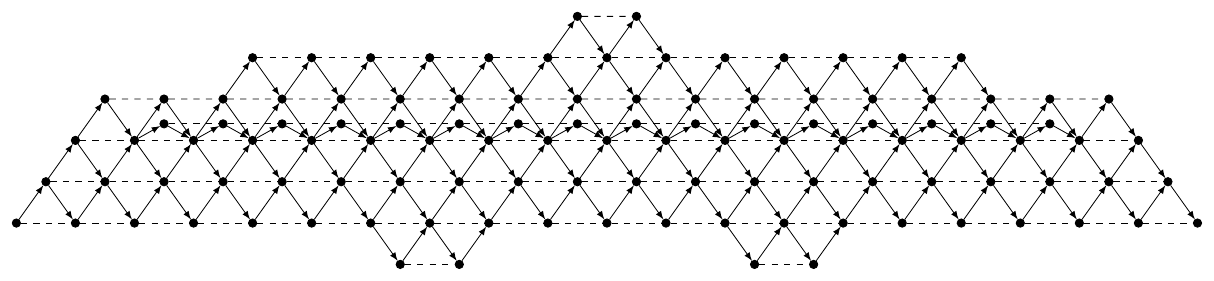}
        \caption{$\AR(2\nmod\Lambda(9,4))$}
        \label{fig:ARquiver2modLambda94}
    \end{subfigure}
    \begin{subfigure}{\linewidth}
        \centering
        \includegraphics[width=.6\linewidth]{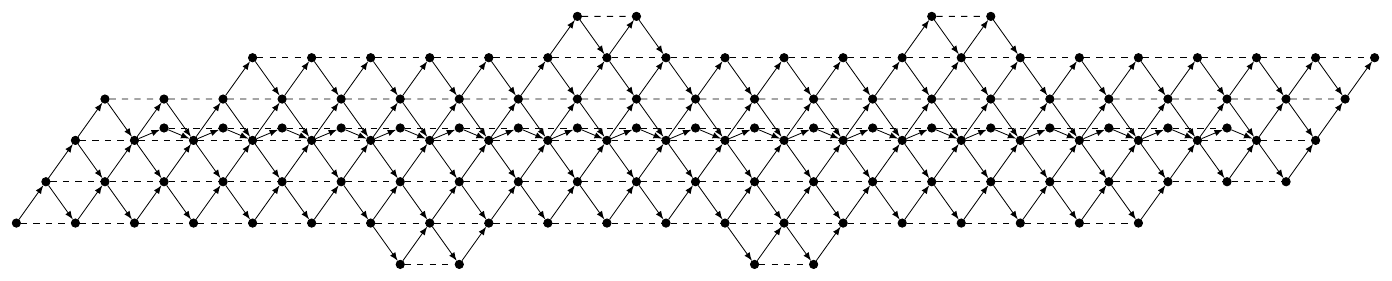}
        \caption{$\AR(2\nmod\Lambda(10,4))$}
        \label{fig:ARquiver2modLambda104}
    \end{subfigure}
    \begin{subfigure}{\linewidth}
        \centering
        \includegraphics[width=.75\linewidth]{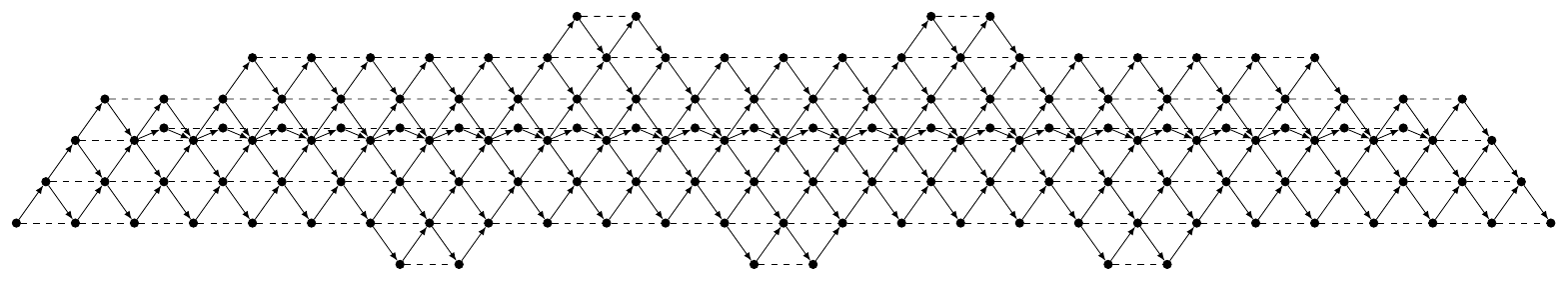}
        \caption{$\AR(2\nmod\Lambda(11,4))$}
        \label{fig:ARquiver2modLambda114}
    \end{subfigure}

    \caption{AR-quivers of $2\nmod\Lambda(n,4)$}
    \label{fig:ARquiver2modLambdan4}
\end{figure}
\end{landscape}

\begin{landscape}
\begin{figure}
    \centering
    \begin{subfigure}{\linewidth}
        \centering
        \includegraphics[width=0.6\linewidth]{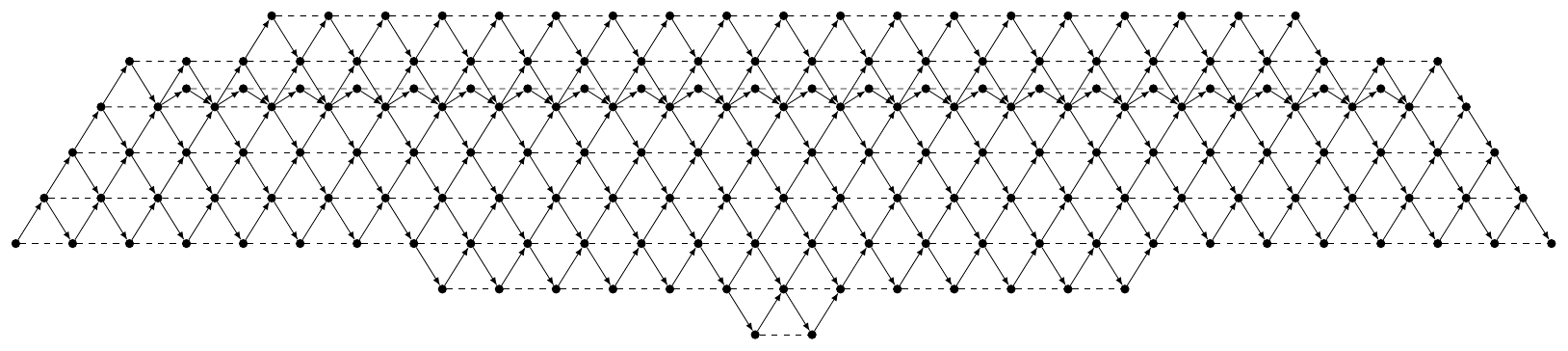}
        \caption{$\AR(2\nmod\Lambda(9,5))$}
        \label{fig:ARquiver2modLambda95}
    \end{subfigure}
    \begin{subfigure}{\linewidth}
        \centering
        \includegraphics[width=0.8\linewidth]{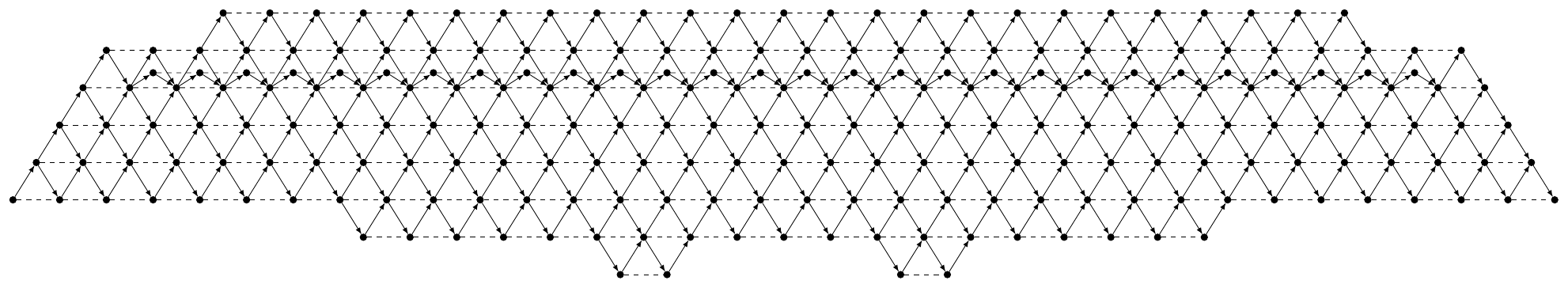}
        \caption{$\AR(2\nmod\Lambda(10,5))$}
        \label{fig:ARquiver2modLambda105}
    \end{subfigure}
    \begin{subfigure}{\linewidth}
        \centering
        \includegraphics[width=0.9\linewidth]{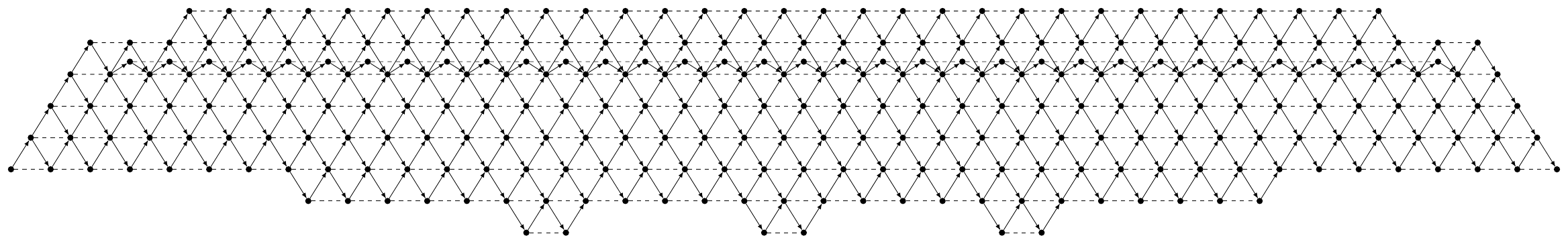}
        \caption{$\AR(5\nmod\Lambda(11,5))$}
        \label{fig:ARquiver2modLambda115}
    \end{subfigure}

    \caption{AR-quivers of $2\nmod\Lambda(n,5)$}
    \label{fig:ARquiver2modLambdan5}
\end{figure}
\end{landscape}

\begin{landscape}
\begin{figure}
    \centering
    \begin{subfigure}{\linewidth}
        \centering
        \includegraphics[width=0.8\linewidth]{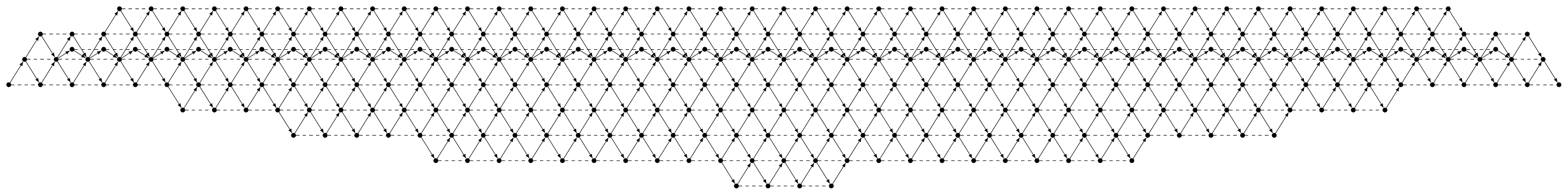}
        \caption{$\AR(4\nmod\Lambda(9,3))$}
        \label{fig:ARquiver4modLambda93}
    \end{subfigure}
    \begin{subfigure}{\linewidth}
        \centering
        \includegraphics[width=0.9\linewidth]{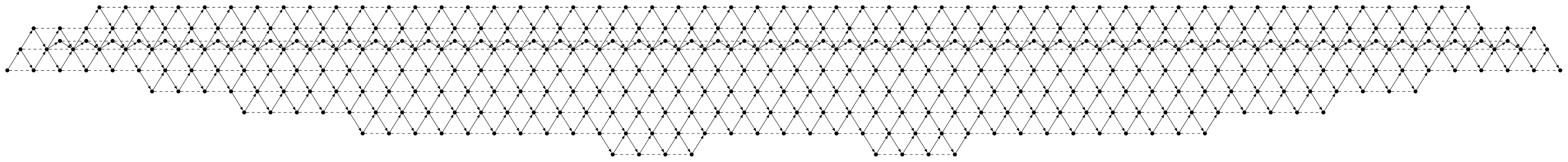}
        \caption{$\AR(4\nmod\Lambda(10,3))$}
        \label{fig:ARquiver4modLambda103}
    \end{subfigure}
    \begin{subfigure}{\linewidth}
        \centering
        \includegraphics[width=\linewidth]{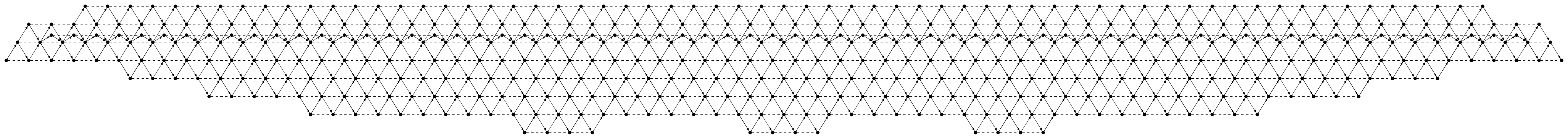}
        \caption{$\AR(4\nmod\Lambda(11,3))$}
        \label{fig:ARquiver4modLambda113}
    \end{subfigure}

    \caption{AR-quivers of $4\nmod\Lambda(n,3)$}
    \label{fig:ARquiver4modLambdan3}
\end{figure}
\end{landscape}

\subsection{The infinite cases}
We are now ready to prove that $\mmod\Lambda(n,l)$ is infinite when $(m,n,l)$ is not included in \Cref{tab:FiniteNakayama}. That is, we want to show:
\begin{proposition}\label{prop:linNakHomOfInfType}
    Let $2<l<n-1$. Then $\mmod\Lambda(n,l)$ is of infinite type when
    \begin{enumerate}
        \item $m>4$ and $n\geq 9$,
        \item $m>2$ and either $(n,l)=(8,4)$ or $n>8$ and $l>3$,
        \item $m>1$, $n>8$ and $l>5$.
    \end{enumerate}
\end{proposition}

Let us first note the following consequence of \Cref{Corol:lemma:restrictionAdjointExtended}.
\begin{lemma}\label{lemma:inheritInfiniteness}
    For $n>3$, let $\Lambda(n,l)$ be given such that $\mmod\Lambda(n,l)$ is of infinite type. Then
    \begin{enumerate}
        \item $\mmod\Lambda(n+1,l)$ is also of infinite type, and \label{lemma:inheritInfinitenessA}
        \item if $l>n/2$, $\mmod\Lambda(n+1,l+1)$ is also of infinite type. \label{lemma:inheritInfinitenessB}
    \end{enumerate}
\end{lemma}
\begin{proof}
    For the claim on $\Lambda_1=\Lambda(n+1,l+1)$ we let $\{e_0,e_1,\ldots,e_{n-1},e_{n}\}$ be a complete set of primitive idempotents, and let $k=\lceil n/2\rceil$. 
    The algebra $e\Lambda_1 e$ for $e=\sum_{i\neq k}e_i$ is isomorphic to $\Lambda(n,l)$, hence the claim follows from \cref{Corol:lemma:restrictionAdjointExtended}.

    The claim on $\Lambda_2=\Lambda(n+1,l)$ is shown similarly, instead using the idempotent $e=\sum_{i=0}^{n-1}e_i$.
\end{proof}

With this result in mind, the proof is reduced to a few base cases. 
For each case the existence of a $\tau_{[m]}$-periodic object is observed, which along with the existence of a postprojective component $\cP$ from \Cref{lemma:postprojective2} imply by \Cref{lemma:finiteComponentIsWholeAR} that $\cP$ can't be finite.

\begin{lemma} \label{lemma:infBaseCases}
We have that
\begin{enumerate}
    \item $3\nmod\Lambda(8,4)$, \label{lemma:3modLambda84Inf}
    \item $5\nmod\Lambda(9,3)$, \label{lemma:5modLambda93Inf}
    \item \label{lemma:3modLambda95Inf}
    $3\nmod\Lambda(9,5)$,
    \item \label{lemma:2modLambda96Inf}
    $2\nmod\Lambda(9,6)$, and
    \item \label{lemma:2modLambda97Inf}
    $2\nmod\Lambda(9,7)$
\end{enumerate}
 is of infinite type.
\end{lemma}
\begin{proof}
    \begin{enumerate}
        \item Consider $\shift{M_{4,5}}$. We have
        \[
            \shift{M_{4,5}}\xleftarrow{\tau_{[3]}} [P_3\to P_5\to P_7] \xleftarrow{\tau_{[2]}} \shift{M_{4,5}}
        \]
        \item Consider $\nshift{P_5}{2}$. We have
        \[
            \nshift{P_5}{2}\xleftarrow{\tau_{[5]}}\nshift{P_3}{3}\xleftarrow{\tau_{[5]}}\nshift{P_1}{4}\xleftarrow{\tau_{[5]}}I_7\xleftarrow{\tau_{[5]}}\shift{I_5}\xleftarrow{\tau_{[5]}}\nshift{P_5}{2}
        \]
        \item Consider $\shift{M_{3,5}}$. We have
        \[
            \shift{M_{3,5}}\xleftarrow{\tau_{[3]}}[P_3\to P_6\to P_8]\xleftarrow{\tau_{[3]}}\shift{M_{3,5}}
        \]
        \item Consider $M^\bullet=[I_0\to P_8]$. We have
        \[
            M^\bullet \xleftarrow{\tau_{[2]}} \shift{M_{1,3}}\xleftarrow{\tau_{[2]}}M_{5,7}\xleftarrow{\tau_{[2]}}M^\bullet
        \]
        \item Consider $M^\bullet =[I_0\to P_8]$, We have
        \[
            M^\bullet\xleftarrow{\tau_{[2]}}\shift{M_{1,2}}\xleftarrow{\tau_{[2]}}I_4\xleftarrow{\tau_{[2]}}\shift{P_4}\xleftarrow{\tau_{[2]}}M_{6,7}\xleftarrow{\tau_{[2]}}M^\bullet
        \]
    \end{enumerate}
\end{proof}

We are now ready to prove the proposition.
\begin{proof}[Proof of \Cref{prop:linNakHomOfInfType}]
\begin{enumerate}
    \item[(3)] Observe by \Cref{lemma:infBaseCases}(\ref{lemma:2modLambda97Inf}) and inductively using \Cref{lemma:inheritInfiniteness}(\ref{lemma:inheritInfinitenessB}) we obtain that $2\nmod\Lambda(n,n-2)$ is of infinite type for each $n\geq 9$. Combining this with \Cref{lemma:infBaseCases}(\ref{lemma:2modLambda96Inf}) and \Cref{lemma:inheritInfiniteness}(\ref{lemma:inheritInfinitenessA}) we then see that $2\nmod(n,l)$ is of infinite type for $n\geq 9$ and $6\leq l <n-1$.

    \item[(2)] The claim for $n\geq9$ and $l\geq 6$ follows from (3). Now, the rest follows from  \Cref{lemma:infBaseCases}(\ref{lemma:3modLambda95Inf}) and \Cref{lemma:infBaseCases}(\ref{lemma:3modLambda84Inf}) together with \Cref{lemma:inheritInfiniteness}(\ref{lemma:inheritInfinitenessA}).

    \item[(1)] The claim for $l>3$ follows from (2). The claim for $l=3$ follows from \Cref{lemma:infBaseCases}(\ref{lemma:5modLambda93Inf}) together with \Cref{lemma:inheritInfiniteness}(\ref{lemma:inheritInfinitenessA}).
\end{enumerate}
\end{proof}

\begin{remark}
    In some sense we may look at the extended module categories of \Cref{lemma:infBaseCases} as being minimally infinite. 
    An interesting observation, albeit not used here, is that for these algebras all projectives lie in one connected postprojective component, see \Cref{fig:postprojectivecomponentsWholePage} for the beginning of their postprojective components. 
    Same as in the classical case \cite[Theorem 2]{HV83}. 

    For the extended module categories further right in Table \ref{tab:FiniteNakayama} we can observe that the same is not true. 
    We can, for an example, consider $3\nmod\Lambda(9,4)$. Here we have $\rad P_8=\tau_{[3]}^3(\rad P_8)$, which tells us there is a finite path of irreducible morphisms from $P_8$ to itself. The postprojective component is directed, hence $P_8$ can't appear in the postprojective component.
\end{remark}

    \begin{figure}
        \begin{subfigure}{\linewidth}
            \centering
            \scalebox{0.4}{\includegraphics{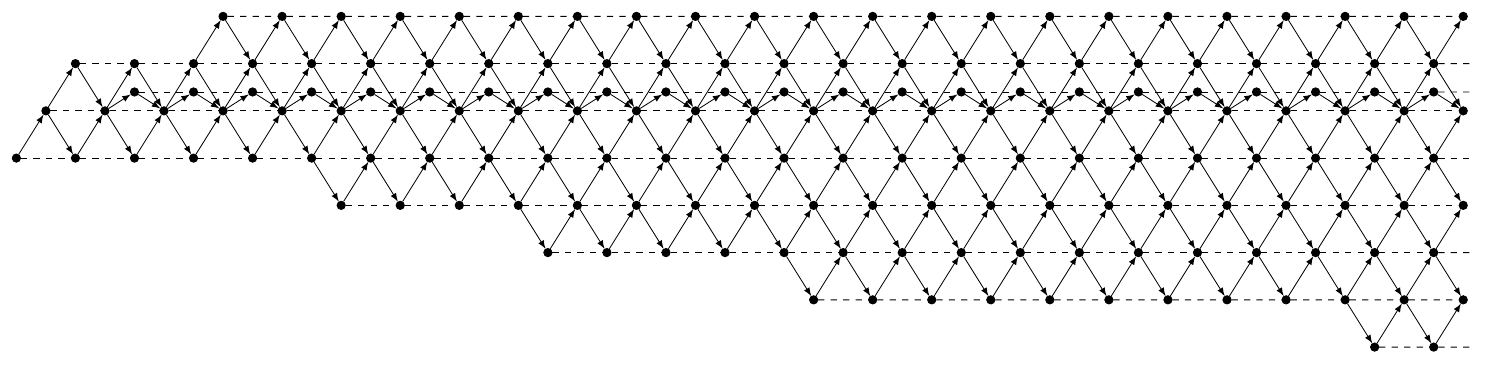}}
            \caption{$\AR(5\nmod\Lambda(9,3))$}
            \label{fig:beginningPostProjective5modLambda93}
        \end{subfigure}

        \begin{subfigure}{0.49\linewidth}
            \centering
            \scalebox{0.4}{\includegraphics{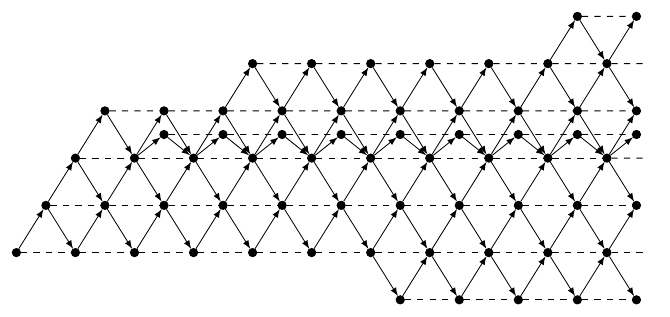}}
            \caption{$\AR(3\nmod\Lambda(8,4))$}
            \label{fig:beginningPostProjective3modLambda84}
        \end{subfigure}
        \begin{subfigure}{0.49\linewidth}
            \centering
            \scalebox{0.4}{\includegraphics{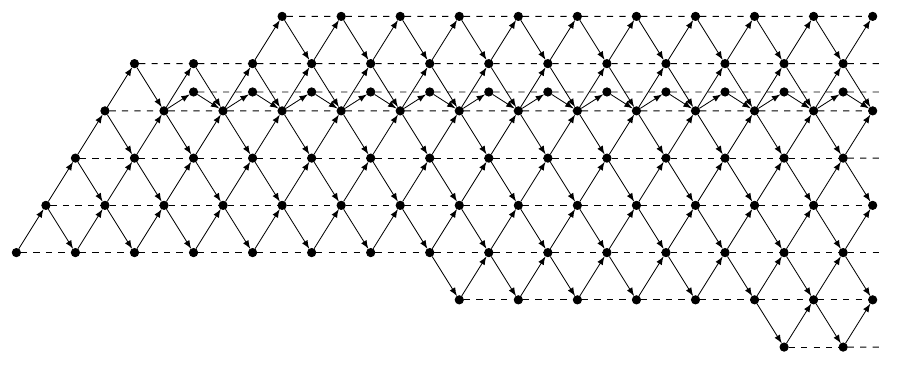}}
            \caption{$\AR(3\nmod\Lambda(9,5))$}
            \label{fig:beginningPostProjective3modLambda95}
        \end{subfigure}
        \begin{subfigure}{0.49\linewidth}
            \centering
            \scalebox{0.4}{\includegraphics{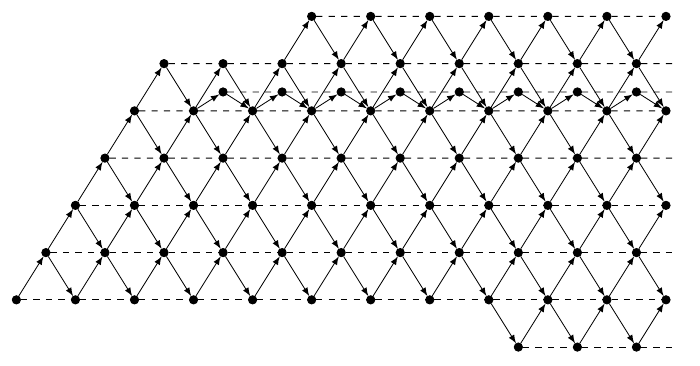}}
            \caption{$\AR(2\nmod\Lambda(9,6))$}
            \label{fig:beginningPostProjective2modLambda96}
        \end{subfigure}
        \begin{subfigure}{0.49\linewidth}
            \centering
            \scalebox{.4}{\includegraphics{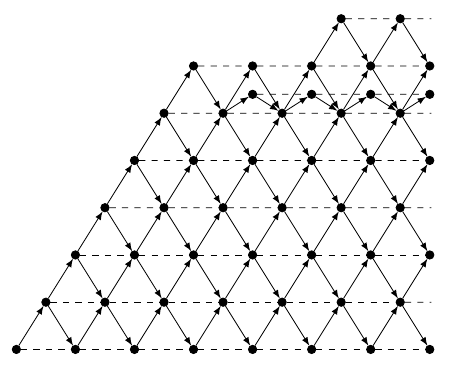}}
            \caption{$\AR(2\nmod\Lambda(9,7))$}
            \label{fig:beginningPostProjective2modLambda97}
        \end{subfigure}
    \caption{The start of the postprojective component in $\AR(\mmod\Lambda(n,l))$ for \emph{minimal} infinite type.}
    \label{fig:postprojectivecomponentsWholePage}
\end{figure}

\subsection{Remark on a certain bound of cohomology}\label{subsubsec:BoundOnCohomology}
Before we move on, let us remark a property of the finite cases which will come in handy later on. 
We can observe in the AR-quivers \Cref{fig:ARquiver2modLambdan4,fig:ARquiver2modLambdan5,fig:ARquiver4modLambdan3} that as $n$ grows, the new projectives lie on relatively short $\tau_{[m]}$-orbits. 
Specifically,
\begin{enumerate}
    \item in $2\nmod\Lambda(n,4)$ we have $\tau_{[2]}^{-}P_i=\nshift{I_{i-6}}{1}$ for $6\leq i\leq n-1$,
    \item in $2\nmod\Lambda(n,5)$ we have $\tau_{[2]}^-P_i=\nshift{I_{i-8}}{1}$ for $8\leq i \leq n-1$, and
    \item in $4\nmod\Lambda(n,3)$ we have $\tau_{[4]}^{-3}P_i=\nshift{I_{i-8}}{2}$ for $8\leq i \leq n-1$.
\end{enumerate}
Combining this with \Cref{lemma:pathFromProjToInj} we can deduce that there exists an upper bound on the distance $|v-w|$ for which both $(\DimVec(X^\bullet)_\ast)_v$ and $(\DimVec(X^\bullet)_\ast)_w$ can be non-zero. 
Moreover, this bound can be chosen independently of $n$.

In order to make this slightly more apparent, consider $2\nmod\Lambda(n,4)$ for $n$ sufficiently large. If we look at a part in the middle of the AR-quiver, we can see from the directedness of the quiver that we have a part which do not admit a path from $P_i$ and a part which do not admit a path to $\shift{I_{i-6}}$, and these two parts are overlapping in only a few nodes. 
\[
\includegraphics[width=.75\linewidth]{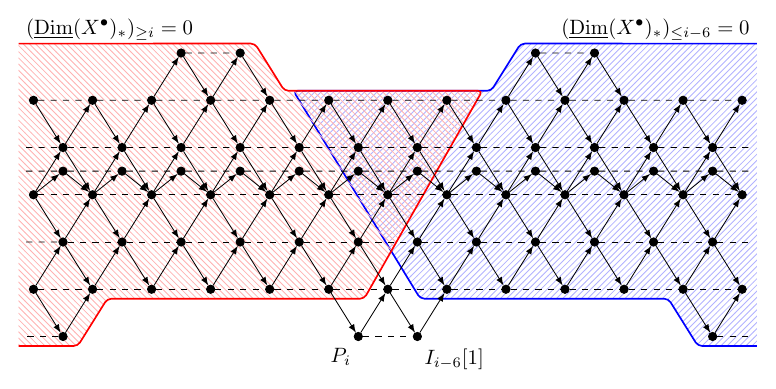}
\]

The same kind of feature holds for $3\nmod\Lambda(n,3)$ and $2\nmod\Lambda(n,2)$. In fact, we can easily calculate that
\begin{enumerate}
    \item for $3\nmod\Lambda(n,3)$, we have $\tau_{[3]}^{-2}P_i=\nshift{I_{i-6}}{2}$ for $6\leq i\leq n-1$, and
    \item for $2\nmod\Lambda(n,3)$, we have $\tau_{[2]}^-P_i=\shift{I_{i-4}}$ for $4\leq i\leq n-1$.
\end{enumerate}
And, of course, for $\mod\Lambda(n,l)$ we have $\tau^0 P_i=I_{i-l+1}$ for $l\leq i\leq n-1$. Let us therefore consider the remaining cases when $\mmod\Lambda$ is finite.

For $\mmod\Lambda(n,2)$ we can see that if $1\leq i\leq n-1$ we have $\tau_{[m]}^-\nshift{P_i}{j}=\nshift{P_{i-i}}{j+1}=\nshift{I_{i-2}}{j+1}$ for $-(m-2)\leq j\leq 0$.
Hence, the projectives $P_i$ for $m\leq i\leq n-1$ lie on a $\tau_{[m]}$-orbit of length $m$:
\[
    \begin{tikzpicture}
        \node (Step1) at (0,0) [] {$P_i$};
        \node (Step2) at (2,0) [] {$\shift{P_{i-1}}$};
        \node (dots) at (4,0) [] {$\cdots$};
        \node (Stept) at (7.8,0) [] {$\nshift{P_{i-(m-1)}}{m-1}=\nshift{I_{i-m}}{m-1}$};
        \draw[draw,->,decorate,decoration={zigzag,amplitude=0.7pt,segment length=1.2mm,pre=lineto,post length=4pt}] (Step1)--(Step2)node[midway,above]{$\tau_{[m]}^{-}$};
        \draw[draw,->,decorate,decoration={zigzag,amplitude=0.7pt,segment length=1.2mm,pre=lineto,post length=4pt}] (Step2)--(dots)node[midway,above]{$\tau_{[m]}^{-}$};
        \draw[draw,->,decorate,decoration={zigzag,amplitude=0.7pt,segment length=1.2mm,pre=lineto,post length=4pt}] (dots)--(Stept)node[midway,above]{$\tau_{[m]}^{-}$};
    \end{tikzpicture}
\]
We include the AR-quivers for a few $\mmod\Lambda(n,2)$ to illustrate, see \Cref{fig:ARquivermmodLambdan2}.

The last remaining case is that of $\mmod\Lambda(n,n-1)$ which is significantly different from the others as the relation length and number of vertices are not independent.
We will in fact always have a non-zero morphism from $I_0\to P_{n-1}$, and thus the bound would have to be $n$. 

We end this section by summarizing the last observations as follows.
\begin{lemma}\label{lemma:boundOnCohomology}
    Let
    \begin{enumerate}
        \item $l=2$ and $1\leq m<\infty$,
        \item $l=3$ and $1\leq m\leq 4$,
        \item $l=4$ and $1\leq m\leq 2$,
        \item $l=5$ and $1\leq m\leq 2$, or
        \item $l<n$ and $m=1$.
    \end{enumerate}
    For all $n$ and $X^\bullet\in \mmod\Lambda(n,l)$ we have 
    \[
    (\DimVec(X^\bullet)_i)_v\neq 0\quad \text{ and }\quad (\DimVec(X^\bullet)_j)_w\neq 0,
    \] 
    for some $i,j\in[-(m-1),0]$ only if $|v-w|\leq m\cdot(l-1)+1$. 

\end{lemma}

\begin{figure}
    \centering
    \begin{subfigure}{.49\linewidth}
        \centering
        \includegraphics[width=0.5\linewidth]{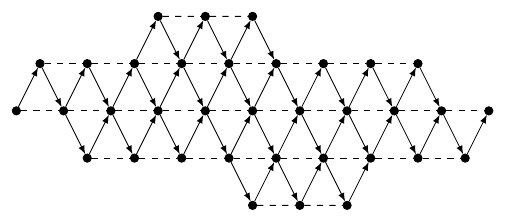}
        \caption{$\AR(3\nmod\Lambda(5,2))$}
        \label{fig:ARquiver3modLambda52}
    \end{subfigure}
    \begin{subfigure}{.49\linewidth}
        \centering
        \includegraphics[width=0.6\linewidth]{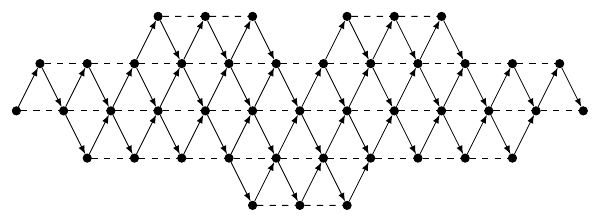}
        \caption{$\AR(3\nmod\Lambda(6,2))$}
        \label{fig:ARquiver3modLambda62}
    \end{subfigure}
    \begin{subfigure}{.49\linewidth}
        \centering
        \includegraphics[width=0.75\linewidth]{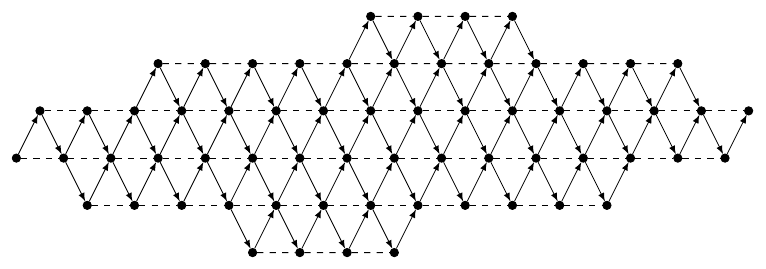}
        \caption{$\AR(4\nmod\Lambda(6,2))$}
        \label{fig:ARquiver4modLambda62}
    \end{subfigure}
    \begin{subfigure}{.49\linewidth}
        \centering
        \includegraphics[width=0.85\linewidth]{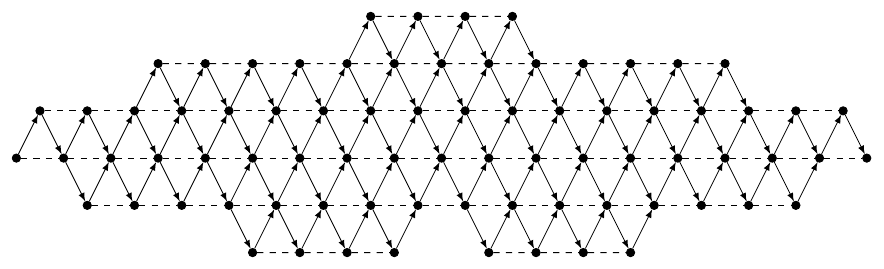}
        \caption{$\AR(4\nmod\Lambda(7,2))$}
        \label{fig:ARquiver4modLambda72}
    \end{subfigure}
    \caption{AR-quivers of $\mmod\Lambda(n,2)$}
    \label{fig:ARquivermmodLambdan2}
\end{figure}

\section{Transfer to the cyclic case}
Having developed an understanding of the extended modules of linear Nakayama algebras with homogeneous relations, we seek now to extend this to the cyclic case. 
This will be done through the methods of covering.

\subsubsection*{Idea of covering}
Covering theory has its roots in the world of topology and was introduced as a tool for representation theorists by \cite{Rie80}, \cite{BG81} and \cite{Gab81}.
The main idea is to find an often infinite (covering) quiver which maps nicely down onto the quiver we want to study. 

The constructions produce well-behaved functors which can be used to push known theory of the covering quiver down onto theory of the quiver in question. 
Classically, the Auslander-Reiten theory of modules is often what one wants to push down.

\subsubsection*{Overview of the section}
We start off by some preliminary theory about covering-constructions and known properties.
The infinite quiver \ref{eq:InfNakayamaQuiver} is introduced as the covering of interest for \ref{eq:CyclicNakayamaQuiver}. 
Considering $\bA_{d}$ as a subquiver of $\bA_{\infty}$, it is shown that for $d$ sufficiently large there are modules supported over $\mod\Lambda(d,l)$ that are sent to modules over a cyclic Nakayama algebra, $\Gamma$, in a manner compatible with the Nakayama functors, see \Cref{lemma:reduceMorphismToFinite,lemma:nakayamaCommuteWithPushDown}.

The bounds of \Cref{lemma:boundOnCohomology} are used to show that complexes suitably far into $\AR(\mod\Lambda(d,l))$ behaves the same regardless of whether they are seen as complexes over the covering algebra or $\Lambda(d,l)$, see \Cref{lemma:homOfComplexesCovering}. 
This is then used to show that AR-triangles with terms far enough into $\AR(\mmod\Lambda(n,l))$, are sent to AR-triangles of $\mmod\Gamma$, see \Cref{prop:ARtrianglesSentToARtriangles}. We can therefore conclude on the finiteness of $\AR(\mmod\Gamma)$, see \Cref{theorem:finitenessOfCyclicNakayama}.

\subsection{Preliminaries on covering}
We follow \cite{Gre05}. 
Let $Q$ be a quiver; a surjective quiver morphism $\hat{\pi}\colon \widehat{Q}\to Q$ is a \emph{regular covering} if there exists a group $G$ and a set-map $W\colon Q_1\to G$ such that\footnote{$W$ extends to paths through the group operation $-\ast -\colon G\times G\to G$; $W(\alpha_t\alpha_{t-1}\cdots \alpha_1)=W(\alpha_t)\ast W(\alpha_{t-1})\ast \cdots \ast W(\alpha_1)$.} 
\begin{itemize}
    \item $\widehat{Q}_0=Q_0\times G$ (denoting $(v,g)$ as $v_g$),
    \item $\widehat{Q}_1=Q_1\times G$ (denoting $(\alpha,g)$ as $\alpha_g$) where for $\alpha\colon v\to w\in Q_1$ we have $\alpha_g\colon v_g\to w_{[gW(\alpha)]}$, and
    \item $\hat{\pi}(v_g)=v$ and $\hat{\pi}(\alpha_g)=\alpha$.
\end{itemize}
Let $(Q,\rho)$ be a bound quiver and $q\in \rho$ a relation on $\K Q$. A relation $\hat{q}=\sum_{i=1}^m c_i \hat{p}_i$ on $\K \widehat{Q}$ is \emph{a lifting of} $q$ if $q=\sum_{i=1}^{m}c_i \pi(\hat{p}_i)$. 
Now, a surjective morphism of bound quivers $\pi\colon (\widehat{Q},\hat{\rho})\to (Q,\rho)$ is a \emph{covering}, if 
\begin{enumerate}
    \item ${\pi}\colon \widehat{Q}\to Q$ is a regular covering,
    \item for all relations $\sum_{i=1}^m c_i p_i$ on $\K Q$, we have $W(p_i)=W(p_j)$ for $1\leq i,j\leq m$, and
    \item $\hat{I}=\langle \hat{\rho}\rangle$ is generated by all the possible liftings of relations in $\rho$.
\end{enumerate}

\begin{remark}
A covering may also be realized through grading the path algebra $\Lambda=\K Q/\langle \rho\rangle$ using the set function $W\colon Q_1\to G$, see \cite{Gre83}. 
\end{remark}

Let us fix a covering $\pi\colon (\widehat{Q},\hat{\rho})\to (Q,\rho)$, and denote by $\Gamma=\K Q/\langle\rho\rangle$ and $\widehat{\Gamma}=\K \widehat{Q}(\langle\hat{\rho}\rangle)$ the associated path algebras. Moreover, we will denote $\rep(\widehat{Q},\hat{\rho})$ by $\mod\widehat{\Gamma}$. 

There is an exact additive functor $F_\pi\colon \mod\widehat{\Gamma}\to \mod\Gamma$, called the \emph{push-down functor} \cite[3.2]{BG81}, which is given on representations $M\in\mod\widehat{\Gamma}$ by
\[
F_\pi(M)_v=\bigoplus_{x\in\pi^{-1}(v)}M_{x}\ \text{ for }v\in Q_0
\]
Let $M$ be a representation in $\mod\widehat{\Gamma}$ and $h$ a group element of $G$. We can define another representation $\gGroupShift{M}{h} \in \mod\widehat{\Gamma}$, given by $\gGroupShift{M}{h}_{v_g}=M_{v_{h^{-1}g}}$, which we call the $h$\emph{-translate of $M$}. We can see that $\gGroupShift{-}{h}$ define an automorphism on $\mod\widehat{\Gamma}$. Observe that $F_\pi(M)=F_\pi(\gGroupShift{M}{h})$ for all $h\in G$.

The push-down functor is exact, hence we may extend it to a functor of the derived categories $F_\pi^{\D}\colon \D^b(\mod\widehat{\Gamma})\to\D^b(\mod\Gamma)$. Also, the translation $\gGroupShift{-}{h}$ extends naturally to an automorphism of the derived category. In the following proposition, we have collected results on coverings which will shortly come in handy. Note that both $F_\pi^{\D}$ and $\gGroupShift{-}{h}$ also restricts down to $\mmod$.
\begin{proposition}\label{prop:pushdownProperties}
    Let $\pi\colon (\widehat{Q},\hat{\rho})\to (Q,\rho)$ be a covering given by a torsion-free group $G$, and let $F_\pi\colon \mod\widehat{\Gamma}\to \mod\Gamma$ be the additive and exact push-down functor.
    \begin{enumerate}
        \item $F_\pi$ is faithful, and $M\in \mod\widehat{\Gamma}$ is indecomposable if and only if $F_\pi(M)\in \mod\Gamma$ is indecomposable. Moreover, if $M\in \mod\widehat{\Gamma}$ is indecomposable and $F_\pi(M)\cong F_\pi(N)$, then $N\cong\gGroupShift{M}{h}$ for some $h\in G$.  \label{prop:pushdownProperties1}
        \item $P\in \mod\widehat{\Gamma}$ is projective if and only if $F_\pi(P)\in\mod\Gamma$ is projective. \label{prop:pushdownProperties2}
        \item $I\in \mod\widehat{\Gamma}$ is injective if and only if $F_\pi(I)\in\mod\Gamma$ is injective.\label{prop:pushdownProperties3}
        \item For $M,N\in \mod\widehat{\Gamma}$ we have \label{prop:pushdownProperties4}
            \[
            \begin{split}
                \Hom_{\Gamma}(F_\pi(M),F_\pi(N))&\cong \bigoplus_{h\in G}\Hom_{\widehat{\Gamma}}(\gGroupShift{M}{i},N)\\
                &\cong\bigoplus_{h\in G}\Hom_{\widehat{\Gamma}}(M,\gGroupShift{N}{i})
            \end{split}
            \]
        \item $F^{\D}_\pi\colon \Db(\mod\widehat{\Gamma})\to \Db(\mod\Gamma)$ is faithful, and $X^\bullet\in\Db(\mod\widehat{\Gamma})$ is indecomposable if and only if $F^{\D}_\pi(X^\bullet)$ is indecomposable. Moreover, if $X^\bullet\in D^b(\mod\widehat{\Gamma})$ is indecomposable and $F_\pi^{\D}(X^\bullet)\cong F_\pi^{\D}(Y^\bullet)$, then $X^\bullet\cong \gGroupShift{Y^\bullet}{h}$ for some $h\in G$. \label{prop:pushdownProperties5}
        \item For $X^\bullet,Y^\bullet\in \Db(\mod\widehat{\Gamma})$ we have \label{prop:pushdownProperties6}
            \[
            \begin{split}
                \Hom_{\Db(\mod\Gamma)}(F_\pi(X^\bullet),F_\pi(Y^\bullet))&\cong \bigoplus_{h\in G}\Hom_{\Db(\mod\widehat{\Gamma})}(\gGroupShift{X^\bullet}{i},Y^\bullet)\\
                &\cong\bigoplus_{h\in G}\Hom_{\Db(\mod\widehat{\Gamma})}(X^\bullet,\gGroupShift{Y^\bullet}{i})
            \end{split}
            \]
    \end{enumerate}
\end{proposition}
\begin{proof}
    See \cite[Lemmas 3.2 and 3.5]{Gab81}, \cite[Propositions 3.2]{BG81}, \cite[Theorem 3.2 and Proposition 3.4]{GG82} and \cite[Lemmas 3.7 and A.2, Theorem A.3]{GH11}. We note that the second part of (\ref{prop:pushdownProperties5}) follows by an adaption of \cite[Lemma 3.5]{Gab81}.
\end{proof}

\subsection{Cyclic Nakayama algebras}
Let us now consider cyclic Nakayama algebras with homogeneous relations. That is, the bound quivers on the form $($\ref{eq:CyclicNakayamaQuiver}$,\rho_{\Delta,l})$, where \ref{eq:CyclicNakayamaQuiver} is the quiver
\[
\begin{tikzpicture}[scale=.8,baseline=(current  bounding  box.center)]
    \node (n-1) at (90:1cm) [nodeDots] {};
    \node (n-1-label) at (90:1.1cm) [scale=.7,above]{$n-1$};
    
    \node (n-2) at (30:1cm) [nodeDots] {};
    \node (n-2-label) at (30:1.1cm) [scale=.7,right] {$n-2$};

    \node (cdot1) at ($(-30:1)+(0,.08)$) [nodeCDots,scale=.7] {};
    \node (cdot2) at ($(-30:1)+(-.04,0)$) [nodeCDots,scale=.7] {};
    \node (cdot3) at ($(-30:1)+(-.08,-.08)$) [nodeCDots,scale=.7] {};

    \node (2) at (-90:1cm) [nodeDots] {};
    \node (2-label) at (-90:1.1cm) [scale=.7,below] {$2$};

    \node (1) at (-150:1cm) [nodeDots] {};
    \node (1-label) at (-150:1.1cm) [scale=.7,left] {$1$};

    \node (0) at (150:1cm) [nodeDots] {};
    \node (0-label) at (150:1.1cm) [scale=.7,left] {$0$};

    \draw[-latex] ([xshift=2.5pt,yshift=-1.5pt]n-1.center)--([xshift=-2.5pt,yshift=1pt]n-2.center)node[midway,anchor=south west,scale=.7]{$\delta_{n-1}$};
    \draw[-latex] ([yshift=-1.5pt]n-2.center)--([yshift=1.5pt]cdot1.center)node[midway,right,scale=.7]{$\delta_{n-2}$};
    \draw[-latex] ([xshift=-2.5pt,yshift=-1.5pt]cdot3.center)--([xshift=2.5pt,yshift=1pt]2.center)node[midway,anchor=north west,scale=.7]{$\delta_{3}$};
    \draw[-latex] ([xshift=-2.5pt,yshift=1.5pt]2.center)--([xshift=2.5pt,yshift=-1pt]1.center)node[midway,anchor=north east,scale=.7]{$\delta_{2}$};
    \draw[-latex] ([yshift=3pt]1.center)--([yshift=-2.5pt]0.center)node[midway,anchor=east,scale=.7]{$\delta_{1}$};
    \draw[-latex] ([xshift=2.5pt,yshift=1.5pt]0.center)--([xshift=-2.5pt,yshift=-1pt]n-1.center)node[midway,anchor=south east,scale=.7]{$\delta_{0}$};
\end{tikzpicture}
\]
indexed by $\bZ/n\bZ$, and $\rho_{\Delta,l}$ are all paths of length $l$. Let $W\colon (\DCyc{n})_1\to \bZ$ be defined by $W(\delta_0)=1$ and $W(\delta_i)=0$ for $i\neq 0$. 
This gives a covering which can be identified with $($\ref{eq:InfNakayamaQuiver}$,\hat{\rho})$, where \ref{eq:InfNakayamaQuiver} is the infinite quiver
\begin{equation}\label{eq:InfNakayamaQuiver}\tag{$\bA_{\infty}$}
\begin{tikzpicture}[xscale=.7,tips=proper,baseline=(current  bounding  box.center)]
    \node (cdot1Left) at (-.2,0) [nodeCDots] {};
    \node (cdot2Left) at (0,0) [nodeCDots] {};
    \node (cdot3Left) at (.2,0) [nodeCDots] {};

    \node (2) at (2,0) [nodeDots] {};
    \node (2-label) at (2,-.1) [below,scale=.7] {$2$};
    \node (1) at (4,0) [nodeDots] {};
    \node (1-label) at (4,-.1) [below,scale=.7] {$1$};
    \node (0) at (6,0) [nodeDots] {};
    \node (0-label) at (6,-.1) [below,scale=.7] {$0$};
    \node (-1) at (8,0) [nodeDots] {};
    \node (-1-label) at (8,-.1) [below,scale=.7] {$-1$};
    \node (-2) at (10,0) [nodeDots] {};
    \node (-2-label) at (10,-.1) [below,scale=.7] {$-2$};

    \node (cdot1Right) at (11.8,0) [nodeCDots] {};
    \node (cdot2Right) at (12,0) [nodeCDots] {};
    \node (cdot3Right) at (12.2,0) [nodeCDots] {};

    \draw[-latex] ([xshift=.2cm]cdot3Left.center)edge([xshift=-.2cm]2.center);
    \draw[-latex] ([xshift=.2cm]2.center)edge([xshift=-.2cm]1.center);
    \draw[-latex] ([xshift=.2cm]1.center)edge([xshift=-.2cm]0.center);
    \draw[-latex] ([xshift=.2cm]0.center)edge([xshift=-.2cm]-1.center);
    \draw[-latex] ([xshift=.2cm]-1.center)edge([xshift=-.2cm]-2.center);
    \draw[-latex] ([xshift=.2cm]-2.center)edge([xshift=-.2cm]cdot1Right.center);
\end{tikzpicture}
\end{equation}
and $\hat{\rho}$ are all paths of length $l$. 
Under this identification, we see that the covering morphism is induced by the canonical projection $\bZ\to \bZ/n\bZ$. The translation $\gGroupShift{M}{i}$ is given by $\gGroupShift{M}{i}_j=M_{j-ni}$. 
We continue with the notation from the last subsection; let $\Gamma=\K \DCyc{n}/\langle\rho_{\Delta,l}\rangle$ and $\widehat{\Gamma}=\K \bA_\infty/\langle \hat{\rho}\rangle$.

Fix $n,d,l>0$. We can find an injection of the linear quiver $\bA_d$ into $\bA_\infty$ by $x\mapsto x$. This injection give rise to an embedding $\iota\colon \mod\Lambda(d,l)\hookrightarrow \mod\widehat{\Gamma}$ which in conjunction with the push-down functor $F_\pi$, plays an essential part in the following. Specifically, we will be considering the composition
\[
    \begin{tikzpicture}
        \node (A) at (0,0){$F_\Lambda\colon\vphantom{\widehat{\Gamma}_\Lambda}$};
        \node[right=.2cm of A](B) {$\mod\Lambda(d,l)\vphantom{\widehat{\Gamma}_\Lambda}$};
        \node[right=1cm of B] (C) {$\mod\widehat{\Gamma}\vphantom{\widehat{\Gamma}_\Lambda}$};
        \node[right=1cm of C] (D) {$\mod\Gamma\vphantom{\widehat{\Gamma}_\Lambda}$};
        \draw[>->] ([yshift=-1pt]B.east)--([yshift=-1pt]C.west)node[midway,above,scale=.8]{$\iota$};
        \draw[->] ([yshift=-1pt]C.east)--([yshift=-1pt]D.west)node[midway,above,scale=.8]{$F_\pi$};
    \end{tikzpicture}
\]
and the induced functor on the extended module category, $F_\Lambda^m\colon \mmod\Lambda(n,l)\to \mmod\Gamma$. Under certain condition we will see that this functor preserves AR-triangles. 

\subsubsection{AR-triangles} First, we need some notation and observations.
Let $i\in\bZ$, then we denote by $P_i^{\widehat{\Gamma}}$ ($I_i^{\widehat{\Gamma}},S_i^{\widehat{\Gamma}}$) the indecomposable projective (injective, simple) of $\mod\widehat{\Gamma}$ in vertex $i$. 
Similarly, for $i\in [0,d-1]$ ($j\in\bZ/n\bZ$) we define $P_i^\Lambda,I_i^\Lambda$ and $S_i^\Lambda$ ($P_j^\Gamma,I_j^\Gamma$ and $S_j^\Gamma)$. 
We have the following easy observation
\begin{observation}
    Let $i\in [0,d-1]$. Then
    \begin{enumerate}
        \item $\iota(S_i^\Lambda)=S_i^{\widehat{\Gamma}}$,
        \item $\iota(P_i^\Lambda)=P_i^{\widehat{\Gamma}}$ if and only if $i\in [l,d]$, and
        \item $\iota(I_i^\Lambda)=I_i^{\widehat{\Gamma}}$ if and only if $j\in [0,d-l]$.
    \end{enumerate}
\end{observation}

Our first step towards understanding how $F_\Lambda^m$ affects AR-triangles is to understand how $F_\Lambda$ behaves with respect to the Nakayama functor, and thus what it does to morphisms between projectives.
\begin{lemma}\label{lemma:reduceMorphismToFinite}
    Let $d>n+2l$ and $i\in[l,d-l-1]$. 
    \[
    \Hom_{\widehat{\Gamma}}(\iota(P_i^\Lambda),P_k^{\widehat{\Gamma}})=\begin{cases}
        \Hom_\Lambda(P_i^\Lambda,P_k^\Lambda)&,\text{ if }k\in [0,d-1]\\
        0&,\text{ otherwise.}
    \end{cases}
    \]
\end{lemma}
\begin{proof}
    Recall that, for a path algebra, there is a non-zero morphism from a projective at vertex $x$ to a module $M$ if and only if the simple at $x$ is in the composition series of $M$.

    Thus, we have $\Hom_\Lambda(P_i^\Lambda,P_j^\Lambda)\neq 0$ if and only if $i\in [j-l+1,l]$. 
    Hence, if $i\in[l,d-l-1]$ then $j\geq i\geq l$, and $\iota(P_j^\Lambda)=P_j^{\widehat{\Gamma}}$.
    Likewise, if $\Hom_{\widehat{\Gamma}}(\iota(P_i^\Lambda),P_k^{\widehat{\Gamma}})\neq 0$, then $i\in [k-l+1,k]$. Therefore, $k\leq i+l-1\leq d$ and $P_k^{\widehat{\Gamma}}=\iota(P_k^\Lambda)$.
    This finishes the proof.
\end{proof}

Let $j\in \bZ$ and $Q=\bigoplus_{i=j}^{n+j-1}P_i^{\widehat{\Gamma}}$, then we can observe that $F_\pi(Q)=\Gamma$. Also, for each $i\in[l,d-l-1]$, $\iota(P_i^\Lambda)$ appears as a direct summand of $\gGroupShift{Q}{k}$ for some $k\in\bZ$.
Combining this with our previous lemma and \Cref{prop:pushdownProperties}(\ref{prop:pushdownProperties4}), gives us for $i\in[l,d-l-1]$ that
\[
\begin{split}
    \Hom_{\Gamma}(F_\Lambda(P_i^\Lambda),\Gamma)&\cong\bigoplus_{k\in \bZ}\Hom_{\widehat{\Gamma}}(\iota(P_i^\Lambda),\gGroupShift{Q}{k})\\
    &\cong \bigoplus_{j\in\bZ}\Hom_{\widehat{\Gamma}}(\iota(P_i^\Lambda),P_j^{\widehat{\Gamma}})\\
    &\cong \bigoplus_{j\in[0,d-1]}\Hom_\Lambda(P_i^\Lambda,P_j^\Lambda)\\
    &\cong \Hom_\Lambda(P_i^\Lambda,\Lambda)
\end{split}
\]
The right-hand side of this isomorphism is a module over $\Lambda^{\mathrm{op}}$, 
so we may apply the corresponding covering functor $F_\Lambda^{\mathrm{op}}\colon \mod\Lambda^{\mathrm{op}}\to \mod\Gamma^{\mathrm{op}}$ to it. 
We see that $F_\Lambda^{op}\Hom_\Lambda(P_i^\Lambda,\Lambda)\cong \Hom_\Gamma(F_\Lambda(P_i^\Lambda),\Gamma)$ as modules over $\Gamma^{op}$.
Now, observe also that $F_\Lambda^{op}\circ D\simeq D\circ F\colon \mod\Lambda\to\mod\Gamma^{op}$ and $F_\Lambda\circ D\simeq D\circ F^{op}\colon \mod\Lambda^{op}\to\mod\Gamma$.
We summize that
\begin{lemma}\label{lemma:nakayamaCommuteWithPushDown}
    Let $d>n+2l$ and let $\mathcal{Q}\coloneqq\add\{P_i^\Lambda\,|\,i\in [l,d-l-1]\}$. Then we have a natural isomorphism
    $F_\Lambda\nu_\Lambda \cong \nu_\Gamma F_\Lambda$
    on $\mathcal{Q}$, where $\nu_\Lambda$ and $\nu_\Gamma$ are the Nakayama functors of the respective algebras. 

    Dually, let $\mathcal{I}\coloneqq\add\{I_i^\Lambda\,|\,i\in [l,d-l-1]\}$. Then we have a natural isomorphism
    $F_\Lambda\nu_\Lambda^-\cong \nu_\Gamma^-F_\Lambda$
    on $\mathcal{I}$, where $\nu_\Lambda^-$ and $\nu_\Gamma^-$ are the inverse nakayama functors.
\end{lemma}

The next step towards considering AR-triangles under $F_\Lambda^m$, is to see how the functor acts on AR-translations. 
Let $\mathcal{N}$ be the subcategory of objects $Z^\bullet$ in $\mmod\Lambda$ whose projective presentation $\projres_m Z^\bullet$ have terms consisting only of projectives in $\mathcal{Q}$. %$P_i^\Lambda$ with $i\in[l,d-l-1]$. 
Then we obtain by $F_\Lambda$ being exact, that
\begin{equation}\label{eq:ARtranslateCommuteWithPushdown}
\begin{split}
    F_\Lambda^{m}\circ \tau_{[m]}(Z^\bullet)&= F_\Lambda^{m} (\softTrunc^{\leq 0}_\Lambda(\nu_\Lambda\shift{\projres_m^\Lambda(Z^\bullet)}))\\
    &\cong \softTrunc_\Gamma^{\leq 0}(F_\Lambda^{m}(\nu_\Lambda\shift{\projres_m^\Lambda(Z^\bullet)}))\\
    &\cong\softTrunc_\Gamma^{\leq 0}(\nu_\Gamma\shift{F_\Lambda^{m}(\projres_m^\Lambda(Z^\bullet))})\\
    &\cong\softTrunc_\Gamma^{\leq 0}(\nu_\Gamma\shift{(\projres_m^\Gamma(F_\Lambda^{m}Z^\bullet))})\\
    &=\tau_{[m]}^\Gamma\circ F_\Lambda^{m}(Z^\bullet).
\end{split}
\end{equation}
for all $Z^\bullet\in \mathcal{N}$. We are almost at a point where we can look at the actual AR-triangles. 
However, we still lack one necessary insight on morphisms in the extended module category.

Assume from now that $m$ and $l$ are chosen according to the hypothesis of \Cref{lemma:boundOnCohomology}. Thus, $\mmod\Lambda(d,l)$ is of finite type, and for each $X^\bullet\in\mmod\Lambda(d,l)$, we have
\[
    (\DimVec(X^\bullet)_i)_v\neq 0\quad \text{ and }\quad (\DimVec(X^\bullet)_j)_w\neq 0,
\] 
only if $B=m(l-1)+1\geq |v-w|$.
Assume also that $d\gg0$ is sufficiently large and let $\mathcal{L}$ be the full convex subquiver of $\AR(\mod\Lambda(d,l))$ given by indecomposables which do not admit a sequence of non-zero morphisms to $\nshift{I_{B+ml-1}}{m-1}$ and no sequence of non-zero morphisms from $P_{d-B-ml}$.
Hence, for $X^\bullet\in\mathcal{L}$ we have $(\DimVec(X^\bullet)_\ast)_v=0$ for $v\notin [B+ml,d-B-ml-1]$.
\[
\begin{tikzpicture}
    \fill[pattern=north west lines,pattern color=support] (3,1.5)--(1,1.5)--(2.5,0)--(3,0);
    \fill[pattern=north west lines,pattern color=support,path fading=east] (3,1.5)rectangle(4.3,0);
    \draw (0,0)--(1,0)--(1.2,-.2)--(2.3,-.2)--(2.5,0)--(3,0);
    \draw (0,1.5)--(3,1.5);
    \draw[path fading=east](3,1.5)--(4.3,1.5);
    \draw[path fading=east](3,0)--(4.3,0);
    \draw[dashed] (1,1.5)--(2.5,0);
    \draw[rigid,fill] (2.3,-.2) circle (1pt);
    \node at (2.3,-.3)[anchor=north west,scale=.7]{$\nshift{I_{B+ml-1}}{m-1}$};

    \node at (4.4,.75) [nodeCDots]{};
    \node at (4.5,.75) [nodeCDots]{};
    \node at (4.6,.75) [nodeCDots]{};
    \node at (4.5,1.2)[fill=white]{$\mathcal{L}$};
    \begin{scope}[shift={(5.5,0)}]
        \fill[pattern=north west lines,pattern color=support] (.5,1.5)--(2.5,1.5)--(1,0)--(0.5,0);
        \fill[pattern=north west lines,pattern color=support,path fading=west] (-.8,0)rectangle(.5,1.5);
        \draw (0.5,0)--(1,0)--(1.2,-.2)--(2.3,-.2)--(2.5,0)--(3.5,0);
        \draw (0.5,1.5)--(3.5,1.5);
        \draw[path fading=west] (-.8,0)--(0.5,0);
        \draw[path fading=west] (-.8,1.5)--(0.5,1.5);
        \draw[dashed] (1,0)--(2.5,1.5);
        \draw[rigid,fill] (1.2,-.2) circle (1pt);
        \node at (1.2,-.3)[anchor=north east,scale=.7]{$P_{d-B-ml}$};
    \end{scope}
\end{tikzpicture}
\]

We then want to show the following.
\begin{proposition}\label{lemma:homOfComplexesCovering}
    Let $X^\bullet\in\mathcal{L}$ and $Z^\bullet\in\mmod\Lambda(d,l)$. Then
    \begin{enumerate}
        \item $\Hom_{m\mhyphen\widehat{\Gamma}}(\iota(X^\bullet),\gGroupShift{\iota(Z^\bullet)}{h})\neq 0$, or dually
        \item $\Hom_{m\mhyphen\widehat{\Gamma}}(\gGroupShift{\iota(Z^\bullet)}{h},\iota(X^\bullet))\neq 0$ 
    \end{enumerate}
    only if $\gGroupShift{\iota(Z^\bullet)}{h}\cong \iota(\overline{Z}^\bullet)$ for some $\overline{Z}^\bullet\in\mmod\Lambda(d,l)$
\end{proposition}
We will only prove the first statement, the second is proven similarly. To help us prove this, we need the following observations and claims.
\begin{observation}\label{Observation:MorphismsCovering}
        Let $M\in \mod\widehat{\Gamma}$ and $P^\bullet\in \homotopy^b(\proj\widehat{\Gamma})$. Consider the sets
        \[
        S(M)=\{
            v\in \bZ\,|\, S_v^{\widehat{\Gamma}} \text{ is a composition factor of }M
        \}
        \]
        and 
        \[
        P(P^\bullet)=\{
            v\in \bZ\, |\, P_v^{\widehat{\Gamma}}\text{ is a direct summand of }P^j\text{ for some }j
        \}.
        \]
        If $S(M)\cap P(P^\bullet)=\varnothing$, then $\Hom_{\homotopy^b(\proj\widehat{\Gamma})}(P^\bullet,\nshift{M}{i})=0$ for all $i\in \bZ$.
\end{observation}

\begin{claim}\label{claim:MorphismsCovering1}
    Let $Y^\bullet,V^\bullet$ be objects of $\mmod\widehat{\Gamma}$. If $\Hom_{\D^b(\mod\widehat{\Gamma})}(Y^\bullet,\nshift{(H^iV^\bullet)}{j})=0$ for all $i,j\in\bZ$, then $\Hom_{m\mhyphen\widehat{\Gamma}}(Y^\bullet,V^\bullet)=0$.
    \begin{proof}
        Consider the following soft truncation triangle (see \eqref{eq:softTruncation}) of $\softTrunc^{\leq p}V^\bullet$.
        \[
        \begin{tikzpicture}
            \node (A) at (0,0) []{$\softTrunc^{\leq p-1}\softTrunc^{\leq p}V^\bullet$};
            \node[right=1.2cm of A] (B) {$\softTrunc^{\leq p}V^\bullet$};
            \node[right=1.2cm of B] (C) {$\softTrunc^{\geq p}\softTrunc^{\leq p} V^\bullet$};
            \node[right=1.2cm of C] (D) {$\shift{\softTrunc^{\leq p-1}\softTrunc^{\leq p}V^\bullet}$};

            \draw[->] (A)--(B);
            \draw[->] (B)--(C);
            \draw[->] (C)--(D);
        \end{tikzpicture}
        \]
        We know that $\softTrunc^{\leq p-1}\softTrunc^{\leq p}V^\bullet=\softTrunc^{\leq p-1}V^\bullet$ and $\nshift{H^p(V^\bullet)}{p}=\softTrunc^{\geq p}\softTrunc^{\leq p} V^\bullet$. 
        Thus, from the long exact sequence induced by $\Hom_{\D^b(\widehat{\Gamma})}(Y^\bullet,-)$ we get that
        \[
        \Hom_{\D^b(\mod\widehat{\Gamma})}(Y^\bullet,\softTrunc^{\leq p-1}V^\bullet)\cong \Hom_{\D^b(\mod\widehat{\Gamma})}(Y^\bullet,\softTrunc^{\leq p}V^\bullet).
        \]
        It follows from $V^\bullet\in\mmod\widehat{\Gamma}$ that $\softTrunc^{\leq 0}V^\bullet\cong V^\bullet$ and $\softTrunc^{\leq -m}V^\bullet\cong 0$. Thus, we are done.
    \end{proof}
    \end{claim}

    \begin{claim}\label{claim:MorphismsCovering2}
        Let $X^\bullet\in\mathcal{L}$ and consider the brutal truncation $\brutalTrunc_{\geq -(m-1)}\projres X^\bullet$. If $P_v^\Lambda$ is a direct summand in a term of $\brutalTrunc_{\geq -(m-1)}\projres X^\bullet$ then $v\in[B,d-B-ml-1]$.
    \begin{proof}
        Let $k\geq 0$ be such that $H^{-k}(X^\bullet)\neq 0$. For a direct summand $M$ of $H^{-k}(X^\bullet)$, let $Q^\bullet$ denote the minimal projective resolution $\projres M$.
        \[
        \begin{tikzpicture}
            \node (leftDots1) at (-.1,0)[nodeCDots]{};
            \node[right=.1cm of leftDots1,nodeCDots] (leftDots2){};
            \node[right=.1cm of leftDots2,nodeCDots] (leftDots3){};
            \node[right=.8cm of leftDots3] (A) {$Q^{-t-1}$};
            \node[right=.8cm of A] (B) {$Q^{-t}$};
            \node[right=.8cm of B,nodeCDots] (middleDots1){};
            \node[right=.1cm of middleDots1,nodeCDots] (middleDots2){};
            \node[right=.1cm of middleDots2,nodeCDots] (middleDots3){};
            \node[right=.8cm of middleDots3] (C) {$Q^{-1}$};
            \node[right=.8cm of C] (D) {$Q^0$};
            \node[right=.8cm of D] (E) {$0$};
            \node[right=.8cm of E,nodeCDots] (rightDots1) {};
            \node[right=.1cm of rightDots1,nodeCDots] (rightDots2){};
            \node[right=.1cm of rightDots2,nodeCDots] (rightDots3){};
            \node[below right =.4cm of D] (M) {$M$};
            \draw[->] ([xshift=3pt]leftDots3.center)--(A);
            \draw[->] (A)--(B);
            \draw[->] (B)--([xshift=-3pt]middleDots1.center);
            \draw[->] ([xshift=3pt]middleDots3.center)--(C);
            \draw[->] (C)--(D);
            \draw[->] (D)--(E);
            \draw[->] (E)--([xshift=-3pt]rightDots1.center);
            \draw[->>] (D)--(M);
        \end{tikzpicture}
        \]
        %Since $\Lambda(d,l)$ is a Nakayama algebra, each $Q^i$ is indecomposable. 
        If $P_v^\Lambda$ is a direct summand of $Q^j$ and $Q^{j-1}=P_w^\Lambda$, then $|w-v|\leq l$. 
        From the assumption of $X^\bullet\in\mathcal{L}$, we know that if $\Top M=S_u^\Lambda$ then $u\in[B+ml,d-B-ml-1]$. Thus, if $Q^j=P_v^\Lambda$ then $v\geq B+ml-jl$.
        The minimal projective resolution $\projres H^{-k}(X^\bullet)$ is the direct sum of the minimal projective resolutions of its summands. Hence, if $P_v^\Lambda$ is a direct summand of $(\projres H^{-k}(X^\bullet))^j$, then $v\geq B+ml-jl$.

        Consider now $X^\bullet_i\coloneqq \softTrunc^{\geq -i}X^\bullet$, and the triangle
        \[
        \begin{tikzpicture}
            \node (A) at (0,0)[]{$\softTrunc^{\leq -i-1}X^\bullet_i$};
            \node[right=1.5cm of A] (B) {$X_i^\bullet$};
            \node[right=1.5cm of B] (C) {$\softTrunc^{\geq -i}X_i^\bullet$};
            \node[right=1.5cm of C] (D) {$\shift{\softTrunc^{\leq -i-1}X^\bullet_i}$};

            \draw[->] (A)--(B);
            \draw[->] (B)--(C);
            \draw[->] (C)--(D);
        \end{tikzpicture}
        \]
        where $\softTrunc^{\geq -i}X^\bullet_i\cong \nshift{H^i(X^\bullet)}{i}$ and $\softTrunc^{\leq -i-1}X^\bullet_i\cong X_{i+1}^\bullet$. The triangle equivalence $\homotopy^{-,b}(\proj\Lambda)\xrightarrow{\simeq} \D^b(\mod\Lambda)$, then give a triangle of minimal projective resolutions,
        \[
        \begin{tikzpicture}
            \node (A) at (0,0)[]{$\projres X^\bullet_{i+1}$};
            \node[right=1.5cm of A] (B) {$\projres X_i^\bullet$};
            \node[right=1.5cm of B] (C) {$\projres\softTrunc^{\geq -i}X_i^\bullet$};
            \node[right=1.5cm of C] (D) {$\shift{\projres\softTrunc^{\leq -i-1}X^\bullet_i}.$};

            \draw[->] (A)--(B);
            \draw[->] (B)--(C);
            \draw[->] (C)--(D);
        \end{tikzpicture}
        \]
        Hence, the direct summands of $(\projres X^\bullet_i)^j$ are direct summands of $(\projres \softTrunc^{\leq -i-1}X^\bullet_i)^j$ or $(\projres X^\bullet_{i+1})^j$, and the claim follows from induction on the sequence
        \[
        X^\bullet\cong X_0^\bullet,\,X_1^\bullet,\,\ldots,\,X^\bullet_{m-1}\cong H^{-m-1}(X^\bullet).
        \]
    \end{proof}
    \end{claim}

For $Y^\bullet\in\mmod\widehat{\Gamma}$, denote the convex hull of its support by $\mathcal{I}_{Y^\bullet}=[a_{Y^\bullet},\,b_{Y^\bullet}]\subseteq \bZ$, that is
\[
    \begin{split}
        a_{Y^\bullet}&=\min\{v\in \ALinOrInf\,|\, \exists\,i\in \bZ\text{ s.t. }(\DimVec(X^\bullet)_i)_v\neq 0 \}\\
        b_{Y^\bullet}&=\max\{v\in \ALinOrInf\,|\, \exists\,i\in \bZ\text{ s.t. }(\DimVec(X^\bullet)_i)_v\neq 0 \}\\
    \end{split}
\]
Observe that $\mathcal{I}_{\gGroupShift{Y^\bullet}{h}}=[a_{Y^\bullet}-nh,\, b_{Y^\bullet}-nh]$.
\begin{claim}\label{claim:MorphismsCovering3}
    Let $Z^\bullet \in \mmod\Lambda$ and $h\in\bZ$ such that $\mathcal{I}_{\gGroupShift{\iota(Z^\bullet)}{h}}\subseteq[0,d-1]$. 
    Then $\gGroupShift{\iota(Z^\bullet)}{h}$ has a preimage in $\mmod\Lambda$ under $\iota$.
    \begin{proof}
        If $h=0$ it is trivially true, as $Z^\bullet$ is a preimage of $\gGroupShift{\iota(Z^\bullet)}{0}=\iota(Z^\bullet)$.

        Every term $\iota(Z^i)$ consists of terms $\iota(M_{v,w})$. 
    If $h>0$, we can assume that $Z^\bullet=\softTrunc^{\leq 0}\injres Z^\bullet$ and therefore $v,w\geq a_{\iota(Z^\bullet)}$, and $v-hn,\,w-hn\geq a_{\iota(Z^\bullet)}-hn\geq 0$. Hence, $(\gGroupShift{\iota( Z^\bullet)}{h})^i$ consists of terms $\iota(M_{v-hn,\, w-hn})$, and $\gGroupShift{\iota( Z^\bullet)}{h}$ has a preimage $\overline{Z}^\bullet$ in $\mmod\Lambda(d,l)$. If $h<0$, we assume $Z^\bullet=\softTrunc^{\geq -(m-1)}\projres Z^\bullet$ and similarly obtain a preimage. This concludes the proof.
    \end{proof}
\end{claim}
\begin{proof}[Proof of \Cref{lemma:homOfComplexesCovering}]
    We only prove the first statement, the second is proven similarly. 

    Consider $\Hom_{m\mhyphen\widehat{\Gamma}}(\iota(X^\bullet),\gGroupShift{(\iota Z^\bullet)}{h})$. We have
    \[
    \begin{split}
        \Hom_{m\mhyphen \widehat{\Gamma}}(\iota(X^\bullet),\gGroupShift{(\iota Z^\bullet)}{h})&\cong \Hom_{\homotopy^b(\mod\widehat{\Gamma})}(\projres \iota(X^\bullet),\gGroupShift{(\iota Z^\bullet)}{h})\\
        &\cong\Hom_{\homotopy^b(\mod\widehat{\Gamma})}(\brutalTrunc_{\geq -(m-1)}\projres \iota(X^\bullet),\gGroupShift{(\iota Z^\bullet)}{h})\\
        &\cong\Hom_{\homotopy^b(\mod\widehat{\Gamma})}(\iota(\brutalTrunc_{\geq -(m-1)}\projres X^\bullet),\gGroupShift{(\iota Z^\bullet)}{h})
    \end{split}
    \]
    where the last isomorphism follows from \Cref{claim:MorphismsCovering2}.
    For this to be non-zero, it follows from \Cref{Observation:MorphismsCovering}, \Cref{claim:MorphismsCovering1} and \ref{claim:MorphismsCovering2} that $\mathcal{I}_{\gGroupShift{\iota Z^\bullet}{h}}\cap [B,d-B-ml-1]\neq \varnothing$. 
    The bound $B$ for objects in $\mmod\Lambda(d,l)$ gives $B\geq b_{\iota Z^\bullet}-a_{\iota Z^\bullet}=b_{\gGroupShift{\iota Z^\bullet}{h}}-a_{\gGroupShift{\iota Z^\bullet}{h}}$. Hence, $\mathcal{I}_{\gGroupShift{(\iota Z^\bullet)}{h}}\subseteq [0,d-1]$ and the proof follows from \Cref{claim:MorphismsCovering3}.
\end{proof}

\begin{lemma}\label{lemma:DescriptionOfRadicalCover}
    Let $X^\bullet$ be indecomposable in $\mod\widehat{\Gamma}$. Then $\rad\End_{\D^b(\mod\Gamma)}(F_\pi^{\D}(X^\bullet))$ is given by
    \[
        \rad\End_{\D^b(\mod\widehat{\Gamma})}(X^\bullet)\oplus \bigoplus_{i\in \bZ\setminus \{0\}}\Hom_{\D^b(\mod\widehat{\Gamma})}(\gGroupShift{X^\bullet}{h},X^\bullet).
    \]
\end{lemma}
\begin{proof}
    The proof follows from the proof of \cite[Theorem 3.1]{GG82}, by observing that $\End_{\D^b(\mod\Gamma)}(F_\pi^{\D}(X^\bullet))$ is a $\bZ$-graded algebra:
    \[
    \End_{\D^b(\mod\Gamma)}(F_\pi^{\D}(X^\bullet))\cong \bigoplus_{i\in \bZ}\Hom_{\D^b(\mod\widehat{\Gamma})}(\gGroupShift{X^\bullet}{i},X^\bullet).
    \] 
\end{proof}

The tools are now assembled, and we are ready to look at AR-triangles. 
\begin{proposition}\label{prop:ARtrianglesSentToARtriangles}
    Let $Z^\bullet$ and $\tau_{[m]}^\Lambda Z^\bullet$ be in $\mathcal{L}$, then the AR-triangle
    \begin{equation}\label{eq:PropCoverARinLambda}
    \begin{tikzpicture}[baseline]
        \node (A) at (0,0)  {$\tau_{[m]}^\Lambda Z^\bullet\vphantom{[1]}$};
        \node[right=1.5cm of A] (B)  {$E^\bullet\vphantom{[1]^\Lambda_{[m]}}$};
        \node[right=1.5cm of B] (C) {$Z^\bullet\vphantom{[1]^\Lambda_{[m]}}$};
        \node[right=1.5cm of C] (Ashift){$\shift{\tau_{[m]}^\Lambda Z^\bullet}$,};
        \draw[->] ([yshift=2pt]A.east)--([yshift=2pt]B.west)node[midway,above,scale=.8]{$\alpha$};
        \draw[->] ([yshift=2pt]B.east)--([yshift=2pt]C.west)node[midway,above,scale=.8]{$\beta$};
        \draw[->] ([yshift=2pt]C.east)--([yshift=2pt]Ashift.west)node[midway,above,scale=.8]{$\delta$};
    \end{tikzpicture}
    \end{equation}
    in $\mmod\Lambda$, is sent to an AR-triangle
    \begin{equation}\label{eq:PropCoverARinGamma}
    \begin{tikzpicture}[baseline]
        \node (A) at (0,0) [] {$\tau_{[m]}^\Gamma F_\Lambda^m Z^\bullet\vphantom{[1]_{[m]}^\gamma}$};
        \node[right=1.5cm of A] (B) {$F_\Lambda^m E^\bullet\vphantom{[1]_{[m]}^\gamma}$};
        \node[right=1.5cm of B] (C) {$F_\Lambda^m Z^\bullet\vphantom{[1]_{[m]}^\gamma}$};
        \node[right=1.5cm of C] (Ashift) {$\shift{\tau_{[m]}^\Gamma F_\Lambda^m Z^\bullet}\vphantom{[1]_{[m]}^\gamma}$,};
        \draw[->] ([yshift=2pt]A.east)--([yshift=2pt]B.west)node[midway,above,scale=.8]{$F_\Lambda^m\alpha$};
        \draw[->] ([yshift=2pt]B.east)--([yshift=2pt]C.west)node[midway,above,scale=.8]{$F_\Lambda^m\beta$};
        \draw[->] ([yshift=2pt]C.east)--([yshift=2pt]Ashift.west)node[midway,above,scale=.8]{$F_\Lambda^m\delta$};
    \end{tikzpicture}
    \end{equation}
    in $\mmod\Gamma$.
\end{proposition}
\begin{proof}
    We know by \eqref{eq:ARtranslateCommuteWithPushdown} that \eqref{eq:PropCoverARinGamma} is the image of \eqref{eq:PropCoverARinLambda}. 
    Moreover, by \Cref{prop:pushdownProperties}(\ref{prop:pushdownProperties5}) both $\tau_{[m]}^\Gamma F_\Lambda^m Z^\bullet$ and $F_\Lambda^m Z^\bullet$ are indecomposable, and $F_\Lambda^m\delta\neq 0$.
    There exists some AR-triangle $\tau_{[m]}^\Gamma F_\Lambda^m Z^\bullet\xrightarrow{\alpha'} Y^\bullet\xrightarrow{\beta'} F_\Lambda Z^\bullet \xrightarrow{\delta'} \shift{\tau_{[m]}^\Gamma F_\Lambda^m Z^\bullet}$ in $\mmod\Gamma$, and we can construct the commutative diagram
    \[
    \begin{tikzpicture}[baseline]
        \node (A) at (0,0) [] {$\tau_{[m]}^\Gamma F_\Lambda^m Z^\bullet\vphantom{[1]_{[m]}^\gamma}$};
        \node[right=1.5cm of A] (B) {$Y^\bullet\vphantom{[1]_{[m]}^\gamma}$};
        \node[right=1.5cm of B] (C) {$F_\Lambda^m Z^\bullet\vphantom{[1]_{[m]}^\gamma}$};
        \node[right=1.5cm of C](Ashift) {$\shift{\tau_{[m]}^\Gamma F_\Lambda^m Z^\bullet}\vphantom{[1]_{[m]}^\gamma}$,};
        \draw[->] ([yshift=2pt]A.east)--([yshift=2pt]B.west)node[midway,above,scale=.8]{$\alpha'$};
        \draw[->] ([yshift=2pt]B.east)--([yshift=2pt]C.west)node[midway,above,scale=.8]{$\beta'$};
        \draw[->] ([yshift=2pt]C.east)--([yshift=2pt]Ashift.west)node[midway,above,scale=.8]{$\delta'$};
        
        \node[below=1.5cm of A,anchor=south] (A2) {$\tau_{[m]}^\Gamma F_\Lambda^m Z^\bullet\vphantom{[1]_{[m]}^\gamma}$};
        \node[below=1.5cm of B,anchor=south] (B2) {$F_\Lambda^m E^\bullet\vphantom{[1]_{[m]}^\gamma}$};
        \node[below=1.5cm of C,anchor=south] (C2) {$F_\Lambda^m Z^\bullet\vphantom{[1]_{[m]}^\gamma}$};
        \node[below=1.5cm of Ashift,anchor=south](Ashift2) {$\shift{\tau_{[m]}^\Gamma F_\Lambda^m Z^\bullet}\vphantom{[1]_{[m]}^\gamma}$,};
        \draw[->] ([yshift=2pt]A2.east)--([yshift=2pt]B2.west)node[midway,above,scale=.8]{$F_\Lambda^m\alpha$};
        \draw[->] ([yshift=2pt]B2.east)--([yshift=2pt]C2.west)node[midway,above,scale=.8]{$F_\Lambda^m\beta$};
        \draw[->] ([yshift=2pt]C2.east)--([yshift=2pt]Ashift2.west)node[midway,above,scale=.8]{$F_\Lambda^m\delta$};
        \draw[double equal sign distance] (A)--(A2);
        \draw[double equal sign distance] (Ashift)--(Ashift2);
        \draw[->] (B)--(B2)node[midway,left,scale=.8]{$f$};
        \draw[dashed,->] (C)--(C2)node[midway,right,scale=.8]{$g$};
    \end{tikzpicture}
    \]
    where $f$ exists since $F_\Lambda^m\alpha$ is not a section. If $g$ is an isomorphism, we are done. 

    Assume therefore that $g$ is not an isomorphism, and thus lie in $\rad\End_{\D^b(\mod\Gamma)}(Z^\bullet)$. 
    From \Cref{lemma:DescriptionOfRadicalCover} and \Cref{lemma:homOfComplexesCovering}, we then find a non-retraction $\overline{g}\colon \bigoplus Z_i^\bullet\to Z^\bullet\in\mmod\Lambda$ such that $F_\Lambda^m(\overline{g})=g$. 
    The fact that \eqref{eq:PropCoverARinLambda} is an AR-triangle in $\mmod\Lambda$ then tells us that $\overline{g}$ factors through $\beta$.
    That is, we have $h\colon \bigoplus Z_i^\bullet\to E^\bullet$ such that $\overline{g}=\beta\circ h$, and consequently $g=F_\Lambda^m(\beta)\circ F_\Lambda^m(h)$. Observe now that
    \[
    \begin{split}
        \delta'&=F_\Lambda^m(\delta)\circ g\\
        &=F_\Lambda^m(\delta)\circ F_\Lambda^m(\beta)\circ F_\Lambda^m(h)\\
        &=0\circ F_\Lambda^m(h)=0
    \end{split}
    \]
    which is a contradiction, hence the proof is done.
\end{proof}

\subsection{Components of the AR-quiver}
Assume now that $d$ is sufficiently large and $m,l$ remain chosen as to satisfy \Cref{lemma:boundOnCohomology}. 
We will at some point in our knitting procedure of $\AR(\mmod\Lambda(d,l))$ reach the subquiver $\mathcal{L}$, and thus start seeing AR-triangles which are also AR-triangles after being sent to $\mmod\Gamma$. 
We have found a component of $\AR(\mmod\Gamma)$ if we start seeing repetitions under $F_\Lambda^m$ when moving even further ahead in $\mathcal{L}$. We will now make this more explicit.

Consider a projective $P_i=M_{i-l+1,\,i}$ which lie in $\mathcal{L}$. We know that the AR-triangle starting in $P_i$ is on the form:
\begin{equation*}
    \begin{tikzpicture}[baseline]
        \node (A) at (0,0) [] {$P_i$};
        \node[right=1.5cm of A] (B)  {$[P_{i-l+1}\to P_i]$};
        \node[right=1.5cm of B] (C)  {$P_{i-l+1}$};
        \node[right=1.5cm of C] (D)  {$\shift{P_i}$};
        \draw[->] ([yshift=1pt]A.east)--([yshift=1pt]B.west);
        \draw[->] ([yshift=1pt]B.east)--([yshift=1pt]C.west);
        \draw[->] ([yshift=1pt]C.east)--([yshift=1pt]D.west);
    \end{tikzpicture}
\end{equation*}
This triangle is obviously sent by $F_\Lambda^m$ to the same triangle as the one beginning in $P_{i+n}$. 
Moving diagonally 'upwards' from $P_i$ in the AR-quiver, we can see that the AR-triangle beginning in $[P_{i-l+1}\to P_i]$ is sent to the same AR-triangle as the one beginning in $[P_{i+n-l+1}\to P_{i+n}]$. Iteratively, all the AR-triangles beginning on the diagonal from $P_i$ are sent to the same triangles as those on the diagonal from $P_{i+n}$.
\[
\begin{tikzpicture}
    \fill[pattern=north west lines,pattern color=support] (3,1.5)--(2.8,1.5)--(1.2,-.2)--(2.3,-.2)--(2.5,0)--(3,0);
    \fill[pattern=north west lines,pattern color=support,path fading=east] (3,1.5)rectangle(4.3,0);
    \draw (0,0)--(1,0)--(1.2,-.2)--(2.3,-.2)--(2.5,0)--(3,0);
    \draw (0,1.5)--(3,1.5);
    \draw[path fading=east](3,1.5)--(4.3,1.5);
    \draw[path fading=east](3,0)--(4.3,0);
    \draw[dashed] (1.2,-.2)--(2.8,1.5);
    \draw[rigid,fill] (1.2,-.2) circle (1pt);
    \node at (1.2,-.3)[below,scale=.7]{$P_i\vphantom{_{i+n}}$};

    \node at (4.4,.75) [nodeCDots]{};
    \node at (4.5,.75) [nodeCDots]{};
    \node at (4.6,.75) [nodeCDots]{};
    \begin{scope}[shift={(5.5,0)}]
        \fill[pattern=north west lines,pattern color=support] (0.5,1.5)--(2.8,1.5)--(1.2,-.2)--(1,0)--(0.5,0);
        \fill[pattern=north west lines,pattern color=support,path fading=west] (-.8,0)rectangle(0.5,1.5);
        \draw (0.5,0)--(1,0)--(1.2,-.2)--(2.3,-.2)--(2.5,0)--(3.5,0);
        \draw (0.5,1.5)--(3.5,1.5);
        \draw[path fading=west] (-.8,0)--(0.5,0);
        \draw[path fading=west] (-.8,1.5)--(0.5,1.5);
        \draw[dashed] (1.2,-.2)--(2.8,1.5);
        \draw[rigid,fill] (1.2,-.2) circle (1pt);
        \node at (1.2,-.3)[below,scale=.7]{$P_{i+n}$};
        \node at (-.5,1.1)[fill=white,scale=.7]{$\AR(\mmod\Gamma)$};
    \end{scope}
\end{tikzpicture}
\]
We can therefore see that we obtain a repetition under $F_\Lambda^m$, and everything in between the repetition is sent to AR-triangles in $\mmod\Gamma$. 
Hence, by identifying the repetition we obtain a component of $\AR(\mmod\Gamma)$, which by \Cref{lemma:finiteComponentIsWholeAR} is the whole of $\AR(\mmod\Gamma)$. One direction of the following theorem is then proven.

\begin{theorem}\label{theorem:finitenessOfCyclicNakayama}
    Let $\Gamma$ be a cyclic nakayama algebra with homogeneous relations of length $l$. $\mmod\Gamma$ is finite if and only if
    \begin{enumerate}
        \item $l=2$,
        \item $m\leq4$ and $l=3$,
        \item $m\leq 2$ and either $l=4$ or $l=5$, or
        \item $m=1$
    \end{enumerate}
\end{theorem}
\begin{proof}
    The 'if' direction follows from the discussion preceding the theorem.
    
    For the 'only if' direction, assume that $\Gamma$ is a cyclic Nakayama algebra which do not satisfy the condition, then by \Cref{thm:WhenIsLinearNakayamaFinite} we see that $\mmod\Lambda(d,l)$ is of infinite type if we let $d=\max\{9,\,l+2\}$. 
    Further, from \Cref{prop:pushdownProperties}(\ref{prop:pushdownProperties5}) it follows that given an indecomposable $N\in\mmod\Gamma$, there are only finitely many indecomposables $M\in \mmod\Lambda(d,l)$ such that $F_\Lambda^m(M)\cong N$. Moreover, each indecomposable of $\mmod\Lambda(d,l)$ is sent to an indecomposable of $\mmod\Gamma$, so we conclude that $\mmod\Gamma$ is of infinite type.
\end{proof}

\subsection{Examples}
Note that in practice we may not need to work explicitly with the convex subquiver $\mathcal{L}$. If we identify a repetition beginning earlier than $\mathcal{L}$ in the postprojective component, it will carry on until we reach the subquiver.
\begin{example}
    Consider the algebra $\Gamma$ given by $(\Delta_{1},\rho_{\Delta,2})$, i.e.
    \[
    \begin{tikzpicture}
        \node (A) at (0,0) [nodeDots]{};
        \node at (-.1,-.1) [below,scale=.8]{$0$};
        \draw[-latex] (0,.1) ..controls (.1,.6) and (.6,.1) .. (.1,0)node[midway,above]{$\delta$};
    \end{tikzpicture}
    \]
    with relations generated by $\delta^2$. $\Gamma$ has two indecomposable modules, the projective $\Gamma$ and a simple $S$.

    Let $d\gg1$ and consider $2\nmod\Lambda=2\nmod\Lambda(d,2)$. We can start knitting the AR-quiver of this as follows
    \[
        \includegraphics{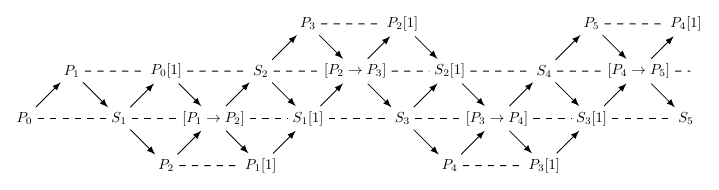}
    \]
    If we now use the push-down functor $F_\Lambda^2$, we obtain a repeating pattern which will continue until we eventually reach $\mathcal{L}$.
    \[
        \includegraphics{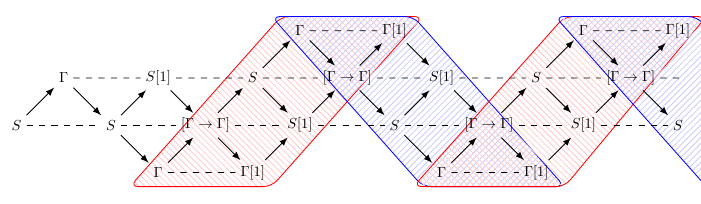}
    \]
    Hence, we have obtained a finite component of $\AR(2\nmod\Gamma)$,
    \[
        \includegraphics{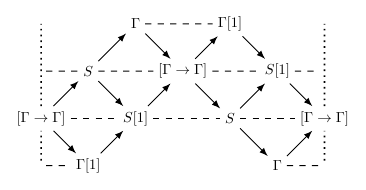}
    \]
    which by \Cref{lemma:finiteComponentIsWholeAR} is then the whole AR-quiver. The same procedure for $3\nmod\Gamma$ gives us
    \[
        \includegraphics[width=.7\linewidth]{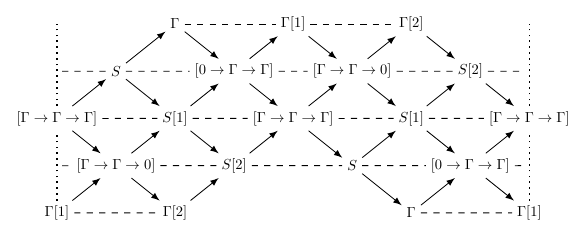}
    \]
\end{example}

\begin{example}
    Let $\Gamma$ be given by 
    \[
    \begin{tikzpicture}[tips=proper,scale=.7]
        \node (1) at (1,-1) [nodeDots] {};
        \node (1-label) at (1,-1.1)[anchor=west,scale=.8]{$1$};
        \node (0) at (-1,-1) [nodeDots] {};
        \node (0-label) at (-1.1,-1)[anchor=east,scale=.8]{$0$};

        \draw[-latex] ([xshift=-5pt,yshift=-2.5pt]1.center)--([xshift=5pt,yshift=-2.5pt]0.center)node[midway,below,scale=.8]{$\delta_1$};
        \draw[-latex] ([xshift=5pt,yshift=2.5pt]0.center)--([xshift=-5pt,yshift=2.5pt]1.center)node[midway,above,scale=.8]{$\delta_0$};
    \end{tikzpicture}
    \]
    with relations generated by all paths of length $3$, and denote the modules
    \[
    \begin{tikzpicture}[tips=proper,scale=.7,baseline={([yshift=-3pt]current bounding box.center)}]
        \node (1) at (1,-1) [] {$\K$};
        
        \node (0) at (-1,-1) [] {$\K$};

        \draw[-latex] ([xshift=-5pt,yshift=-2.5pt]1.center)--([xshift=5pt,yshift=-2.5pt]0.center)node[midway,below,scale=.8]{$0$};
        \draw[-latex] ([xshift=5pt,yshift=2.5pt]0.center)--([xshift=-5pt,yshift=2.5pt]1.center)node[midway,above,scale=.8]{$1$};
    \end{tikzpicture}
    \quad\text{and}\quad 
    \begin{tikzpicture}[tips=proper,scale=.7,baseline={([yshift=-3pt]current bounding box.center)}]
        \node (1) at (1,-1) [] {$\K$};
        
        \node (0) at (-1,-1) [] {$\K$};

        \draw[-latex] ([xshift=-5pt,yshift=-2.5pt]1.center)--([xshift=5pt,yshift=-2.5pt]0.center)node[midway,below,scale=.8]{$1$};
        \draw[-latex] ([xshift=5pt,yshift=2.5pt]0.center)--([xshift=-5pt,yshift=2.5pt]1.center)node[midway,above,scale=.8]{$0$};
    \end{tikzpicture}
    \]
    by $U$ and $L$ respectively.

    We let $d\gg 1$, and start knitting $\AR(2\nmod\Lambda(d,3))$, see \Cref{fig:StartOfAR2modLambdad3}. 
    After applying $F_\Lambda^2$, a repeating pattern emerges, see \Cref{fig:StartOfAR2modLambdad3Repetition}. We therefore obtain the following AR-quiver of $2\nmod\Gamma$, where the vertices marked with lines over and/or under are identified.
    \[
        \includegraphics[width=.9\linewidth]{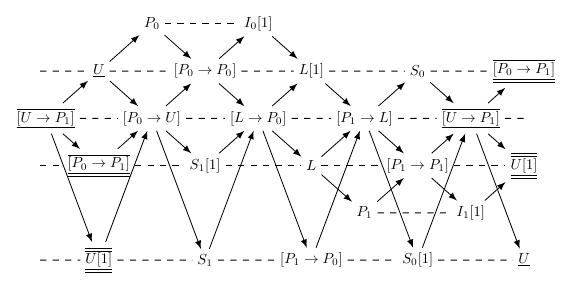}
    \]
\end{example}

\begin{example}
    Let $\Gamma$ be given by 
    \[
    \begin{tikzpicture}[tips=proper,scale=.7]
        \node (3) at (0,0) [nodeDots] {};
        \node (3-label) (0,.1)[above,scale=.8] {$3$};
        \node (2) at (1,-1) [nodeDots] {};
        \node (2-label) at (1.1,-1)[right,scale=.8]{$2$};
        \node (1) at (0,-2) [nodeDots] {};
        \node (1-label) at (0,-2.1)[below,scale=.8]{$1$};
        \node (0) at (-1,-1) [nodeDots] {};
        \node (0-label) at (-1.1,-1)[left,scale=.8]{$0$};

        \draw[-latex] ([xshift=3pt,yshift=-3pt]3.center)--([xshift=-3pt,yshift=3pt]2.center)node[midway,anchor=south west,scale=.8]{$\delta_3$};
        \draw[-latex] ([xshift=-3pt,yshift=-3pt]2.center)--([xshift=3pt,yshift=3pt]1.center)node[midway,anchor=north west,scale=.8]{$\delta_2$};
        \draw[-latex] ([xshift=-3pt,yshift=3pt]1.center)--([xshift=3pt,yshift=-3pt]0.center)node[midway,anchor=north east,scale=.8]{$\delta_1$};
        \draw[-latex] ([xshift=3pt,yshift=3pt]0.center)--([xshift=-3pt,yshift=-3pt]3.center)node[midway,anchor=south east,scale=.8]{$\delta_0$};
    \end{tikzpicture}
    \]
    with relations generated by all paths of length $3$. Then the AR-quiver of $3\nmod\Gamma$ can be found through the same procedure as above, and it has the following form
    \[
    \includegraphics[width=.7\linewidth]{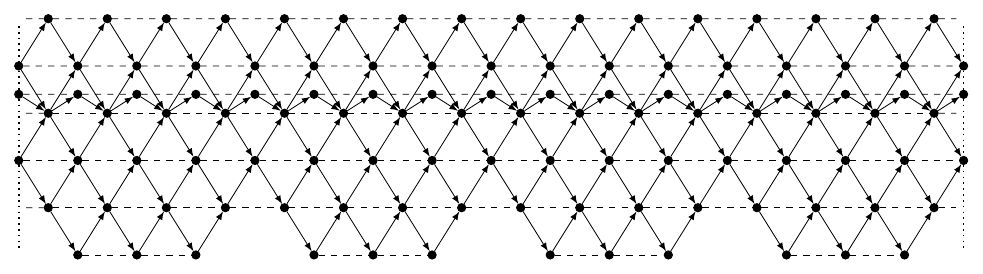}
    \]
\end{example}

\begin{landscape}
    \begin{figure}
        \centering
        \includegraphics[width=.9\linewidth]{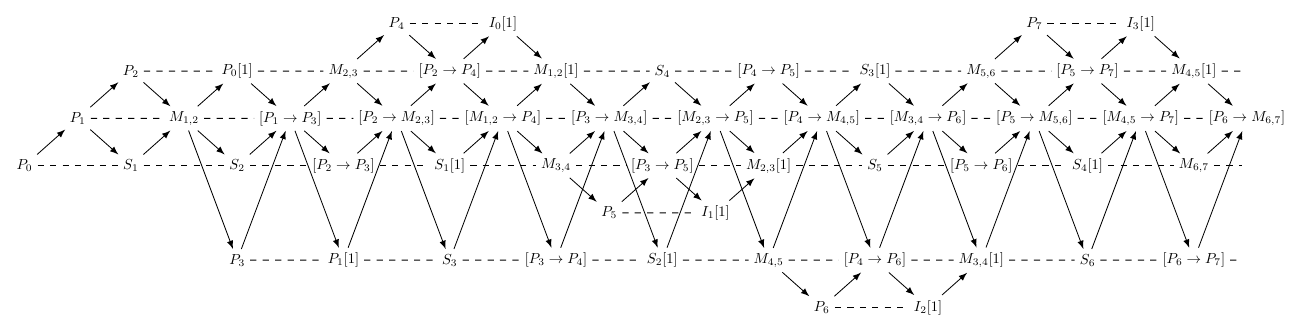}
        \caption{The beginning of $\AR(2\nmod\Lambda(d,3))$}
        \label{fig:StartOfAR2modLambdad3}
    \end{figure}
    \begin{figure}
        \centering
        \includegraphics[width=.9\linewidth]{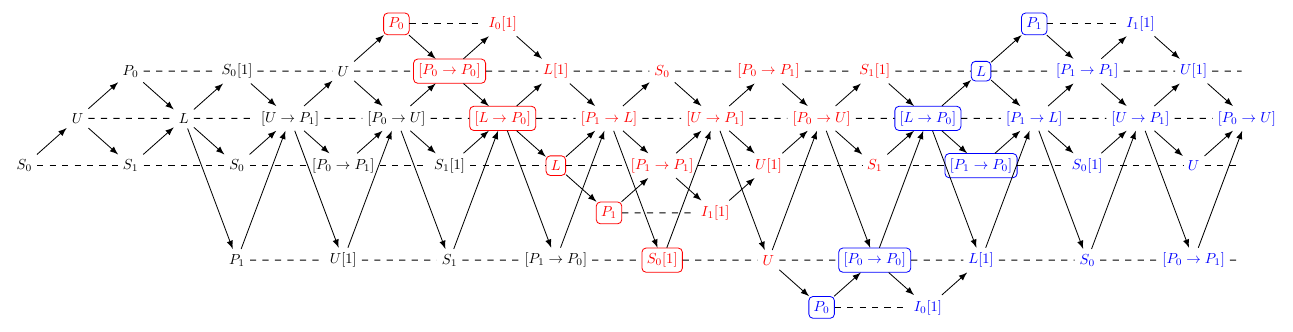}
        \caption{Image of $F_\Lambda^2$ with repetition marked.}
        \label{fig:StartOfAR2modLambdad3Repetition}
    \end{figure}
\end{landscape}

\printbibliography

@misc{AHJKPT25+,
  title={Higher torsion classes, $\tau_d$-tilting theory, and silting complexes},
  author={August, Jenny and Haugland, Johanne and Jacobsen, Karin M and Kvamme, Sondre and Palu, Yann and Treffinger, Hipolito},
  note={In preparation},
}

@article{AIR14,
  author     = {Adachi, Takahide and Iyama, Osamu and Reiten, Idun},
  title      = {{$\tau$}-tilting theory},
  journal    = {Compos. Math.},
  fjournal   = {Compositio Mathematica},
  volume     = {150},
  year       = {2014},
  number     = {3},
  pages      = {415--452},
  issn       = {0010-437X},
  mrclass    = {16G10 (16E35 16S90 18E40)},
  mrnumber   = {3187626},
  mrreviewer = {Octavio Mendoza Hern\'{a}ndez},
  doi        = {10.1112/S0010437X13007422},
  url        = {https://doi.org/10.1112/S0010437X13007422}
}

@book {ARS,
	AUTHOR = {Auslander, Maurice and Reiten, Idun and Smal\o , Sverre O.},
	TITLE = {Representation theory of {A}rtin algebras},
	SERIES = {Cambridge Studies in Advanced Mathematics},
	VOLUME = {36},
	PUBLISHER = {Cambridge University Press, Cambridge},
	YEAR = {1995},
	PAGES = {xiv+423},
	ISBN = {0-521-41134-3},
	MRCLASS = {16G10 (16-02 16E30 16G70)},
	MRNUMBER = {1314422},
	MRREVIEWER = {Andrzej Skowro\'{n}ski},
	DOI = {10.1017/CBO9780511623608},
	URL = {https://doi.org/10.1017/CBO9780511623608},
}

@book {ASS06,
    AUTHOR = {Assem, Ibrahim and Simson, Daniel and Skowro\'nski, Andrzej},
     TITLE = {Elements of the representation theory of associative algebras.
              {V}ol. 1},
    SERIES = {London Mathematical Society Student Texts},
    VOLUME = {65},
      NOTE = {Techniques of representation theory},
 PUBLISHER = {Cambridge University Press, Cambridge},
      YEAR = {2006},
     PAGES = {x+458},
      ISBN = {978-0-521-58423-4; 978-0-521-58631-3; 0-521-58631-3},
   MRCLASS = {16G10 (16-02)},
  MRNUMBER = {2197389},
MRREVIEWER = {Peter\ W.\ Donovan},
       DOI = {10.1017/CBO9780511614309},
       URL = {https://doi.org/10.1017/CBO9780511614309},
}

@article {AST08,
    AUTHOR = {Assem, Ibrahim and Souto, Mar\'ia Jos\'e{} Salorio and
              Trepode, Sonia},
     TITLE = {Ext-projectives in suspended subcategories},
   JOURNAL = {J. Pure Appl. Algebra},
  FJOURNAL = {Journal of Pure and Applied Algebra},
    VOLUME = {212},
      YEAR = {2008},
    NUMBER = {2},
     PAGES = {423--434},
      ISSN = {0022-4049,1873-1376},
   MRCLASS = {18E30 (16G70 18E35)},
  MRNUMBER = {2357343},
MRREVIEWER = {Pedro\ A.\ Guil Asensio},
       DOI = {10.1016/j.jpaa.2007.06.002},
       URL = {https://doi.org/10.1016/j.jpaa.2007.06.002},
}

@book {Bar15,
    AUTHOR = {Barot, Michael},
     TITLE = {Introduction to the representation theory of algebras},
 PUBLISHER = {Springer, Cham},
      YEAR = {2015},
     PAGES = {x+179},
      ISBN = {978-3-319-11474-3; 978-3-319-11475-0},
   MRCLASS = {16Gxx},
  MRNUMBER = {3309698},
MRREVIEWER = {Wolfgang\ Rump},
       DOI = {10.1007/978-3-319-11475-0},
       URL = {https://doi.org/10.1007/978-3-319-11475-0},
}

@article {Bau83,
    AUTHOR = {Bautista, Raymundo},
     TITLE = {Irreducible morphisms and the radical of a category},
   JOURNAL = {An. Inst. Mat. Univ. Nac. Aut\'onoma M\'exico},
  FJOURNAL = {Anales del Instituto de Matem\'aticas. Universidad Nacional
              Aut\'onoma de M\'exico},
    VOLUME = {22},
      YEAR = {1982},
     PAGES = {83--135},
      ISSN = {0185-0644},
   MRCLASS = {16A64},
  MRNUMBER = {736555},
MRREVIEWER = {C.\ M.\ Ringel},
}

@incollection {BBD82,
    AUTHOR = {Be{\u i}linson, A. A. and Bernstein, J. and Deligne, P.},
     TITLE = {Faisceaux pervers},
 BOOKTITLE = {Analysis and topology on singular spaces, {I} ({L}uminy,
              1981)},
    SERIES = {Ast\'erisque},
    VOLUME = {100},
     PAGES = {5--171},
 PUBLISHER = {Soc. Math. France, Paris},
      YEAR = {1982},
   MRCLASS = {32C38},
  MRNUMBER = {751966},
MRREVIEWER = {Zoghman\ Mebkhout},
}

@article {BG81,
    AUTHOR = {Bongartz, K. and Gabriel, P.},
     TITLE = {Covering spaces in representation-theory},
   JOURNAL = {Invent. Math.},
  FJOURNAL = {Inventiones Mathematicae},
    VOLUME = {65},
      YEAR = {1981/82},
    NUMBER = {3},
     PAGES = {331--378},
      ISSN = {0020-9910,1432-1297},
   MRCLASS = {16A64},
  MRNUMBER = {643558},
       DOI = {10.1007/BF01396624},
       URL = {https://doi.org/10.1007/BF01396624},
}

@incollection {Gab81,
    AUTHOR = {Gabriel, P.},
     TITLE = {The universal cover of a representation-finite algebra},
 BOOKTITLE = {Representations of algebras ({P}uebla, 1980)},
    SERIES = {Lecture Notes in Math.},
    VOLUME = {903},
     PAGES = {68--105},
 PUBLISHER = {Springer, Berlin-New York},
      YEAR = {1981},
      ISBN = {3-540-11179-4},
   MRCLASS = {16A64},
  MRNUMBER = {654725},
}

@article {GG82,
    AUTHOR = {Gordon, Robert and Green, Edward L.},
     TITLE = {Graded {A}rtin algebras},
   JOURNAL = {J. Algebra},
  FJOURNAL = {Journal of Algebra},
    VOLUME = {76},
      YEAR = {1982},
    NUMBER = {1},
     PAGES = {111--137},
      ISSN = {0021-8693},
   MRCLASS = {16A64 (16A46)},
  MRNUMBER = {659212},
MRREVIEWER = {Idun\ Reiten},
       DOI = {10.1016/0021-8693(82)90240-X},
       URL = {https://doi.org/10.1016/0021-8693(82)90240-X},
}

@article {GH11,
    AUTHOR = {Green, Edward L. and Happel, Dieter},
     TITLE = {Gradings and derived categories},
   JOURNAL = {Algebr. Represent. Theory},
  FJOURNAL = {Algebras and Representation Theory},
    VOLUME = {14},
      YEAR = {2011},
    NUMBER = {3},
     PAGES = {497--513},
      ISSN = {1386-923X,1572-9079},
   MRCLASS = {16E35 (16W50)},
  MRNUMBER = {2785920},
MRREVIEWER = {Silvana\ Bazzoni},
       DOI = {10.1007/s10468-009-9200-3},
       URL = {https://doi.org/10.1007/s10468-009-9200-3},
}

@article {Gre83,
    AUTHOR = {Green, Edward L.},
     TITLE = {Graphs with relations, coverings and group-graded algebras},
   JOURNAL = {Trans. Amer. Math. Soc.},
  FJOURNAL = {Transactions of the American Mathematical Society},
    VOLUME = {279},
      YEAR = {1983},
    NUMBER = {1},
     PAGES = {297--310},
      ISSN = {0002-9947,1088-6850},
   MRCLASS = {16A64 (05C20 16A03 16A90 57M10)},
  MRNUMBER = {704617},
MRREVIEWER = {Christine\ Riedtmann},
       DOI = {10.2307/1999386},
       URL = {https://doi.org/10.2307/1999386},
}

@incollection {Gre05,
    AUTHOR = {Green, Edward L.},
     TITLE = {The work of {R}oberto {M}art\'inez-{V}illa},
 BOOKTITLE = {Algebraic structures and their representations},
    SERIES = {Contemp. Math.},
    VOLUME = {376},
     PAGES = {49--59},
 PUBLISHER = {Amer. Math. Soc., Providence, RI},
      YEAR = {2005},
      ISBN = {0-8218-3630-7},
   MRCLASS = {16-03 (01A60 01A70 16G10 16G70)},
  MRNUMBER = {2147014},
MRREVIEWER = {D.\ Zacharia},
       DOI = {10.1090/conm/376/06950},
       URL = {https://doi.org/10.1090/conm/376/06950},
}

@misc{Gup24,
      title={$d$-term silting objects, torsion classes, and cotorsion classes}, 
      author={Esha Gupta},
      year={2024},
      eprint={2407.10562},
      archivePrefix={arXiv},
      primaryClass={math.RT},
      url={https://arxiv.org/abs/2407.10562}, 
}

@misc{GZ25,
      title={Semibricks and wide subcategories in extended module categories}, 
      author={Esha Gupta and Yu Zhou},
      year={2025},
      eprint={2511.08157},
      archivePrefix={arXiv},
      primaryClass={math.RT},
      url={https://arxiv.org/abs/2511.08157}, 
}

@article {Hap82,
    AUTHOR = {Happel, Dieter},
     TITLE = {Composition factors for indecomposable modules},
   JOURNAL = {Proc. Amer. Math. Soc.},
  FJOURNAL = {Proceedings of the American Mathematical Society},
    VOLUME = {86},
      YEAR = {1982},
    NUMBER = {1},
     PAGES = {29--31},
      ISSN = {0002-9939,1088-6826},
   MRCLASS = {16A64 (16A46)},
  MRNUMBER = {663860},
       DOI = {10.2307/2044390},
       URL = {https://doi.org/10.2307/2044390},
}

@book{Hap88,
    AUTHOR = {Happel, Dieter},
     TITLE = {Triangulated categories in the representation theory of
              finite-dimensional algebras},
    SERIES = {London Mathematical Society Lecture Note Series},
    VOLUME = {119},
 PUBLISHER = {Cambridge University Press, Cambridge},
      YEAR = {1988},
     PAGES = {x+208},
      ISBN = {0-521-33922-7},
   MRCLASS = {16A46 (16A48 16A62 16A64 18E30)},
  MRNUMBER = {935124},
MRREVIEWER = {Alfred\ G.\ Wiedemann},
       DOI = {10.1017/CBO9780511629228},
       URL = {https://doi.org/10.1017/CBO9780511629228},
}

@article {HS10,
    AUTHOR = {Happel, Dieter and Seidel, Uwe},
     TITLE = {Piecewise hereditary {N}akayama algebras},
   JOURNAL = {Algebr. Represent. Theory},
  FJOURNAL = {Algebras and Representation Theory},
    VOLUME = {13},
      YEAR = {2010},
    NUMBER = {6},
     PAGES = {693--704},
      ISSN = {1386-923X,1572-9079},
   MRCLASS = {16G10 (16E35)},
  MRNUMBER = {2736030},
MRREVIEWER = {Xueqing\ Chen},
       DOI = {10.1007/s10468-009-9169-y},
       URL = {https://doi.org/10.1007/s10468-009-9169-y},
}

@article{HV83,
  author     = {Happel, Dieter and Vossieck, Dieter},
  title      = {Minimal algebras of infinite representation type with
                preprojective component},
  journal    = {Manuscripta Math.},
  fjournal   = {Manuscripta Mathematica},
  volume     = {42},
  year       = {1983},
  number     = {2-3},
  pages      = {221--243},
  issn       = {0025-2611,1432-1785},
  mrclass    = {16A64},
  mrnumber   = {701205},
  mrreviewer = {V.\ Dlab},
  doi        = {10.1007/BF01169585},
  url        = {https://doi.org/10.1007/BF01169585}
}

@article {INP24,
    AUTHOR = {Iyama, Osamu and Nakaoka, Hiroyuki and Palu, Yann},
     TITLE = {Auslander-{R}eiten theory in extriangulated categories},
   JOURNAL = {Trans. Amer. Math. Soc. Ser. B},
  FJOURNAL = {Transactions of the American Mathematical Society. Series B},
    VOLUME = {11},
      YEAR = {2024},
     PAGES = {248--305},
      ISSN = {2330-0000},
   MRCLASS = {16G70 (16E30 16G50 18E10)},
  MRNUMBER = {4697926},
MRREVIEWER = {Qinghua\ Chen},
       DOI = {10.1090/btran/159},
       URL = {https://doi.org/10.1090/btran/159},
}

@incollection {Iya08,
	AUTHOR = {Iyama, Osamu},
	TITLE = {Auslander--{R}eiten theory revisited},
	BOOKTITLE = {Trends in representation theory of algebras and related
	topics},
	SERIES = {EMS Ser. Congr. Rep.},
	PAGES = {349--397},
	PUBLISHER = {Eur. Math. Soc., Z\"{u}rich},
	YEAR = {2008},
	MRCLASS = {16G70 (16G10 16G30 16G50)},
	MRNUMBER = {2484730},
	MRREVIEWER = {Markus Schmidmeier},
	DOI = {10.4171/062-1/8},
	URL = {https://doi.org/10.4171/062-1/8},
}

@article {Jor09,
    AUTHOR = {J{\o}rgensen, Peter},
     TITLE = {Auslander-{R}eiten triangles in subcategories},
   JOURNAL = {J. K-Theory},
  FJOURNAL = {Journal of K-Theory. K-Theory and its Applications in Algebra,
              Geometry, Analysis \& Topology},
    VOLUME = {3},
      YEAR = {2009},
    NUMBER = {3},
     PAGES = {583--601},
      ISSN = {1865-2433,1865-5394},
   MRCLASS = {16G70 (18E30)},
  MRNUMBER = {2507732},
MRREVIEWER = {Jue\ Le},
       DOI = {10.1017/is008007021jkt056},
       URL = {https://doi.org/10.1017/is008007021jkt056},
}

@article{JJ20,
title = {Maximal $\tau_d$-rigid pairs},
JOURNAL = {J. Algebra},
FJOURNAL = {Journal of Algebra},
volume = {546},
pages = {119-134},
year = {2020},
issn = {0021-8693},
doi = {https://doi.org/10.1016/j.jalgebra.2019.10.046},
url = {https://www.sciencedirect.com/science/article/pii/S0021869319306167},
author = {Karin M. Jacobsen and Peter Jørgensen},
MRCLASS = {18E10 (16G10 18G80)},
MRNUMBER = {4032280},
MRREVIEWER = {Chrysostomos Psaroudakis}
}

@incollection {Kel96,
    AUTHOR = {Keller, Bernhard},
     TITLE = {Derived categories and their uses},
 BOOKTITLE = {Handbook of algebra, {V}ol.\ 1},
    SERIES = {Handb. Algebr.},
    VOLUME = {1},
     PAGES = {671--701},
 PUBLISHER = {Elsevier/North-Holland, Amsterdam},
      YEAR = {1996},
      ISBN = {0-444-82212-7},
   MRCLASS = {18E30},
  MRNUMBER = {1421815},
MRREVIEWER = {Jeremy\ Rickard},
       DOI = {10.1016/S1570-7954(96)80023-4},
       URL = {https://doi.org/10.1016/S1570-7954(96)80023-4},
}

@article {KL06,
    AUTHOR = {Krause, Henning and Le, Jue},
     TITLE = {The {A}uslander-{R}eiten formula for complexes of modules},
   JOURNAL = {Adv. Math.},
  FJOURNAL = {Advances in Mathematics},
    VOLUME = {207},
      YEAR = {2006},
    NUMBER = {1},
     PAGES = {133--148},
      ISSN = {0001-8708,1090-2082},
   MRCLASS = {16G30 (16E05 18E30)},
  MRNUMBER = {2264068},
MRREVIEWER = {Jun-ichi\ Miyachi},
       DOI = {10.1016/j.aim.2005.11.008},
       URL = {https://doi.org/10.1016/j.aim.2005.11.008},
}

@book {Kra21,
    AUTHOR = {Krause, Henning},
     TITLE = {Homological theory of representations},
    SERIES = {Cambridge Studies in Advanced Mathematics},
    VOLUME = {195},
 PUBLISHER = {Cambridge University Press, Cambridge},
      YEAR = {2022},
     PAGES = {xxxiv+482},
      ISBN = {978-1-108-83889-4},
   MRCLASS = {16G10 (16E35 16E45 16G20 18E10 18G80 20C15)},
  MRNUMBER = {4327095},
MRREVIEWER = {E.\ R.\ Alvares},
}

@book {KS94,
    AUTHOR = {Kashiwara, Masaki and Schapira, Pierre},
     TITLE = {Sheaves on manifolds},
    SERIES = {Grundlehren der mathematischen Wissenschaften [Fundamental
              Principles of Mathematical Sciences]},
    VOLUME = {292},
      NOTE = {With a chapter in French by Christian Houzel,
              Corrected reprint of the 1990 original},
 PUBLISHER = {Springer-Verlag, Berlin},
      YEAR = {1994},
     PAGES = {x+512},
      ISBN = {3-540-51861-4},
   MRCLASS = {58G07 (18F20 32C38 35A27)},
  MRNUMBER = {1299726},
}

@article {Liu10,
    AUTHOR = {Liu, Shiping},
     TITLE = {Auslander-{R}eiten theory in a {K}rull-{S}chmidt category},
   JOURNAL = {S\~ao Paulo J. Math. Sci.},
  FJOURNAL = {S\~ao Paulo Journal of Mathematical Sciences},
    VOLUME = {4},
      YEAR = {2010},
    NUMBER = {3},
     PAGES = {425--472},
      ISSN = {1982-6907},
   MRCLASS = {16G70 (16G20 18E05 18E30)},
  MRNUMBER = {2856194},
MRREVIEWER = {Alex\ S.\ Dugas},
       DOI = {10.11606/issn.2316-9028.v4i3p425-472},
       URL = {https://doi.org/10.11606/issn.2316-9028.v4i3p425-472},
}

@article{MM23,
title = {n-Term silting complexes in {$\mathrm{K}^b(\mathrm{proj}(\Lambda))$}},
JOURNAL = {J. Algebra},
FJOURNAL = {Journal of Algebra},
volume = {622},
pages = {98-133},
year = {2023},
issn = {0021-8693},
doi = {https://doi.org/10.1016/j.jalgebra.2023.01.017},
url = {https://www.sciencedirect.com/science/article/pii/S0021869323000327},
author = {Mart\'{\i}nez, Luis and Mendoza, Octavio},
MRCLASS = {16G10 (18G05 18G80)},
MRNUMBER = {4547875},
MRREVIEWER = {E. R. Alvares}
}

@misc{McM21,
      title={Support $\tau_2$-tilting and 2-torsion pairs}, 
      author={Jordan McMahon},
      year={2021},
      eprint={2102.08254},
      archivePrefix={arXiv},
      primaryClass={math.RT},
      url={https://arxiv.org/abs/2102.08254}, 
}

@misc{milicic2014lectures,
  title={Lectures on derived categories},
  author={Milicic, Dragan},
  howpublished={Preprint, available at \url{http://www.math.utah.edu/~milicic/Eprints/dercat.pdf}},
  note={[Accessed 27 nov 2025]},
  year={2014}
}

@article {NP19,
    AUTHOR = {Nakaoka, Hiroyuki and Palu, Yann},
     TITLE = {Extriangulated categories, {H}ovey twin cotorsion pairs and
              model structures},
   JOURNAL = {Cah. Topol. G\'eom. Diff\'er. Cat\'eg.},
  FJOURNAL = {Cahiers de Topologie et G\'eom\'etrie Diff\'erentielle
              Cat\'egoriques},
    VOLUME = {60},
      YEAR = {2019},
    NUMBER = {2},
     PAGES = {117--193},
      ISSN = {1245-530X,2681-2363},
   MRCLASS = {18E10 (18E30 18E35 18G15 18G55)},
  MRNUMBER = {3931945},
MRREVIEWER = {Panyue\ Zhou},
}

@misc{qpa,
	shorthand={QPA22},
	title={QPA - Quivers, path algebras and representation - a GAP package},
	author={The QPA-team},
	url={https://folk.ntnu.no/oyvinso/QPA/},
    year = {2022}
	
}

@article {Rie80,
    AUTHOR = {Riedtmann, C.},
     TITLE = {Algebren, {D}arstellungsk\"ocher, \"Uberlagerungen und
              zur\"uck},
   JOURNAL = {Comment. Math. Helv.},
  FJOURNAL = {Commentarii Mathematici Helvetici},
    VOLUME = {55},
      YEAR = {1980},
    NUMBER = {2},
     PAGES = {199--224},
      ISSN = {0010-2571,1420-8946},
   MRCLASS = {16A64 (16A36 16A46)},
  MRNUMBER = {576602},
MRREVIEWER = {V.\ Dlab},
       DOI = {10.1007/BF02566682},
       URL = {https://doi.org/10.1007/BF02566682},
}

@book{Rin84,
  author     = {Ringel, Claus Michael},
  title      = {Tame algebras and integral quadratic forms},
  series     = {Lecture Notes in Mathematics},
  volume     = {1099},
  publisher  = {Springer-Verlag, Berlin},
  year       = {1984},
  pages      = {xiii+376},
  isbn       = {3-540-13905-2},
  mrclass    = {16A64 (11E81 16A46)},
  mrnumber   = {774589},
  mrreviewer = {Yu. A. Drozd},
  doi        = {10.1007/BFb0072870},
  url        = {https://doi.org/10.1007/BFb0072870}
}

@misc{RV25,
      title={$\tau_d$-tilting theory for linear Nakayama algebras}, 
      author={Endre S. Rundsveen and Laertis Vaso},
      year={2025},
      eprint={2410.19505},
      archivePrefix={arXiv},
      primaryClass={math.RT},
      url={https://arxiv.org/abs/2410.19505}, 
      note = {arXiv preprint},
}

@article {Vos01,
    AUTHOR = {Vossieck, Dieter},
     TITLE = {The algebras with discrete derived category},
   JOURNAL = {J. Algebra},
  FJOURNAL = {Journal of Algebra},
    VOLUME = {243},
      YEAR = {2001},
    NUMBER = {1},
     PAGES = {168--176},
      ISSN = {0021-8693,1090-266X},
   MRCLASS = {16G10 (16E10 18E30)},
  MRNUMBER = {1851659},
MRREVIEWER = {Bin\ Zhu},
       DOI = {10.1006/jabr.2001.8783},
       URL = {https://doi.org/10.1006/jabr.2001.8783},
}

@book {Wei94,
    AUTHOR = {Weibel, Charles A.},
     TITLE = {An introduction to homological algebra},
    SERIES = {Cambridge Studies in Advanced Mathematics},
    VOLUME = {38},
 PUBLISHER = {Cambridge University Press, Cambridge},
      YEAR = {1994},
     PAGES = {xiv+450},
      ISBN = {0-521-43500-5; 0-521-55987-1},
   MRCLASS = {18-01 (16-01 17-01 20-01 55Uxx)},
  MRNUMBER = {1269324},
MRREVIEWER = {Kenneth\ A.\ Brown},
       DOI = {10.1017/CBO9781139644136},
       URL = {https://doi.org/10.1017/CBO9781139644136},
}

@misc{Zho25,
      title={Tilting theory for extended module categories}, 
      author={Yu Zhou},
      year={2025},
      eprint={2411.15473},
      archivePrefix={arXiv},
      primaryClass={math.RT},
      url={https://arxiv.org/abs/2411.15473}, 
}

@article{ZZ23,
title = {Support $\tau_n$-tilting pairs},
JOURNAL = {J. Algebra},
FJOURNAL = {Journal of Algebra},
volume = {616},
pages = {193-211},
year = {2023},
issn = {0021-8693},
doi = {https://doi.org/10.1016/j.jalgebra.2022.10.035},
url = {https://www.sciencedirect.com/science/article/pii/S0021869322005270},
author = {Panyue Zhou and Bin Zhu},
MRCLASS = {18G80 (16G10 18E10)},
MRNUMBER = {4512512},
MRREVIEWER = {Karin M. Jacobsen}
}
%\bibliographystyle{amsalpha}
%\bibliography{references.bib}

\end{document}